\documentclass[12pt,letter,reqno]{amsart}

\usepackage[margin=1in]{geometry}

\usepackage{amssymb,amsfonts,amsthm,mathrsfs}
\usepackage{marginnote}
\usepackage{comment}
\usepackage{enumitem}
\usepackage{graphicx}
\usepackage{bm}
\usepackage{xfrac}

\usepackage[dvipsnames]{color}

\usepackage[colorlinks=true, pdfstartview=FitV, linkcolor=black, citecolor=blue, urlcolor=blue]{hyperref}

\def\Xint#1{\mathchoice
{\XXint\displaystyle\textstyle{#1}}%
{\XXint\textstyle\scriptstyle{#1}}%
{\XXint\scriptstyle\scriptscriptstyle{#1}}%
{\XXint\scriptscriptstyle%
\scriptscriptstyle{#1}}%
\!\int}
\def\XXint#1#2#3{{\setbox0=\hbox{$#1{#2#3}{%
\int}$ }
\vcenter{\hbox{$#2#3$ }}\kern-.6\wd0}}
\def\barint{\, \Xint -} % \, corrects the \! used in the definition
\def\bariint{\barint_{} \kern-.4em \barint}
\def\bariiint{\bariint_{} \kern-.4em \barint}
\renewcommand{\iint}{\int_{}\kern-.34em \int} %\ minor space between the integrals
\renewcommand{\iiint}{\iint_{}\kern-.34em \int} %\ minor space between the integrals

\DeclareMathAlphabet{\mathcal}{OMS}{cmsy}{m}{n}

\theoremstyle{plain}

\newtheorem{theorem}{Theorem}[section]

\newtheorem{definition}[theorem]{Definition}
\newtheorem{lemma}[theorem]{Lemma}

\newtheorem{proposition}[theorem]{Proposition}

\theoremstyle{definition}
\newtheorem{remark}[theorem]{Remark}

\newcommand{\R}{\mathbb{R}}

\newcommand{\N}{\mathbb{N}}
\newcommand{\T}{\mathbb{T}}

\newcommand{\bI}{\mathbf{I}}

\newcommand{\p}{\partial}

\newcommand{\norm}[1]{\lVert #1 \rVert}

\newcommand{\into}{\hookrightarrow}

\newcommand{\E}{\mathcal{E}}

\newcommand{\bu}{\bm{u}}

\newcommand{\X}{\bm{X}}
\newcommand{\e}{\bm{e}}

\newcommand{\Y}{\bm{Y}}
\newcommand{\Z}{\bm{Z}}
\renewcommand{\b}{\bm{b}}

\newcommand{\abs}[1]{\left\lvert #1 \right\rvert}
\newcommand{\wh}[1]{\widehat{#1}}
\newcommand{\wt}[1]{\widetilde{#1}}
\newcommand{\mc}[1]{\mathcal{#1}}
\renewcommand{\bar}[1]{\overline{#1}}

\allowdisplaybreaks

\title{Rods in flows: the PDE theory of immersed elastic filaments} 

\author[Albritton]{Dallas Albritton} 
\address[Dallas Albritton]{University of Wisconsin-Madison, Department of Mathematics, 480 Lincoln Dr, Madison, WI 53706, USA}
\email{dalbritton@wisc.edu}

\author[Ohm]{Laurel Ohm} 
\address[Laurel Ohm]{University of Wisconsin-Madison, Department of Mathematics, 480 Lincoln Dr, Madison, WI 53706, USA}
\email{lohm2@wisc.edu}

\numberwithin{equation}{section}
\setlist[enumerate]{leftmargin=*}

\begin{document}

\begin{abstract}
We investigate a family of curve evolution equations approximating the motion of a Kirchhoff rod immersed in a low Reynolds number fluid. The rod is modeled as a framed curve whose energy consists of the bending energy of the curve and the twisting energy of the frame. The equations we consider may be realized as gradient flows of the rod energy under a certain anisotropic metric coming from resistive force theory. Ultimately, our goal is to provide a comprehensive treatment of the PDE theory of immersed rod dynamics. We begin by analyzing the problem without the frame, in which case the evolution is globally well-posed and solutions asymptotically converge to Euler elasticae. Next, in the planar setting, we demonstrate the existence of large time-periodic solutions forced by intrinsic curvature relevant to undulatory swimming. Finally, the majority of the paper is devoted to the evolution of an immersed Kirchhoff rod in 3D, which involves a strong coupling between the curve and frame. %,
%involving a tightly coupled curve and frame evolution.
Under a physically reasonable assumption on the curvature, solutions exist globally-in-time and asymptotically converge to rod equilibria.
\end{abstract}

\maketitle

%%%%%%%%
\setcounter{tocdepth}{1}
\tableofcontents
%%%%%%%%

\parskip   2pt plus 0.5pt minus 0.5pt

%!TEX root = RFT GWP.tex

\section{Introduction}
We investigate a family of curve evolution equations approximating the motion of a Kirchhoff rod immersed in a low Reynolds number fluid. These models arise in the study of undulatory swimming \cite{gray1955propulsion,cox1970motion, wiggins1998flexive}, DNA supercoiling \cite{fuller1971writhing,coleman2000theory,shi1994kirchhoff}, polymer dynamics in fluids \cite{goldstein1995nonlinear}, and pattern selection in the growth of biofilaments \cite{wolgemuth2004dynamic,moulton2013morphoelastic}. While rod equilibria are well studied, the above phenomena are inherently dynamical in nature, and it is important to develop a theory for the physically relevant equations of motion. 
One can consider filament evolution equations which incorporate fluid effects at various levels of complexity. In this paper, we give a thorough treatment of a selection of models derived from resistive force theory.

\subsubsection*{Curve evolutions}

We begin with classical elastohydrodynamics. Let $I=[0,1]$ or $\T : =\R/\mathbb{Z}$. The filament position $\X:I\times [0,T]\to \R^3$ evolves according to
\begin{equation}\label{eq:classical}
\begin{aligned}
\frac{\p\X}{\p t}(s,t) &= -\big({\bf I}+\gamma\X_s\otimes\X_s\big)\big(\X_{sss}-\lambda(s,t)\X_s\big)_s \\
\abs{\X_s}^2&=1 \\
\X_{ss}\big|_{s=0,1}&=0\,, \quad (\X_{sss}-\lambda\X_s)\big|_{s=0,1}=0 \qquad \text{if }I=[0,1]\,,
\end{aligned}
\end{equation}
or $\X$ is periodic in $s$ if $I=\T$. Here the matrix $\big({\bf I}+\gamma\X_s\otimes\X_s\big)$ is the resistive force theory (local slender body theory) approximation of the hydrodynamic effects on the filament evolution \cite{gray1955propulsion, johnson1979flagellar, keller1976swimming, pironneau1974optimal}, where the shape factor $0\le \gamma\le1$ incorporates the anisotropy of the viscous drag on a thin filament. For an infinitely thin filament, $\gamma\to 1$, weighting the drag along the tangential direction of the filament twice as much as the normal directions.
The terms $-(\X_{sss}-\lambda(s,t)\X_s)_s$ model the elastic response of the filament according to Euler-Bernoulli beam theory~\cite{camalet2000generic, camalet1999self, hines1978bend, tornberg2004simulating, wiggins1998flexive, wiggins1998trapping}. The filament tension $\lambda(s,t)$ serves as a Lagrange multiplier to enforce the local inextensibility constraint $\abs{\X_s}^2=1$.
The boundary conditions come from the requirement that the filament is force- and torque-free at its endpoints. It is due to the surrounding low Reynolds number fluid that the evolution is parabolic, whereas the evolution of a non-immersed beam would be hyperbolic. %, the Euler-Lagrange equations for \eqref{eq:Eintrinsic}. %yield  %of the %Kirchhoff rod energy \eqref{eq:Eintrinsic}.
 % for the rod displacement.}

The system~\eqref{eq:classical} with $\gamma=0$ is known as the \emph{curve-straightening flow} for a (locally) inextensible filament, i.e., the $L^2$ gradient flow for the filament bending energy 
\begin{equation}\label{eq:bending_energy}
    \frac{1}{2}\int_{I}\kappa^2\,ds = \frac{1}{2}\int_{I}\abs{\X_{ss}}^2\,ds
\end{equation}
with the constraint $\abs{\X_s}^2=1$. The critical points of~\eqref{eq:bending_energy} are called the \emph{Euler elasticae}~\cite{EulerOriginal} and themselves have a rich history recounted in, for example, \cite{levien2008elastica,matsutani2010euler}. %They correspond to steady states of~\eqref{eq:classical} for any ari
The elasticae are the steady states of any gradient flow of~\eqref{eq:bending_energy}, and various flows %of~\eqref{eq:bending_energy}
have been proposed~\cite{langer1984total, langer1984knotted, langer1985curve, wen1993l2, wen1995curve, mantegazza2021survey} to understand how the elasticae are attained dynamically. However, the above papers replace the local inextensibility constraint %$|\X_s|^2=1$ 
with the non-local constraint that the~\emph{total} length of the filament is fixed. %, in both the $L^2$ metrics on $\X$ and the tangent angle.
Such flows have the same equilibria, without capturing the physical dynamics. On the other hand, the system~\eqref{eq:classical} in the planar setting $\X(s,t) \in \R^2$ is simpler to formulate. In that case, with the inextensibility constraint, the curve-straightening flow has been shown to be globally well-posed in~\cite{koiso1996motion,oelz2011curve} for closed loops and free ends, respectively.

%In addition, there is a vast literature analyzing the classical problem in which the \emph{total} length of the filament is fixed, but the arclength parameterization is not necessarily preserved by the flow

%The locally inextensible curve-straightening flow has been shown to be globally well-posed in \cite{koiso1996motion} in the closed setting and \cite{oelz2011curve} in the free end setting. 
In this paper, we interpret~\eqref{eq:classical} as the gradient flow of the bending energy \eqref{eq:bending_energy} with respect to the anisotropic metric 
\begin{equation}\label{eq:metric1}
 \langle \bm{W},\bm{Y}\rangle_\gamma = \int_{I} \left({\bf I} + \gamma \X_s\otimes\X_s\right)^{-1} \bm{W}\cdot\bm{Y}\,ds
 \end{equation}
 defined on perturbations $\bm{W},\bm{Y}$ of an inextensible curve $\X$ in $\R^3$. 
 We subsequently exploit the connection to geometric flows to prove global well-posedness and convergence to equilibrium for arbitrary solutions to~\eqref{eq:classical} in both the closed loop and free end settings:

%We may exploit the close connection with well-studied geometric flows to obtain global well-posedness and convergence to steady states % with $\gamma > 0$ also arises 

\begin{theorem}%[Global well-posedness for \eqref{eq:classical}]
\label{thm:GWP_classical}
The classical elastohydrodynamics model~\eqref{eq:classical} on $I = \T$ and $[0,1]$ is globally well-posed in $H^2_s$, and its solutions converge strongly to Euler elasticae as $t \to +\infty$.
\end{theorem}

A crucial tool we borrow from geometric flows is the \emph{Lojasiewicz inequality}, which we prove for both curves and rods in Section~\ref{sec:rod_loja}.

%\lo{The global regularity problem is ``critical" in that the full power of the smoothing is exploited in the estimates. [This statement seems like something we want to highlight but it feels very disconnected right here...]}

An important application of~\eqref{eq:classical} is to the investigation of undulatory swimming at low Reynolds number \cite{el2020optimal, gadelha2010nonlinear, gadelha2019flagellar, hu2022enhanced, montenegro2015spermatozoa, lauga2013shape, lauga2007floppy,  spagnolie2010optimal, lauga2009hydrodynamics, mori2023well, ohm2024well}. In such applications, we also want to prescribe a %some kind of 
time-dependent forcing along the filament to drive its motion. If we restrict the filament to planar deformations only, we may consider the free-end version of \eqref{eq:classical} with the addition of forcing by \emph{preferred curvature} $\zeta(s,t)$:
\begin{equation}\label{eq:2D_forced}
\frac{\p\X}{\p t}(s,t) = -\big({\bf I}+\gamma\X_s\otimes\X_s\big)\big(\X_{sss}-\lambda(s,t)\X_s-(\zeta\e_{\rm n})_s\big)_s\,, \quad \abs{\X_s}^2=1\,.
\end{equation}
Here $\e_{\rm n}(s,t) = \X_s^\perp$ is the unique (up to sign) in-plane unit normal vector to the filament $\X(s,t)$. %Note that,
Integrating \eqref{eq:2D_forced} in arc length, we %may
obtain a formal expression for the center-of-mass velocity of the filament: 
\begin{equation}\label{eq:swimming}
\bm{V}(t) := \int_0^1\frac{\p\X}{\p t}(s,t)\,ds = -\gamma\int_0^1\X_s\big(\X_s\cdot(\X_{ss}-\zeta\e_{\rm n})_{ss}-\lambda_s \big)\,ds\,.
\end{equation}
In particular, the drag anisotropy $\gamma> 0$ is necessary for the filament to actually swim in a viscous environment. A time-periodic preferred curvature is typically used to model the actuation of a flexible flagellum. The resulting time-periodic solution models the stroke.

%We may study the evolution \eqref{eq:2D_forced} under different choices of time-dependent $\zeta$.

In~\cite{mori2023well}, the second author and collaborator demonstrated local well-posedness for~\eqref{eq:2D_forced} %is shown for this evolution, along with
and the existence of a unique time-periodic solution given a sufficiently small time-periodic preferred curvature. They subsequently analyzed and performed a numerical optimization on the swimming speed~\eqref{eq:swimming}. The gradient flow structure, however, was not exploited. Here we use this structure to demonstrate the existence of time-periodic solutions given arbitrarily large time-periodic forcing %$\zeta(s,t)$. 
\begin{equation}
    \label{eq:zetacondition}
\zeta \in C(\T_T;H^1_s) \, , \quad \p_t \zeta \in C(\T_T;L^2_s) \, ,
\end{equation}
where $\T_T = \R/T\mathbb{Z}$ is the $T$-periodic interval.
\begin{theorem}%[Large periodic solutions in forced planar setting]
\label{thm:periodic}
Given $\zeta$ satisfying~\eqref{eq:zetacondition}, there exists a solution to~\eqref{eq:2D_forced} in the free end setting whose curvature $\kappa = \X_{ss} \cdot \e_{\rm n}$ is $T$-periodic in time.
\end{theorem}

We emphasize that these solutions are outside of the standard perturbation theory. We exhibit them via a topological fixed point theorem (the Leray-Schauder theorem).

\subsubsection*{Rod evolutions}

%From the continuum mechanics perpsective

%From an elasticity perspective

%On the one hand, the curve evolution... is a perfectly self-consistent model for the dynamics of an inextensible curve. However, from a continuum mechanics perspective, a physical rod has additional degrees of freedom which track twist deformations of the rod.

%In Kirchoff rod theory, the filament is modeled as a curve representing the rod's centerline and a frame capable of tracking its twist deformations 
%Since the curve in \eqref{eq:classical} represents a physical object, it is meaningful to track not just the position of the curve as it evolves in space, but also its twist.
%For this, we turn to Kirchhoff rod theory

An important physical characteristic of a rod is its ability to resist twist about its centerline. This cannot be captured by the curve evolution~\eqref{eq:classical} alone. We thus adopt Kirchoff rod theory~\cite{antman1974kirchhoff, dill1992kirchhoff, langer1996lagrangian,oreilly2017modeling} and introduce an orthonormal frame $(\e_{\rm t},\e_1,\e_2)$ along the curve. Here $\e_{\rm t}=\X_s$ is the unit tangent vector to the curve and the directors $\e_1,\e_2$ span the directions normal to $\X$ at each point. At each time $t$, the frame satisfies the following ODE in arclength $s$: 
\begin{equation}\label{eq:frame_ODE}
\begin{aligned}
\p_s\begin{pmatrix}
\e_{\rm t} \\
\e_1 \\
\e_2
\end{pmatrix}
=
\begin{pmatrix}
0 & \kappa_1 & \kappa_2 \\
-\kappa_1 & 0 & \kappa_3 \\
-\kappa_2 & -\kappa_3 & 0
\end{pmatrix}\begin{pmatrix}
\e_{\rm t}\\
\e_1\\
\e_2
\end{pmatrix}
\end{aligned}
\end{equation}
with $\kappa_1^2+\kappa_2^2=\kappa^2(s,t)$, the squared curvature at $s$. The coefficient $\kappa_3$ corresponds to the twist in the frame along $\X(s,t)$. The frame $(\e_{\rm t},\e_1,\e_2)$ is material and evolves along with the curve.

 %Most generally, may
We consider the energy 
\begin{equation}\label{eq:Eintrinsic}
E[\X,\e_1,\e_2] = \frac{1}{2} \int_0^1 \bigg(\abs{\kappa_1-\zeta_1}^2 + \abs{\kappa_2-\zeta_2}^2 + \eta \abs{\kappa_3-\zeta_3}^2\bigg)\,ds 
\end{equation}
over the space of curves $\X$ satisfying the inextensibility constraint $\abs{\X_s}^2=1$ and orthonormal frames $(\e_t,\e_1,\e_2)$ satisfying \eqref{eq:frame_ODE}.
Here $\zeta_1(s,t)$ and $\zeta_2(s,t)$ are components of an intrinsic filament curvature and $\zeta_3(s,t)$ is an intrinsic twist \cite{olson2013modeling, wolgemuth2004dynamic, moulton2013morphoelastic,haijun1999spontaneous}.\footnote{
%If the filament admits a Frenet frame for all time, then we may instead consider intrinsic forcing along the principal normal direction only: $\E[\X] = \frac{1}{2}\int_0^1(\kappa-\zeta)^2\,ds$.
%With this form of forcing, we do not have to keep track of a frame evolution, as the curve automatically determines the frame. The obvious disadvantage is there is no guarantee that the filament can maintain a Frenet frame for all time, as it is unclear that inflection points can be avoided. 
Other forms of forcing are possible, such as those considered in \cite{lough2023self, gazzola2018forward}, but these still require tracking a material frame alongside the curve.} The material parameter $\eta$ corresponds to the ratio between the rod's resistance to twist and its bending stiffness and, for a thin circular rod, satisfies $\eta=\frac{1}{1+\nu}$, where $0\le\nu\le \frac{1}{2}$ is Poisson's ratio for the material \cite[p. 195]{coleman2000theory}, \cite{mcmillen2002tendril,goriely2006twisted}.\footnote{Some sources \cite[Chap. 5.14]{oreilly2017modeling}, \cite[p. 662]{bartels2020numerical} use $\eta=\frac{1}{2(1+\nu)}$, in which case $\frac{1}{3}\le\eta\le\frac{1}{2}$.} For isotropic, incompressible elastic materials, $\nu\approx\frac{1}{2}$; in particular, for typical flagellated microswimmers, the twist-to-bend resistance ratio $\eta$ satisfies $\eta<1$.

The coupled rod-frame evolution will be a gradient flow of the energy \eqref{eq:Eintrinsic} with respect to variations in both the curve \emph{and} the frame %, $(\X,\e_1,\e_2)+\varepsilon(\dot{\X},\dot{\e}_1,\dot{\e}_2)$,
under the metric
\begin{equation}\label{eq:metric2}
 \langle\bm{W},\bm{Y} \rangle_\gamma + \langle a,b \rangle_\alpha = \langle\bm{W},\bm{Y} \rangle_\gamma + 
 \alpha \int_0^1a(s)b(s) \,ds\,,
\end{equation} 
where $\langle\bm{W},\bm{Y} \rangle_\gamma$ is as in~\eqref{eq:metric1} and the scalar-valued $a(s)$ parameterizes the in-plane perturbation of the frame in each cross section.
%is the in-plane twist at each cross section of the filament. %: %$\dot\e_1\cdot\e_2 = a(s)$, $\dot\e_2\cdot\e_1 = -a(s)$.
Here $0<\alpha\ll1$ is a rotational friction coefficient~\eqref{eq:alphadef} due to the viscous fluid effects. % -- see. %Physically, $\alpha$ scales with square of the filament radius.
The derivation and physical reasoning for this metric will be discussed in Section~\ref{subsec:derivation}, where the full system with non-zero $\zeta_1,\zeta_2,\zeta_3$ is displayed. %in Section~\ref{subsec:derivation}.
From now on, we will consider $\zeta_1=\zeta_2=\zeta_3=0$, in which case the system can be written in terms of the curve $\X$ and the frame twist $\kappa_3$ as 
\begin{align}
    \p_t\X 
    &=  -({\bf I}+\gamma\X_s\otimes\X_s)\big(\X_{sss}-\lambda\X_s -\eta\kappa_3\X_s\times\X_{ss} \big)_s \,, \quad \abs{\X_s}^2=1\,,
    \label{eq:Xss_0}\\
%
    % \p_t\kappa_3   
    % &= \bigg(\frac{\eta}{\alpha}+\eta\abs{\X_{ss}}^2\bigg)(\kappa_3)_{ss}+ \mc{R}^{\rm f}[\X,\kappa_3]\,, \label{eq:kap3_0}
    \p_t \kappa_3 &= \p_t \X_s \cdot (\X_s \times \X_{ss}) + \frac{\eta}{\alpha} (\kappa_3)_{ss}\,, \label{eq:kap3_0}
\end{align}
where 
% \begin{align}
%     \mc{R}^{\rm f}[\X,\kappa_3] &= \eta\big(\abs{\X_{ss}}^2\big)_s(\kappa_3)_s + \eta\kappa_3\X_{ssss}\cdot\X_{ss} + \eta\kappa_3\abs{\X_{ss}}^4  \label{eq:Rf0} \\
%     &\qquad 
%       -\X_{sssss} \cdot (\X_s\times\X_{ss}) +\lambda\X_{sss}\cdot (\X_s\times\X_{ss}) \,. \nonumber
% \end{align}
% Here 
the tension $\lambda$ satisfies
\begin{equation}\label{eq:tension_0}
    (1+\gamma)\lambda_{ss}  -\abs{\X_{ss}}^2\lambda = 
    -(4+3\gamma)(\X_{sss}\cdot\X_{ss})_s +\abs{\X_{sss}}^2  - \eta\kappa_3\X_s\cdot(\X_{ss}\times\X_{sss})\,.
\end{equation}

A characteristic of the rod system is the strong coupling between the curve and the frame: The right-hand side of the $\kappa_3$ equation~\eqref{eq:kap3_0} contains two terms. The term $\p_t \X_s \cdot (\X_s \times \X_{ss})$ is the twist induced by the motion of the curve. The term $\frac{\eta}{\alpha} (\kappa_3)_{ss}$ captures the diffusion of the twist and arises entirely from the independent degree of freedom of the frame. If we substitute~\eqref{eq:Xss_0} into~\eqref{eq:kap3_0}, we obtain
\begin{equation}
    \label{eq:kappa3strongcoupling}
    \p_t \kappa_3 =  \bigg(\frac{\eta}{\alpha}+ (\eta-1) |\X_{ss}|^2 \bigg)(\kappa_3)_{ss} + \cdots \, .
    %-\kappa_1(\kappa_2)_{sss}+\kappa_2(\kappa_1)_{sss} +\mc{R}_3[\kappa_j] \,,
\end{equation}
% \p_t \kappa_3 = \eta \left( \frac{1}{\alpha} - |\X_{ss}|^2 \right) (\kappa_3)_{ss} - \X_{sssss} \cdot (\X_s\times\X_{ss}) + \cdots\,.
%The curve-frame formulation \eqref{eq:Xss_0}-\eqref{eq:kap3_0} instead brings to
%When $|\X_{ss}|^2 > \frac{1}{\alpha}$, 
Since $\eta < 1$, the apparent backward heat in~\eqref{eq:kappa3strongcoupling} highlights the %very
strong coupling, %between the $\X$ and $\kappa_3$ evolutions,
and one wonders whether the curve is attempting to decrease its energy by converting bend into twist through some mechanism at small scales. While the twist-bend coupling is known to give rise to large-scale instabilities of rod equilibria~\cite{goriely1997nonlinear,goriely2006twisted,bartels2020numerical}, 
%zajac1962stability,michell1890small,bartels2020numerical,van2003instability}, 
% phenomena have typically been investigated through the instabilities of rod equilibria. %, and the equation ceases to be well-posed.  experimentally 
to our knowledge, no such behavior at small scales has been reported in such a way which would suggest ill-posedness of the equation. In Section~\ref{subsec:uncond_GWP}, we present evidence in favor of unconditional well-posedness. Nevertheless, for the purposes of this paper, we will require a condition on the maximum filament curvature to control the feedback from the coupling. 
In particular, we require
\begin{equation}\label{eq:Xss_criterion}
    \norm{\X_{ss}}_{L^\infty_s}^2 < \frac{1}{\alpha}
\end{equation}
%The condition  na{\"i}ve condition
under which the energy dissipation (where $\gamma = 0$ for simplicity)
\begin{equation}
    \label{eq:energydissipationintro}
\frac{dE}{dt} = - \int_I \big|\big(\X_{sss}-\lambda\X_s -\eta\kappa_3\X_s\times\X_{ss} \big)_s\big|^2 - \frac{\eta^2}{\alpha} \int_I (\kappa_3)_s^2 \, ds
\end{equation}
is coercive, that is, controls $\X_{ssss}$ and $(\kappa_3)_s$.

 %A rod may decrease its energy by converting some bend into twist, especially when $\eta \ll 1$.

%, %given by the terms $\eta(\kappa_3\X_s\times\X_{ss})_s$ in \eqref{eq:Xss_0} and $-\X_{sssss} \cdot (\X_s\times\X_{ss})$ in \eqref{eq:kap3_0}.
%It is unclear from energy methods alone whether this strong coupling between the curve and frame is helpful or harmful.

%to hold for all time, where $0<\alpha\ll1$ is the rotational friction coefficient in \eqref{eq:metric2} due to the viscosity of the surrounding fluid. 

Using the curvature criterion \eqref{eq:Xss_criterion}, we obtain conditional well-posedness of the curve-frame evolution in both the free end and closed loop settings. Furthermore, the free-ended filament converges to a straight rod, and the closed loop converges to a rod equilibrium.

%The apparent backward heat operator in \eqref{eq:kappa3_bkwd_heat} is now hidden in a term $-\X_{sssss} \cdot (\X_s\times\X_{ss})$ appearing on the right hand side of \eqref{eq:kap3_0}.

% We wish to highlight a particular curiosity in the equation for $\kappa_3$: even in the absence of internal forcing ($\zeta_1=\zeta_2=\zeta_3=0$), where the $\kappa_3$ evolution is given by
% \begin{equation}\label{eq:kappa3_bkwd_heat}
%     \p_t \kappa_3 =  \bigg(\frac{\eta}{\alpha}+ (\eta-1)( \kappa_1^2+\kappa_2^2) \bigg)(\kappa_3)_{ss} -\kappa_1(\kappa_2)_{sss}+\kappa_2(\kappa_1)_{sss} +\mc{R}_3[\kappa_j] \,,
% \end{equation}
% the twist evolution appears to contain a backward heat operator. 
% Recalling that $0\le\eta<1$, we see that if the initial filament configuration has regions of very large curvature $\kappa^2=\kappa_1^2+\kappa_2^2$, the quantity $\frac{\eta}{\alpha}+ (\eta-1)( \kappa_1^2+\kappa_2^2)$ in \eqref{eq:kappa3_bkwd_heat} can apparently be negative, even for very small rotational friction coefficient $\alpha$. 

\begin{theorem}%[Global well-posedness for curve + frame]
\label{thm:GWP_curveframe}
For $(\X^{\rm in},\kappa_3^{\rm in}) \in H^4_s \times H^1_s$ satisfying~\eqref{eq:Xss_criterion}, there exists a strong solution to the curve-frame evolution \eqref{eq:Xss_0}-\eqref{eq:kap3_0} which can be continued while~\eqref{eq:Xss_criterion} holds. If
\begin{equation}
    \label{eq:asymptoticalpha}
\limsup_{t \to +\infty} \| \X_{ss} \|_{L^\infty_s}^2 < \frac{1}{\alpha} \, ,
\end{equation}
then the solution converges to a rod equilibrium as $t \to +\infty$. Finally, there exist $C_1, C_2 > 0$ such that~\eqref{eq:asymptoticalpha} is guaranteed to hold provided
\begin{equation}
\| \X_{ss}^{\rm in} \|_{H^2_s} + \| \kappa_3^{\rm in} \|_{H^1_s} \leq \frac{1}{C_1} \log \frac{1}{C_2 \alpha} \, .
\end{equation}
%then the solution is global-in-time.
\end{theorem}
%Given that the
%we may also

Since the rotational friction coefficient $\alpha$ is very small, it makes sense to consider the asymptotic behavior of the system \eqref{eq:Xss_0}-\eqref{eq:kap3_0} as $\alpha\to0$. As can be seen from~\eqref{eq:energydissipationintro}, $\kappa_3$ equilibrates on the fast timescale $\alpha$, and in the limit its evolution becomes quasi-steady, with $\kappa_3=\overline\kappa_3(t)$ a function of time only. In the free end setting, $\overline{\kappa}_3 \equiv 0$, i.e., the material frame is always a Bishop frame \cite{bishop1975there}, and the curve evolves via~\eqref{eq:classical}. 
%At each $t$, we then have $(\kappa_3)_{ss}=0$, in particular, $\kappa_3=\overline\kappa_3(t)$, a function of time only.
%In the free end setting, due to the homogeneous boundary conditions for $\kappa_3$, we have $\overline\kappa_3=0$ for all time, i.e. the material frame is always a Bishop frame \cite{bishop1975there}. Given an initial frame with zero twist rate $\overline\kappa_3^{\rm in}=0$, the frame at all times may be uniquely determined by requiring that $\overline\kappa_3$ remains zero.
%
In the more interesting closed loop setting, the curve and constant twist $\bar{\kappa}_3(t)$ satisfy
% The closed filament is more interesting, as $\overline\kappa_3(t)$ does not have to be zero. %and may change over time.
% The closed curve and frame satisfy the system 
\begin{align}
    \p_t\X + \X_{ssss}
    &=  -({\bf I}+\gamma\X_s\otimes\X_s)\big(\X_{sss}-\lambda\X_s -\eta\overline\kappa_3\X_s\times\X_{ss} \big)_s \,, \quad \abs{\X_s}^2=1\,,  \label{eq:Xss_alpha00} \\
    \p_t\overline\kappa_3 &= \eta\overline\kappa_3\int_{\T}\bigg(-\abs{\X_{sss}}^2 +\abs{\X_{ss}}^4\bigg)\,ds \nonumber \\
&\qquad+ \int_{\T}\bigg(\X_{ssss} \cdot (\X_s\times\X_{sss})+\lambda\X_{sss}\cdot (\X_s\times\X_{ss})\bigg)\,ds \, , \label{eq:kap3_alpha00}
\end{align}
while the tension $\lambda$ satisfies 
\begin{equation}\label{eq:lambda00}
    (1+\gamma)\lambda_{ss}  -\abs{\X_{ss}}^2\lambda = 
    -(4+3\gamma)(\X_{sss}\cdot\X_{ss})_s +\abs{\X_{sss}}^2  - \eta\overline\kappa_3\X_s\cdot(\X_{ss}\times\X_{sss})\,.
\end{equation}
We show global well-posedness for the system \eqref{eq:Xss_alpha00}-\eqref{eq:lambda00} without restriction on $\norm{\X_{ss}}_{L^\infty_s}$:

\begin{theorem}%[Global well-posedness for $\alpha\to0$]
\label{cor:alpha0}
The evolution \eqref{eq:Xss_alpha00}-\eqref{eq:kap3_alpha00} is globally well-posed in $(\X,\overline\kappa_3) \in H^2_s \times \R$.
\end{theorem}

Finally, we have reason to suspect that the condition~\eqref{eq:Xss_criterion} is not necessary for well-posedness. Specifically, in Section~\ref{subsec:uncond_GWP}, we linearize the rod-frame equations for a closed filament about the multiply-covered circle $(\kappa_1,\kappa_2,\kappa_3)=(\overline\kappa,0,b)$ for (large) constant $\overline\kappa$ and constant twist $b$. We calculate the spectrum of the linearized operator and discover that the strong coupling between the curve and frame evolutions actually serves to cancel out the apparent backward heat in~\eqref{eq:kappa3strongcoupling}. This suggests that the principal symbol of the evolution \emph{as a system} is elliptic. We intend to exploit this in future work.

% Finally, although the energy methods used in this paper do not distinguish whether the strong coupling between the curve and frame evolutions when $\alpha>0$ is helpful or harmful, we provide evidence in section \ref{subsec:uncond_GWP} that this coupling may actually be helpful. Specifically, for a closed filament, we linearize the rod+frame equations about the multiply-covered circle $(\kappa_1,\kappa_2,\kappa_3)=(\overline\kappa,0,b)$ for (large) constant $\overline\kappa$ and constant twist $b$. We can calculate the spectrum of the linearized operator and find that the strong coupling between the curve and frame evolutions actually serves to cancel out the backward heat evolution in this setting. This suggests a mechanism for showing unconditional well-posedness, but a more in-depth microlocal description of this cancellation may be needed in order to take advantage of this more generally. 

%%%%%%%%%%%%%%%%%%%%%%%%%%%%%%%%%%%%%%%%%%%%%%%%%%%%%%%%%
\subsection{Additional literature}\label{subsec:litreview}

%\emph{Steady state classification.}

\subsubsection*{Steady state classification} Critical points of the bending energy~\eqref{eq:bending_energy} are the Euler elasticae. The classification of 2D elasticae was known to Euler~\cite{EulerOriginal}. With our boundary conditions, the only solutions in 2D are the straight line, the circle, and the lemniscate (figure eight), whereas the situation in 3D is more complex. The steady state equation is $(\X_{sss} - \lambda \X_s)_s = 0$, and we reproduce the most elementary aspects of its analysis below:

In the free end setting, $\X_{sss}-\lambda\X_s = \text{const. } \bm{c} = 0$, due to the second boundary condition of \eqref{eq:classical}. Writing $\X_{sss}$ with respect to the Frenet frame, i.e., $(\kappa_1,\kappa_2,\kappa_3)=(\kappa,0,\tau)$, where $\tau$ is the curve torsion, we have
\begin{align}
    \kappa_s\e_{\rm n}+\tau\kappa\e_{\rm b}-\kappa^2\e_{\rm t} = \lambda \e_{\rm t}\,,
\end{align}
from which we obtain $\kappa_s=\tau\kappa=0$ and $\lambda=-\kappa^2$. Thus, $\tau=0$ and $\kappa={\rm const.}=0$ from the first boundary condition $\X_{ss}\big|_{s=0,1}=0$. Therefore, the free end elasticae are straight lines.

%For the curve evolution, setting the right hand side of \eqref{eq:classical} equal to 0, we see that steady states must satisfy $(\X_{sss}-\lambda\X_s)_s=0$.

For a closed filament, the steady state equation written in the Frenet frame is %writing $(\X_{sss}-\lambda\X_s)_s=0$ 
\begin{align}
    \lambda_s+\frac{3}{2}(\kappa^2)_s &=0 \label{eq:SS1} \\
    \kappa_{ss}-\tau^2\kappa-\kappa^3-\lambda\kappa &= 0 \label{eq:SS2} \\
    2\tau\kappa_s+\tau_s\kappa &=0 \label{eq:SS3} \,.
\end{align}
Integrating \eqref{eq:SS1}, we have $\lambda=-\frac{3}{2}\kappa^2+q$ for constant $q$, which we substitute into~\eqref{eq:SS2}. Multiplying \eqref{eq:SS3} by $\kappa$ and integrating, we obtain the celebrated elastica equations for $\kappa$ and $\tau$:
\begin{align}
    \kappa_{ss}+\frac{1}{2}\kappa^3-\tau^2\kappa-q\kappa =0 \, , \quad \tau\kappa^2 &= {\rm const.} \, ,  \label{eq:elastica}    
\end{align}
with planar deformations corresponding to $\tau=0$.\footnote{The elastica equations \eqref{eq:elastica} arise also as the steady state equations for classical curve-straightening flow under a global, rather than local, length constraint, albeit with a different constant $q$.}
%Integrating \eqref{eq:elastica1}, we have that $q=\big(\frac{1}{2}\int_\T\kappa^3\,ds- \int_\T\tau^2\kappa \,ds\big)/\int_\T\kappa \,ds$. 
%
A full classification of 3D solutions to \eqref{eq:elastica} is due to Langer and Singer in the mid-1980s~\cite{langer1984total,langer1984knotted, langer1985curve}. %If we restrict to planar deformations ($\tau=0$), the solutions are the circle and the figure eight \cite{langer1984total}.
%For 3D deformations, additional
The new steady states in 3D include an infinite family of non-planar curves lying on embedded tori of revolution. % \cite{langer1984knotted, langer1985curve}. 

%and a classification in 3D was given by Langer and Singer in~\cite{langer1984knotted, langer1985curve}.

%
The critical points of the full Kirchhoff rod energy~\eqref{eq:Eintrinsic} are studied in \cite{langer1996lagrangian,kawakubo2002kirchhoff,kawakubo2000stability,ivey1999knot,gerlach2017elastic}. This classification is yet more complicated, although solutions satisfy $\kappa_3 = {\rm const}$.

\subsubsection*{Geometric flows.} 
%As mentioned, the curve evolution \eqref{eq:classical} is directly related to a type of geometric flow known as curve-straightening flow, and
 %\da{Note -- Wen paper with global length constraint, $L^2$ of the tangent angle $\theta$.}
The most directly relevant references to curve-straightening appear below~\eqref{eq:bending_energy}. Additional analysis results for gradient flows of the bending energy \eqref{eq:bending_energy} appear in~\cite{lin2015second, lin2012l2, dall2017gradient}, while gradient flows for the Kirchhoff rod energy \eqref{eq:Eintrinsic} are studied in \cite{lin2004geometric,lin2009evolving}. These papers are concerned with geometric flows which have the same equilibria as the curve and rod evolutions considered above, without investigating the physical dynamics directly. Interestingly, the flow in~\cite{lin2004geometric} has a quasi-steady $\bar{\kappa}_3$ evolution like the one we derive in~\eqref{eq:kap3_alpha00}.

% More broadly, curve-straightening flows are related to curve shortening flows (i.e. motion by mean curvature), for which important results appear in \cite{altschuler1991singularities, altschuler1992shortening, gage1986heat, angenent1990parabolic, angenent1991parabolic}.
%~\eqref{eq:Eintrinsic}
We further note that in the non-immersed setting, the Euler-Lagrange equations for the beam and Kirchhoff rod energies are hyperbolic and have been studied in~\cite{caflisch1984nonlinear,burchard2003cauchy,coleman1993dynamics,maddocks1984stability}.%romero2020variational}.

%%%%%%%%%%%%%%%%%%%%%%%%%%%%%%%%%%%%%%%%%%%%%%%%%%%%%%%

%%%%%%%%%%%%%%%%%%%%%%%%%%%%%%%%%%%%%%%%%%%%%%%%%%%%%%%

\subsubsection*{Classical rod instabilities}
A rich variety of dynamical instabilities for elastic rods have been identified and studied in both the immersed and non-immersed settings. These include buckling instabilities \cite{goriely1997nonlinear,goriely2000nonlinear,goriely1996new,majumdar2014stability}, perversions \cite{mcmillen2002tendril,goriely2000nonlinear,wolgemuth2004dynamic,DOMOKOS2006}, bistable helices~\cite{goldstein2000bistable}, loop and kink formation \cite{porubov2002strain, goyal2005nonlinear, haijun1999spontaneous,goriely1998nonlinear}, twirling and whirling instabilities \cite{koehler2000twirling,wolgemuth2000twirling,lim2004simulations,goriely2000nonlinear}, and Michell's instability \cite{goriely2006twisted,zajac1962stability,michell1890small,bartels2020numerical,van2003instability}. We note that Michell's instability is particularly relevant here due to the importance of the strong coupling between bending and twisting in the well-posedness theory--see Section \ref{subsec:uncond_GWP}.

%%%%%%%%%%%%%%%%%%%%%%%%%%%%%%%%%%%%%%%%%%%%%%%%%%%%%%%

\subsubsection*{Model justification} %Numerical and applied analysis}
%additionally 
We mention further physical justification and applications for the models considered here. %, as well as potential applications of the Kirchhoff rod PDE theory developed here in numerical and applied analysis.
First, the Kirchhoff rod energy has been rigorously derived via $\Gamma$-convergence from 3D elasticity in~\cite{mora2002derivation,mora2004nonlinear,scardia2006nonlinear}. We also highlight the following articles justifying the form of forcing by preferred curvature considered in~\eqref{eq:2D_forced}: \cite{alouges2024some, sartori2016dynamic, riedel2007molecular, camalet2000generic}. 
In addition, we mention the numerical analysis results for Kirchhoff rods in~\cite{oelz2010derivation,lim2008dynamics,dziuk2002evolution,bartels2020numerical} and the convergence of discrete $N$-link rod approximations %used in numerical implementations
to the continuous rod in the planar setting~\cite{moreau2025n}. 
Finally, the analysis of~\eqref{eq:classical} and~\eqref{eq:Xss_0}-\eqref{eq:kap3_0} is part of the second author's program to understand a full model hierarchy of immersed filament dynamics, and paves the way for analyzing filament models which incorporate a more detailed treatment of the fluid-structure interaction~\cite{ohm2024free}.

%!TEX root = RFT GWP.tex

%%%%%%%%%%%%%%%%%%%%%%%%%%%%%%%%%%%%%%%%%%%%%%%%%%%%%%%
%%%%%%%%%%%%%%%%%%%%%%%%%%%%%%%%%%%%%%%%%%%%%%%%%%%%%%%
%%%%%%%%%%%%%%%%%%%%%%%%%%%%%%%%%%%%%%%%%%%%%%%%%%%%%%%
\section{Derivation of the models}
\label{sec:derivation}

\subsection{Resistive force theory}
We begin with a brief description of the origins of resistive force theory, i.e., the form of the relationship 
\begin{equation}\label{eq:RFT}
\bu = ({\bf I}+\gamma\X_s\otimes\X_s)\bm{f}
\end{equation}
between the force density $\bm{f}(s)$ along a slender filament immersed in a viscous fluid and the filament velocity $\bu(s)$.
As a consequence of the highly viscous (Stokesian) nature of the fluid, the relationship between velocity and drag for objects immersed in the fluid is linear. This velocity-drag relationship can be computed explicitly for objects with simple geometries, including slender prolate spheroids with minor-to-major axis ratio $\epsilon$ \cite{chwang1976hydromechanics,kim1991microhydrodynamics,lauga2020fluid}. In the slender limit, the asymptotic expressions for the force-velocity relation in the directions normal and tangent, respectively, to the center cross section of the spheroid are
\begin{equation}\label{eq:prolate}
\bm{f}_\perp = c_\perp \bu_\perp\,, \quad \bm{f}_\parallel = c_\parallel \bu_\parallel\,, \qquad c_\perp\approx \frac{4\pi \mu }{\log(2/\epsilon)+\frac{1}{2}}\,, \quad c_\parallel \approx \frac{2\pi \mu}{\log(2/\epsilon)-\frac{1}{2}}\,, 
\end{equation}
where $\mu$ is the dynamic viscosity. The idea behind resistive force theory~\cite{gray1955propulsion,cox1970motion,batchelor1970slender} is to adopt this relationship pointwise along a curved filament, yielding 
\begin{equation}\label{eq:inverse}
\bm{f}(s)=[c_\perp ({\bf I}-\X_s\otimes\X_s)+c_\parallel \X_s\otimes\X_s]\bu(s)\,.
\end{equation}
Taking the drag force to be in balance with the force due to elastic stresses within the filament (typically modeled using Euler-Bernoulli beam theory), we may invert \eqref{eq:inverse} to obtain, under the time rescaling $t_{\rm new} = c_{\parallel} \,t_{\rm old}$, the expression \eqref{eq:RFT} relating forces within the filament to the velocity along the filament. In accordance with \eqref{eq:prolate}, the contribution of forces in the tangential direction along $\X(s)$ is weighted $(1+\gamma)$ times as much as the normal directions in computing the velocity along the filament, where $\gamma\to 1$ for an infinitely slender filament. Resistive force theory with $\gamma=1$ can also be derived as the leading order term in nonlocal slender body theory \cite{keller1976swimming,lighthill1976flagellar,keller1976slender,johnson1980improved}. The sense in which nonlocal slender body theory approximates the Stokes flow due to a line force density along the length of the filament is made precise in \cite{closed_loop,free_ends}.

%%%%%%%%%%%%%%%%%%%%%%%
\subsection{Derivation of the coupled rod--frame evolution}\label{subsec:derivation}
Here we derive equations for the evolution of an immersed rod with intrinsic preferred curvature and twist. As mentioned, these equations, given in the unforced setting by~\eqref{eq:Xss_0}-\eqref{eq:kap3_0}, may be considered as the gradient flow of the energy~\eqref{eq:Eintrinsic} under the metric~\eqref{eq:metric2}. Here we provide the details of from where this metric comes. % from. 

Recall that, under the assumptions of Kirchhoff rod theory, we must track not only the evolution of the curve $\X$ but also the evolution of the frame $(\e_{\rm t},\e_1,\e_2)$ satisfying~\eqref{eq:frame_ODE} along the curve. The general equations describing curve-frame kinematics are
%The coupling between the curve and frame evolutions satisfies 
\begin{align}
\p_t \X &= \bu \, , \quad \bu_s \cdot \X_s = 0 \, , \label{eq:Xdot}\\
\p_t\e_{\rm t} &= \bu_s  \label{eq:etdot} \\
\p_t \e_1 &= \underbrace{\e_2 \times \bu_s}_{= -(\bu_s \cdot \e_1)\e_{\rm t}} + \,\omega \e_2  \label{eq:e1dot}\\
\p_t \e_2 &= \underbrace{- \e_1 \times \bu_s}_{= -(\bu_s \cdot \e_2)\e_{\rm t}} - \,\omega \e_1 \,. \label{eq:e2dot}
\end{align}
Equations \eqref{eq:e1dot} and \eqref{eq:e2dot} describe the local tilting of the frame in the $\e_{\rm t}$ direction as the centerline moves, along with the in-plane angular velocity %twist rate
$\omega=\p_t\e_1\cdot\e_2=-\p_t\e_2\cdot\e_1$ induced by the centerline motion. Using \eqref{eq:frame_ODE} and \eqref{eq:Xdot}-\eqref{eq:e2dot}, we may write the curve evolution in terms of the curvature components and $\kappa_3$:
\begin{align}
\p_t \kappa_1 &= \p_t \X_{ss} \cdot \e_1 + \X_{ss} \cdot \p_t \e_1 
= \bu_{ss} \cdot \e_1 + \omega \kappa_2  \label{eq:kappa1dot}\\
\p_t \kappa_2 &= \p_t \X_{ss} \cdot \e_2 + \X_{ss} \cdot \p_t \e_2 
= \bu_{ss} \cdot \e_2 - \omega \kappa_1 \label{eq:kappa2dot} \\
\p_t \kappa_3 &= \p_s (\p_t \e_1) \cdot \e_2 + (\p_s\e_1) \cdot (\p_t \e_2) 
= -(\bu_s \cdot \e_1) \kappa_2 + (\bu_s \cdot \e_2) \kappa_1 + \p_s \omega\nonumber \\ 
&= \bu_s \cdot (\X_s\times\X_{ss})+ \p_s \omega \, . \label{eq:kappa3dot}
\end{align}

%Wolgemuth calls this ``tangential component of the angular velocity of the material frames"

Here the velocity field $\bu$ will be determined via resistive force theory~\eqref{eq:RFT}, %, $\bu=({\bf I}+\gamma\X_s\otimes\X_s)\bm{f}$,
and the in-plane angular velocity %local twist rate
$\omega$ is linearly related to the local moment density $M$ about the tangent direction of the filament:
\begin{equation}\label{eq:Talpha}
\omega(s,t)=\frac{1}{\alpha}M(s,t)\,.
\end{equation}
The rotational friction coefficient $0<\alpha\ll1$ is due to viscous effects of the surrounding fluid \cite{wolgemuth2004dynamic,lough2023self,wada2011geometry,wolgemuth2000twirling}. For an immersed filament with finite radius $\epsilon>0$, under our scaling of time, %\footnote{Recall that to obtain \eqref{eq:RFT} from \eqref{eq:inverse}, we rescaled time by $C\abs{\log\epsilon}$.},
the parameter $\alpha$ satisfies
\begin{equation}\label{eq:alphadef}
\alpha \approx \frac{4\pi \mu \epsilon^2}{c_{\parallel}} \approx 2 \epsilon^2 \log \frac{2}{\epsilon} \, .
\end{equation}

It remains to determine the force density $\bm{f}(s,t)$ and the (scalar) cross sectional moment density $M(s,t)$ along the filament. Recall the form \eqref{eq:Eintrinsic} of the energy $\E$, which we separate into three components as:
\begin{equation}\label{eq:Eintrinsic2}
\E[\X,\e_1,\e_2] = \frac{1}{2} \int_{I} \bigg(\abs{\kappa_1-\zeta_1}^2 + \abs{\kappa_2-\zeta_2}^2 + \eta \abs{\kappa_3-\zeta_3}^2\bigg)\,ds =: \E_1 + \E_2 + \eta \E_3 \, .
\end{equation}
As in \cite{olson2013modeling, wolgemuth2004dynamic, moulton2013morphoelastic,haijun1999spontaneous}, we will find $\bm{f}$ and $M$ by taking the variational derivative of this energy with respect to both the curve $\X$ and the frame $(\e_1,\e_2)$.

The configuration space, on which the energy \eqref{eq:Eintrinsic2} is defined, is (informally, without making precise the regularity of the curve and frame)
\begin{equation}
\begin{aligned}
    Q &= \{ (\X,\e_1,\e_2) \; \big| \;  \X,\e_1,\e_2 : [0,1] \to \R^3 \,,\, |\X_s|^2 = |\e_1|^2 = |\e_2|^2 = 1 \,,\\
    &\hspace{3cm} \X_s \times \e_1 = \e_2 \,,\, \text{BCs} \}
\end{aligned}    
\end{equation}
where the boundary conditions (BCs) are either periodicity (in which case $\X,\e_1,\e_2 : \T \to \R^3$) or the appropriate BCs for a free ended filament. These may include the force- and torque-free conditions 
\begin{equation}
\X_{sss}\big|_{s=0,1} = 0\,, \quad  \X_{ss}\big|_{s=0,1} = 0\,,
\end{equation}
in which case the corresponding boundary condition for $\kappa_3$ is $(\kappa_3)_s=0$. We may consider replacing one end with clamped boundary conditions 
\begin{equation}
\X\big|_{s=0} = 0\,, \quad \X_s\big|_{s=0} = \e_z\,,
\end{equation}
where $\e_z$ is a fixed unit vector which we may take to point along the $z$-axis in a fixed Cartesian coordinate system. The corresponding boundary condition for $\kappa_3$ at the clamped end is $(\kappa_3)_s=0$, i.e., no twist is allowed to leave through the clamped end. 
Here we consider only the force- and torque-free boundary conditions, but note that the clamped conditions are also important in practice.

We consider variations $\X(s,\varepsilon)$, $\e_1(s,\varepsilon)$, $\e_2(s,\varepsilon)$, which, using the inextensibility constraint and orthonormality of the frame, necessarily satisfy
\begin{align}
\X(s,\varepsilon) &= \X(s,0) + \varepsilon \dot \X(s,0) + o(\varepsilon) \, , \quad \dot \X_s(s,0) \cdot \X_s(s,0) = 0 \, , \\
\e_1(s,\varepsilon) &= \e_1(s,0) - \varepsilon (\e_1(s,0) \cdot \dot\X_s(s,0)) \e_{\rm t}(s,0) + \varepsilon a(s) \e_2(s,0) + o(\varepsilon) \, , \\
\e_2(s,\varepsilon) &= \e_2(s,0) - \varepsilon (\e_2(s,0) \cdot\dot\X_s(s,0)) \e_{\rm t}(s,0) - \varepsilon a(s) \e_1(s,0) + o(\varepsilon) \, .
\end{align}

Therefore, the tangent space $\mc{T}_{(\X,\e_1,\e_2)} Q$ to the configuration space at $(\X,\e_1,\e_2)$ consists precisely of those perturbations $(\dot \X,\dot \e_1,\dot \e_2)$ satisfying
\begin{align}
\dot \X_s \cdot \X_s = 0 \,, \quad 
\dot\e_1 =  -(\e_1 \cdot \dot\X_s)\e_{\rm t} + a(s)\e_2  \,, \quad 
\dot\e_2 = -(\e_2 \cdot \dot\X_s)\e_{\rm t} - a(s)\e_1  \, .
\end{align}
It is convenient to parameterize the perturbations to the frame by scalar-valued functions $a(s)$; we may therefore identify the tangent space with
\begin{equation}
    \label{eq:tangentspace}
    \mc{T}_{(\X,\e_1,\e_2)} Q \cong \{ (\dot \X, a) : \dot \X_s \cdot \X_s = 0 \} \, .
\end{equation}
In summary, the admissible perturbations are (1) normal perturbations $\dot{\X}$ of the curve, which also `tilt' the frame $\e_1,\e_2$ in the direction of $\e_{\rm t}$, and (2) `twists' $a$ of the frame in the plane of the cross section. The inner product on the tangent space will be
\begin{equation}
\int_{I} \dot{\X} \cdot \dot{\bm{Y}} \, ds + \int_{I} a(s) b(s) \, ds \, .
\end{equation}

We may calculate 
\begin{equation}
    \label{eq:firstvariationofE1}
\begin{aligned}
\frac{d}{d\varepsilon}\bigg|_{\varepsilon=0} \E_1 &= 
\frac{d}{d\varepsilon}\bigg|_{\varepsilon=0} \frac{1}{2}\int_{I} \big((\X_{ss}+\varepsilon\dot\X_{ss}-\zeta_1\e_1-\zeta_2\e_2)\cdot\e_1 \big)^2\,ds \\
&= \int_{I} (\kappa_1-\zeta_1) \big(\dot \X_{ss} \cdot \e_1 + (\X_{ss}-\zeta_2\e_2) \cdot \dot \e_1\big)\,ds \\
% &= \int_{I} \bigg([({\bf I}-\X_s\otimes\X_s)((\kappa_1 - \zeta_1) \e_1)_{s}]_s \cdot \dot \X + (\kappa_1-\zeta_1)(\kappa_2-\zeta_2) a(s)\bigg) \,ds \, .
&= \int_{I} \bigg(\big[((\kappa_1 - \zeta_1) \e_1)_{s}-\lambda_1\X_s\big]_s \cdot \dot \X + (\kappa_1-\zeta_1)(\kappa_2-\zeta_2) a(s)\bigg) \,ds \, ,
\end{aligned}
\end{equation}
where $\lambda_1$ is a Lagrange multiplier to enforce that $\big[((\kappa_1 - \zeta_1) \e_1)_{s}-\lambda_1\X_s\big]_s$ belongs to the tangent space \eqref{eq:tangentspace}, i.e. $\big[((\kappa_1 - \zeta_1) \e_1)_{s}-\lambda_1\X_s\big]_{ss}\cdot\X_s=0$. In particular, using the duple notation of~\eqref{eq:tangentspace}, we may write 
\begin{equation}
    \label{eq:gradientofE1}
\nabla \E_1 = \frac{\p \E_1}{\p \X} + \frac{\p \E_1}{\p a} = \big(\big[((\kappa_1 - \zeta_1) \e_1)_{s}-\lambda_1\X_s\big]_s,0\big) + \big(0,(\kappa_1-\zeta_1)(\kappa_2-\zeta_2)\big) \, .
\end{equation}
Similarly,
\begin{equation}
    \label{eq:gradientofE2}
\nabla \E_2 = \frac{\p \E_2}{\p \X} + \frac{\p \E_2}{\p a} = 
\big( \big[((\kappa_2 - \zeta_2) \e_2)_{s}-\lambda_2\X_s\big]_s, 0\big) + \big(0,- (\kappa_1-\zeta_1)(\kappa_2-\zeta_2) \big) \, .
\end{equation}
Notice the extra minus sign in the frame component.
Finally,
\begin{equation}
\begin{aligned}
\frac{d}{d\varepsilon}\bigg|_{\varepsilon=0} \E_3 
&= \frac{d}{d\varepsilon}\bigg|_{\varepsilon=0} \frac{1}{2} \int_{I}\big(\p_s(\e_1+\varepsilon\dot\e_1)\cdot(\e_2+\varepsilon\dot\e_2)-\zeta_3\big)^2\,ds\\
&= \int_{I} (\kappa_3-\zeta_3) \big(-(\e_1\cdot\dot\X_s)\kappa_2+\p_sa(s) +(\e_2\cdot\dot\X_s)\kappa_1 \big)\, ds \\
&= -\int_{I} \bigg((\kappa_3-\zeta_3)_sa(s)+ \big[(\kappa_3-\zeta_3)(-\kappa_2\e_1 +\kappa_1\e_2 )+\lambda_3\X_s\big]_s\cdot\dot\X\bigg)\, ds \,.
\end{aligned}
\end{equation}
Therefore,
\begin{equation}
    \label{eq:gradientofE3}
\nabla \E_3 = \frac{\p \E_3}{\p \X} + \frac{\p \E_3}{\p a} = \big(\big[(\kappa_3 - \zeta_3)(\kappa_2 \e_1-\kappa_1\e_2)-\lambda_3\X_s\big]_s, 0 \big)+ \big(0, - (\kappa_3-\zeta_3)_s \big) \, , 
\end{equation}
where we note that 
\begin{equation}
\kappa_2 \e_1- \kappa_1\e_2  = -\X_s\times\X_{ss} \, .
\end{equation}

By the principle of virtual work, we expect the force $\bm{f}(s)$ on each filament cross section to be in balance with variations of the energy \eqref{eq:Eintrinsic} with respect to the curve: 
\begin{equation}\label{eq:f_form}
\begin{aligned}
- \bm{f} &= \frac{\p\E_1}{\p \X} + \frac{\p\E_2}{\p \X} + \eta \frac{\p \E_3}{\p \X} \\
&= \big[(\X_{ss}-\zeta_1\e_1-\zeta_2\e_2)_s - \eta (\kappa_3-\zeta_3)\X_s\times\X_{ss}- (\lambda_1+\lambda_2+\lambda_3 +\lambda_4)\X_s\big]_s\,.
\end{aligned}
\end{equation}
Here $\lambda_4$ is an additional tangential contribution serving to enforce the inextensibility constraint $\abs{\X_s}=1$ \emph{after} plugging $\bm{f}$ into \eqref{eq:RFT}. We may consider the Lagrange multiplier terms collectively as the tension $\lambda$, yielding the velocity expression
% In addition to $\bm{n}$, we include the unknown tangential contribution $\wt\lambda\X_s$, which includes the filament tension and may be determined via the inextensibility constraint. As in \cite{oelz2011curve}, we may consider all tangential terms together, and that the filament tension $\lambda$ is given by $\lambda=-\wt\lambda-\kappa^2$. In total, the force density along the filament is $\bm{f}=\bm{n}+(\wt\lambda\X_s)_s$, yielding the velocity expression
\begin{equation}\label{eq:u_RFT}
    \frac{\p\X}{\p t}=\bu = -({\bf I}+\gamma\X_s\otimes\X_s)\big[(\X_{ss}-\zeta_1\e_1-\zeta_2\e_2)_s - \eta (\kappa_3-\zeta_3)\X_s\times\X_{ss}- \lambda\X_s\big]_s \,.
\end{equation}
The cross sectional moment density $M$
is given by 
\begin{equation}\label{eq:M_form}
- M = \frac{\p\E_1}{\p a} + \frac{\p\E_2}{\p a} + \eta \frac{\p \E_3}{\p a} 
= - \eta(\kappa_3-\zeta_3)_s \, .
\end{equation}
Note that the moment contributions from $\E_1$ and $\E_2$ cancel. Recalling the rotational friction coefficient $0<\alpha\ll1$ appearing due to the surrounding viscous fluid \eqref{eq:Talpha}, the twist rate density $\omega(s,t)$ is given by
\begin{equation}\label{eq:T_form}
\omega = \frac{\eta}{\alpha}(\kappa_3-\zeta_3)_s \, .
\end{equation}

Using the above expressions \eqref{eq:u_RFT} and \eqref{eq:T_form} for $\bu$ and $\omega$, we may calculate the full right-hand side of \eqref{eq:kappa1dot}-\eqref{eq:kappa3dot} as
\begin{align}
\p_t \kappa_1 
    &= -(\kappa_1 - \zeta_1)_{ssss} + (\kappa_3)_{sss}(\kappa_2 - \zeta_2)-\eta\kappa_2(\kappa_3 - \zeta_3)_{sss} +\mc{R}_1[\kappa_j,\zeta_j,\lambda] \label{eq:kap1} \\
\p_t \kappa_2 
    &= -(\kappa_2 - \zeta_2)_{ssss} - (\kappa_3)_{sss}(\kappa_1 - \zeta_1) + \eta\kappa_1(\kappa_3 - \zeta_3)_{sss} + \mc{R}_2[\kappa_j,\zeta_j,\lambda] \label{eq:kap2} \\
\p_t \kappa_3 
    &= \frac{\eta}{\alpha}(\kappa_3 - \zeta_3)_{ss} +\eta(\kappa_1^2+\kappa_2^2)(\kappa_3 - \zeta_3)_{ss} -\big( \kappa_1(\kappa_1 - \zeta_1)+\kappa_2(\kappa_2 - \zeta_2)\big)(\kappa_3)_{ss}   \nonumber \\
    &\qquad  -\kappa_1(\kappa_2 - \zeta_2)_{sss}+\kappa_2(\kappa_1 - \zeta_1)_{sss} +\mc{R}_3[\kappa_j,\zeta_j,\lambda]  \label{eq:kap3}\,,
\end{align}
along with the tension equation
\begin{equation}\label{eq:wtlambda}
    (1+\gamma)\lambda_{ss}- (\kappa_1^2+\kappa_2^2)\lambda = -\mc{R}_{\rm t}[\kappa_j,\zeta_j] \,.
\end{equation}
The remainder terms $\mc{R}_i$ are recorded in Appendix~\ref{app:forcing}.
%
% with
% \begin{equation}
% \begin{aligned}
%     \mc{R}_1&= \big(\mc{R}_{\rm a}\big)_s  + (\kappa_3)_{ss}(\kappa_2 - \zeta_2)_s -\eta\big((\kappa_2)_s+\kappa_1\kappa_3\big)(\kappa_3 - \zeta_3)_{ss}\\
%     &\quad +\kappa_3(\kappa_2 - \zeta_2)_{sss} +\kappa_3 (\kappa_3)_{ss}(\kappa_1 - \zeta_1) -\kappa_3\mc{R}_{\rm b}  + \frac{\eta}{\alpha}\kappa_2(\kappa_3-\zeta_3)_s\\
%     %
%     %
%     \mc{R}_2&= 
%     \big(\mc{R}_{\rm b}\big)_s - (\kappa_3)_{ss}(\kappa_1 - \zeta_1)_s -\eta\big(\kappa_2\kappa_3-(\kappa_1)_s\big)(\kappa_3 - \zeta_3)_{ss}  \\
%     &\quad -\kappa_3(\kappa_1 - \zeta_1)_{sss}+ \kappa_3(\kappa_3)_{ss}(\kappa_2 - \zeta_2)  +\kappa_3\mc{R}_{\rm a} - \frac{\eta}{\alpha}\kappa_1(\kappa_3-\zeta_3)_s \\
%     %
%     %
%     \mc{R}_3 &= -\kappa_2\mc{R}_{\rm a}+ \kappa_1\mc{R}_{\rm b} \,.
% \end{aligned}
% \end{equation}

The formulation \eqref{eq:kap1}-\eqref{eq:kap3} in terms of $(\kappa_1,\kappa_2,\kappa_3)$ is not the cleanest-looking, not is it the most convenient formulation for analysis. We display it here to highlight the complicated effect of forcing by intrinsic curvature and twist, especially in regard to the potential backward heat. We proceed in the absence of intrinsic forcing, in which case \eqref{eq:f_form} becomes
%In the case of zero intrinsic forcing ($\zeta_1=\zeta_2=\zeta_3=0$),
\begin{equation}\label{eq:f_form2}
 \bm{f} = -\big(\X_{sss}-\lambda\X_s - \eta \kappa_3\X_s\times\X_{ss}\big)_s\,,
\end{equation}
which, combined with \eqref{eq:u_RFT}, yields the curve evolution equation \eqref{eq:Xss_0}.

\subsection{Energy balance}
To compute the evolution of the energy, we revisit the notion that~\eqref{eq:Xss_0}-\eqref{eq:kap3_0} is a gradient flow of the energy~\eqref{eq:Eintrinsic2} without preferred curvature, which we simply write as $E$, with respect to the metric in~\eqref{eq:metric2}. The abstract theory of gradient flows $\dot x = -\nabla_g V$ tells us that $\dot V = - |\nabla_g V|_g^2$. For rods, this is
\begin{equation}
	\label{eq:E_ID}
\frac{dE}{dt} = - \int_I \big( (\mathbf{I} + \gamma \X_s \otimes \X_s) \wt{\Z}_s \big) \cdot \wt{\Z}_s ds- \frac{\eta^2}{\alpha} \int_I (\kappa_3)_s^2 \, ds
\end{equation}
where
\begin{equation}
	\wt{\Z}_s := (\X_{sss} - \lambda \X_s - \eta \kappa_3 \X_s \times \X_{ss})_s \, .
\end{equation}
While the abstract theory gives~\eqref{eq:E_ID} immediately, it is nevertheless instructive to compute $\frac{dE}{dt}$ directly: Multiply~\eqref{eq:Xss_0} by $(\X_{sss} -\lambda\X_s- \eta \kappa_3\X_s\times\X_{ss})_s$ and integrate by parts to obtain
\begin{equation}\label{eq:Xssnew}
\begin{aligned}
&\frac{1}{2} \frac{d}{dt} \int_{I}\abs{\X_{ss}}^2\,ds + \eta\int_I \frac{\p\X_s}{\p t}\cdot(\kappa_3\X_s\times\X_{ss})\,ds\\
&\quad=-\int_{I} \abs{(\X_{sss} -\lambda\X_s- \eta \kappa_3\X_s\times\X_{ss})_s}^2\,ds \\
&\qquad - \gamma\int_{I} \big(\X_s\cdot(\X_{sss} -\lambda\X_s- \eta \kappa_3\X_s\times\X_{ss})_s\big)^2\,ds\,.
\end{aligned}
\end{equation}
% From \eqref{eq:kappa3dot}, the evolution of $\kappa_3$ may be written 
% \begin{equation}\label{eq:kap3new}
% \p_t\kappa_3 = \frac{\p\X_s}{\p t}\cdot(\X_s\times\X_{ss}) + \frac{\eta}{\alpha}(\kappa_3)_{ss}\,.
% \end{equation}
Multiply~\eqref{eq:kap3_0} by $\eta \kappa_3$ and integrate by parts to obtain
\begin{equation}
	\label{eq:muhkappa3thing}
\frac{\eta}{2} \frac{d}{dt} \int_I \kappa_3^2 \, ds = \eta\int_I \frac{\p\X_s}{\p t}\cdot(\kappa_3\X_s\times\X_{ss})\,ds - \frac{\eta^2}{\alpha} \int_I (\kappa_3)_s^2 \, ds \, .
\end{equation}
Summing~\eqref{eq:Xssnew} and~\eqref{eq:muhkappa3thing} and cancelling the terms containing $\frac{\p\X_s}{\p_t}$ yields~\eqref{eq:E_ID}.
% and add the result to \eqref{eq:Xssnew} to obtain the energy identity
% \begin{equation}
% \begin{aligned}
% \frac{1}{2}\p_t&\int_{I}\bigg(\abs{\X_{ss}}^2 + \eta\kappa_3^2\bigg)\,ds 
% =-\int_{I} \abs{(\X_{sss} -\lambda\X_s- \eta \kappa_3\X_s\times\X_{ss})_s}^2\,ds \\
% &\qquad - \gamma\int_{I} \big(\X_s\cdot(\X_{sss} -\lambda\X_s- \eta \kappa_3\X_s\times\X_{ss})_s\big)^2\,ds - \frac{\eta^2}{\alpha}\int_{I}(\kappa_3)_s^2\,ds\,.
% \end{aligned}
% \end{equation}
%The rod-frame system \eqref{eq:Xss_0}-\eqref{eq:kap3_0} thus has a natural nonincreasing energy which will be used in showing Theorem \ref{thm:GWP_curveframe}.

In the case $\eta=0$, the rod-frame energy identity~\eqref{eq:E_ID} reduces to the energy identity for the curve-only setting:
\begin{equation}\label{eq:energy}
\frac{1}{2} \frac{d}{dt} \int_{I}\abs{\X_{ss}}^2\,ds = -\int_I\abs{(\X_{sss}-\lambda\X_s)_s}^2\,ds - \gamma\int_I\big(\X_s\cdot(\X_{sss}-\lambda\X_s)_s\big)^2\,ds\,.
\end{equation}

%!TEX root = RFT GWP.tex

\section{Immersed filament evolution without forcing}\label{sec:immersednoforcing}

In this section, we prove Theorem~\ref{thm:GWP_classical} for the curve-only evolution~\eqref{eq:classical} on $I = [0,1]$ and $\T$. It will be convenient to consider~\eqref{eq:classical} as a closed system for the tangent vector $\X_s$:
\begin{align}
\p_t \X_s &= - \p_s^4 \X_s + (\lambda \X_s)_{ss} + 3\gamma [(\X_{sss} \cdot \X_{ss}) \X_s]_s + \gamma (\lambda_s \X_s)_s \, , \label{eq:section3xsequation} \\
0&=(1+\gamma) \lambda_{ss} - |\X_{ss}|^2 \lambda + (3\gamma+4) (\X_{sss} \cdot \X_{ss})_s - |\X_{sss}|^2 \, ,
\end{align}
with zero-stress and zero-moment boundary conditions at the endpoints when $I=[0,1]$:
\begin{equation}
\X_{ss}, \, \X_{sss}\big|_{s=0,1} = 0 \, , \quad \lambda\big|_{s=0,1} = 0 \, .
\end{equation}
The constraint $|\X_s|^2 = 1$ is propagated by the evolution. To recover the motion of $\X$, one may integrate the equation \eqref{eq:classical} as in \eqref{eq:swimming}
% \begin{equation}
% \p_t \X = - (\mathbf{I} +\gamma \X_s \otimes \X_s) (\X_{sss} - \lambda \X_s)_s
% \end{equation}
to obtain the evolution for the center-of-mass:
\begin{equation}
\begin{aligned}
\frac{d}{dt} \int_I \X \, ds &= \gamma \int_I (3\X_{ss} \cdot \X_{sss}+\lambda_s) \X_s \, ds \, ,
\end{aligned}
\end{equation}
where we have used identities stemming from differentiating the inextensibility constraint $\abs{\X_s}^2=1$ to rewrite the right-hand side.
Throughout, we note that $\gamma \geq 0$ (with $\gamma\le 1$ for most physically meaningful situations), and we will not track dependence of implicit constants on $\gamma$.

We highlight two key difficulties in obtaining Theorem \ref{thm:GWP_classical}. The system~\eqref{eq:section3xsequation} has a scaling symmetry
\begin{equation}
    \label{eq:scalingsym}
(\X_s ,\lambda) \mapsto (\X_s(\theta s,\theta^4 t),\theta^2\lambda(\theta s,\theta^4 t)) \, , \quad \theta > 0 \, ,
\end{equation}
and standard heuristics in nonlinear PDEs would suggest that~\eqref{eq:section3xsequation} should be well-posed in (sub)critical spaces, e.g., $H^{\sigma}$ ($\sigma > 1/2$) or $C^\alpha$, whose norms do not increase (decrease) when ``zooming in" with the symmetry~\eqref{eq:scalingsym}. However, in practice, the need to solve the elliptic equation for the filament tension $\lambda$ causes the energy space $\X_{ss} \in L^2_s$ to act as critical; solving the PDE in the energy space requires the full power of the smoothing. See \cite[Remark 2.4]{mori2023well} for additional discussion. The second difficulty, again tied to the tension, is to extract control on $\X_{ssss}$ from the energy dissipation quantity $(\X_{sss}-\lambda\X_s)_s$. A key insight of \cite{koiso1996motion} is to separate out the effects of the tension at the level of $\X_{sss}-\lambda\X_s$. % to obtain this control.

\subsection{Linear theory}
\label{sec:lineartheory}
We begin by summarizing linear estimates which will be used here and throughout the paper.

\subsubsection{Parabolic estimates}
Consider the 1D biharmonic heat equation
\begin{equation}
    \label{eq:biharmonicheat}
\begin{aligned}
\p_t u + \p_s^4 u &= f \\
u\big|_{t=0} &= u^{\rm in} \, ,
\end{aligned}
\end{equation}
supplemented with clamped boundary conditions when $I=[0,1]$:
\begin{equation}
    \label{eq:Dirichletbcsbiharmonic}
u\,,\; \p_s u\big|_{s=0,1} = 0 \, .
\end{equation}
The problem~\eqref{eq:biharmonicheat}-\eqref{eq:Dirichletbcsbiharmonic} is relevant at the level of $u = \X_{ss}$.

%When $I = \T$, the linear theory can be developed on the basis is Fourier series, so we focus on the more difficult $I=[0,1]$.

%Let us summarize the basic $L^2$ theory for the problem~\eqref{eq:biharmonicheat}-\eqref{eq:Dirichletbcsbiharmonic}, without semigroup theory.
Let $u^{\rm in} \in L^2_s$ and $f=0$. Then there exists a unique weak solution
\begin{equation}
u \in C(\bar{\R_+};L^2_s) \cap L^2_t H^2_s(I \times \R_+)
\end{equation}
which satisfies
\begin{equation}
    \label{eq:energyeqbelowwhich}
\frac{1}{2} \frac{d}{dt} \| u(\cdot,t) \|_{L^2_s}^2 = - \| \p_s^2 u \|_{L^2_s}^2 \, ,
\end{equation}
thereby yielding global-in-time estimates in the energy space.

If $u^{\rm in} \in H^2_s$ satisfies the boundary conditions~\eqref{eq:Dirichletbcsbiharmonic}, one can multiply~\eqref{eq:biharmonicheat} by $\p_t u$ to obtain \emph{a priori} estimates ``raised" by two spatial derivatives:
\begin{equation}
	\label{eq:secondestimate}
\frac{1}{2} \frac{d}{dt} \| \p_s^2 u \|_{L^2_s}^2 = - \| \p_t u \|_{L^2_s}^2 \, .
\end{equation}

If moreover $u^{\rm in} \in H^4_s$, one can apply $\p_t$ to~\eqref{eq:biharmonicheat} and multiply by $\p_t u$, etc. These estimates form the basis for a solution theory in Sobolev spaces $H^{2k}_s$ with appropriate compatibility conditions on $u^{\rm in}$. However, only~\eqref{eq:energyeqbelowwhich} and~\eqref{eq:secondestimate} will be required in this paper.

The energy estimate~\eqref{eq:energyeqbelowwhich} can also accomodate a non-zero force
\begin{equation}
	\label{eq:fcond1}
f = f_0 + \p_s f_1 + \p_{ss} f_2 \, ,
\end{equation}
\begin{equation}
	\label{eq:fcond2}
f_0 \in L^1_t L^2_s \, , \quad f_1 \in L^{4/3}_t L^2_s \, , \quad f_2 \in L^2_t L^2_s \, ,
\end{equation}
while~\eqref{eq:secondestimate} can accomodate $f \in L^2_t L^2_s$. We have the following:

\begin{lemma}
    \label{lem:parabolicestimates}
Let $T > 0$ and $I = [0,1]$ or $\T$. For every $u^{\rm in} \in L^2(I)$ and $f$ satisfying~\eqref{eq:fcond1}-\eqref{eq:fcond2} on $I \times (0,T)$, there exists a unique weak solution
\begin{equation}
u \in L^\infty_t L^2_s \cap L^2_t H^2_s(I \times (0,T))
\end{equation}
satisfying
\begin{equation}
\| u - u^{\rm in} \|_{L^2_s} \to 0^+ \text{ as } t \to 0^+ \, .
\end{equation}
The weak solution is continuous on $[0,T]$ with values in $L^2_s$ and satisfies
\begin{equation}
\| u \|_{L^\infty_t L^2_s} + \| u \|_{L^2_t \dot H^{2}_s} \lesssim \| u^{\rm in} \|_{L^2_s} + \| f_0 \|_{L^1_t L^2_s} + \| f_1 \|_{L^{4/3}_t L^2_s} + \| f_2 \|_{L^2_t L^2_s} \, ,
\end{equation}
where the space-time norms are taken on $I \times (0,T)$.

For every $u^{\rm in} \in H^2(I)$ (satisfying~\eqref{eq:Dirichletbcsbiharmonic} when $I = [0,1]$) and $f \in L^2_t L^2_s$, the unique weak solution additionally belongs to $C([0,T];H^2_s) \cap L^2_t H^4_s$ and satisfies
\begin{equation}
\label{eq:higherlinearestimate}
\| \p_s^2 u \|_{L^\infty_t L^2_s} + \| \p_s^2 u \|_{L^2_t \dot H^2_s} \lesssim \| f \|_{L^2_t L^2_s} \, .
\end{equation}
\end{lemma}

\begin{remark}[Solvability at $\X_s$ level]
The above considerations were at the level of $\X_{ss}$. However, we often consider the evolution at the level of $\X_s$, for example,
\begin{equation}
	\label{eq:levelofXs}
\p_t v + \p_s^4 v = 0 \, , \quad v\big|_{t=0} = v^{\rm in} \, , \quad \p_s v\,,\, \p_s^2 v\big|_{s=0,1} = 0 \, ,
\end{equation}
with initial data $v^{\rm in} \in H^1(I)$.

To solve~\eqref{eq:levelofXs}, we consider $u^{\rm in} = \p_s v^{\rm in} \in L^2$. Let $u$ be the solution to the clamped problem~\eqref{eq:biharmonicheat}-\eqref{eq:Dirichletbcsbiharmonic} with initial data $u^{\rm in}$. This determines $v$ up to a constant, which can be recovered by the value of $v$ at $s=0$. Define
\begin{equation}
v(s,t) = \int_0^s u(s',t) \, ds' + v^{\rm in}(0) \, .
\end{equation}
Then $v$ solves~\eqref{eq:levelofXs}, and it belongs to $C([0,T];H^1_s) \cap L^2_t H^3_s(I \times (0,T))$.

The construction works the same with a forcing term $f$ on the right-hand side in~\eqref{eq:levelofXs}. Then $u$ solves the clamped problem with forcing term $\p_s f$, and
\begin{equation}
v(s,t) = \int_0^su(s',t) \, ds' + v^{\rm in}(0) + \int_0^t f(0,t') \, dt' \, .
\end{equation}
A suitable requirement is that $f$ satisfies~\eqref{eq:fcond1}-\eqref{eq:fcond2} with $H^1_s$ instead of~$L^2_s$.

\emph{Uniqueness}. Suppose that $\wt{v} \in C([0,T];H^1) \cap L^2_t H^3_s(I \times (0,T))$ is a second solution with initial data $v^{\rm in}$. Then $\wt{u} = \p_s \wt{v}$ solves the clamped problem~\eqref{eq:biharmonicheat}-\eqref{eq:Dirichletbcsbiharmonic}, and therefore $u - \wt{u}$ is zero. Hence, $v - \wt{v}$ is constant-in-space. By the equation $(\p_t + \p_s^4)(v - \wt{v}) = 0$, it is therefore constant-in-time. Since $(v - \wt{v})|_{t=0} = 0$, we may conclude that $v \equiv \wt{v}$.
\end{remark}

%%%%%%%%%%%%%%%%%%%%%%%%%%%%%%%%%
\subsubsection{Tension estimates}

We next derive bounds for the tension $\lambda$.

\begin{lemma}%[Tension estimates]
\label{lem:tension}
Let $\X_s \in H^3_s$. Then there exists a unique solution $\lambda \in H^1_0(0,1)$ in the free end setting or $\lambda \in H^1(\T)$ in the closed loop setting to
\begin{equation}
    \label{eq:tensionequationcurveforlemma}
(1+\gamma) \lambda_{ss} - |\X_{ss}|^2 \lambda + (3\gamma+4) (\X_{sss} \cdot \X_{ss})_s - |\X_{sss}|^2 = 0 
\end{equation}
satisfying the estimate
\begin{equation}
    \label{eq:finaltensionestimate}
\| \lambda \|_{H^1_s} \lesssim \| \X_{ss} \|_{H^1_s}^2 \, .
\end{equation}
If additionally $\X_s \in H^4_s$, then
\begin{equation}
    \label{eq:higherregtensionestimatecurve}
\| \lambda \|_{H^2_s} \lesssim \| \X_{ss} \|_{L^2_s} \| \X_{sss} \|_{L^2_s}^3 +\| \X_{sss} \|_{L^2_s} \| \X_{ssss} \|_{L^2_s} \, .
\end{equation}
\end{lemma}
\begin{proof}
From \eqref{eq:tensionequationcurveforlemma}, we have the energy estimate
\begin{equation}
    \label{eq:energyestimatefortension}
(1+\gamma) \int_I |\lambda_s|^2\,ds + \int_I |\X_{ss}|^2 \lambda^2\,ds = (\gamma+2) \int_I (\X_{sss} \cdot \X_{ss}) \lambda_s\,ds - \int_I |P_{\X_s}^\perp \X_{sss}|^2 \lambda\,ds \, .
\end{equation}
The right-hand side may be estimated according to
\begin{equation}
\int_I \big(|\X_{sss} \cdot \X_{ss}| |\lambda_s| + |P_{\X_s}^\perp \X_{sss}|^2 |\lambda|\big) \, ds \lesssim \| \X_{sss} \|_{L^2_s} \| \X_{ss} \|_{L^\infty_s} \| \lambda_s \|_{L^2_s} + \| \X_{sss} \|_{L^2_s}^2 \| \lambda \|_{L^\infty_s} \, .
\end{equation}
When $I = [0,1]$, both $\norm{\lambda_s}_{L^2_s}$ and $\norm{\lambda}_{L^\infty_s}$ may be controlled by the left-hand side term $\int_0^1 |\lambda_s|^2\,ds$ alone. When $I = \T$, it is necessary to also utilize the term $\int_\T |\X_{ss}|^2 \lambda^2\,ds$. Specifically, we need to demonstrate that the left-hand side of~\eqref{eq:energyestimatefortension} controls $\| \lambda \|_{H^1_s(\T)}^2$. For this, we exploit \emph{Fenchel's theorem} \cite{fenchel1951differential}:
\begin{equation}
\int_\T |\X_{ss}|^2\,ds \geq \int_\T |\X_{ss}|\,ds \geq 2\pi \, .
\end{equation}
For all $f\in L^2(\T)$ and all non-negative potentials $K$ with $\| K \|_{L^1} \geq 2\pi$, we have
\begin{equation}
\int_\T |f|^2\,ds \lesssim \int_\T K|f|^2\,ds + \int_\T |f_s|^2\,ds \, .
\end{equation}
The crucial point is to control the mean $m = \int_\T f\,ds$. Let $\wt{f} = f - m$. Then
\begin{equation}
\begin{aligned}
\int_\T K|m|^2\,ds &= \int_\T K|f - \wt{f}|^2\,ds \leq 2 K|f|^2 + 2 \int_\T K|\wt{f}|^2\,ds \\
&\lesssim \int_\T K|f|^2\,ds + \| \wt{f} \|_{L^\infty}^2 \int_\T K \, ds \, .
\end{aligned}
\end{equation}
In particular, $H^1(\T)$ is a Hilbert space with inner product
\begin{equation}
(f,g)_* := (1+\gamma) \int_\T f_sg_s\,ds + \int_\T |\X_{ss}|^2 fg\,ds \, .
\end{equation}
Then, according to the Riesz representation theorem, the linear PDE
\begin{equation}
(1+\gamma) \lambda_{ss} - |\X_{ss}|^2 \lambda = q
\end{equation}
is uniquely solvable in $H^1(\T)$ for all $q \in (H^1)^*$.

To prove~\eqref{eq:higherregtensionestimatecurve}, we view~\eqref{eq:tensionequationcurveforlemma} as an equality for $\lambda_{ss}$, estimate the remaining terms by a constant multiple of
\begin{equation}
\| \X_{ss} \|_{L^4}^2 \| \X_{ss} \|_{H^1}^2 + \| \X_{ssss} \|_{L^\infty} \| \X_{ss} \|_{L^2} + \| \X_{sss} \|_{L^4}^2 \, ,
\end{equation}
and then exploit interpolation. %\lo{who is "the" interpolation inequalities?}
\end{proof}

\subsection{Local well-posedness}

To establish global well-posedness for \eqref{eq:section3xsequation}, we begin with local well-posedness and, in the next section, argue that the short-time solution can be continued indefinitely.

\begin{proposition}
\emph{(Local well-posedness)} Let $\X_s^{\rm in} \in H^1_s$. Then there exists $T > 0$ and a unique solution
\begin{equation}
    \label{eq:topologyforLWP}
\X_s \in C([0,T];H^1_s) \cap L^2_t H^3_s(I \times (0,T))
\end{equation}
to \eqref{eq:section3xsequation} which depends Lipschitz-continuously on its initial data $\X_s^{\rm in} \in H^1_s$ in the topology of~\eqref{eq:topologyforLWP}.\footnote{The solution is unique in the sense that any two such solutions on time intervals $[0,T_1]$, $[0,T_2]$ agree on their mutual existence time.} For solutions in the above class, if $|\X_s^{\rm in}|^2 = 1$, then $|\X_s|^2 = 1$.

\emph{(Blow-up criterion)} Suppose that $\X_s$ is the maximally defined solution with maximal time of existence $T^* \in (0,+\infty]$. If $T^* < +\infty$, then
\begin{equation}
\| \X_s \|_{L^\infty_t H^1_s \cap L^4_t H^2_s(I \times (0,T^*))} = +\infty \, .
\end{equation}

\emph{(Higher regularity)} If $\X_s^{\rm in} \in H^3_s$ satisfies the boundary conditions, then
\begin{equation}
\X_s \in C([0,T];H^3_s) \cap L^2_t H^5_s(I \times (0,T)) \, , \quad \forall T < +\infty \, .
\end{equation}
\end{proposition}

For solutions satisfying~\eqref{eq:topologyforLWP}, the tension $\lambda$ belongs to $L^2_t H^1_s(I \times (0,T))$.

\begin{proof}
\textbf{1. Local well-posedness in the energy space}. Let $T>0$. We seek a fixed point of the map $\Y_s \mapsto \X_s$, where
\begin{equation}
    \label{eq:Duhamelforcurve}
\begin{aligned}
\X_s &= e^{-t\p_s^4} \X_s^{\rm in} + \int_0^t e^{-(t-t')\p_s^4} \p_s [(\lambda \Y_s)_s + \gamma (\lambda_s \Y_s)] \, dt' \\
&\qquad + 3 \gamma \int_0^t e^{-(t-t')\p_s^4}\p_s (\Y_{sss} \cdot \Y_{ss}) \Y_{s} \, dt'\,,\\
\lambda &= \mc{T}[\Y_{s}]^{-1} \big(|\Y_{sss}|^2 - (3\gamma+4) (\Y_{sss}\cdot\Y_{ss})_s\big) \, ,
\end{aligned}
\end{equation}
and $\mc{T}[\Y_{s}]^{-1}(f)$ is the solution $\lambda$ of the equation
\begin{equation}
(1+\gamma) \lambda_{ss} - |\Y_{ss}|^2 \lambda = f \, .
\end{equation}
Here and throughout, the semigroup $e^{-\p_s^4}$ is understood with the appropriate boundary conditions. We employ the functional setup\footnote{We could also choose $\mathcal{Y}_1 := L^2(0,T;H^3_s) \cap {\rm BCs}$, so that the $\mathcal{Y}_0 \cap \mathcal{Y}_1$ is the natural energy space.}
\begin{equation}
\mathcal{Y}_0 = C([0,T];H^1_s) \, , \quad \mathcal{Y}_1 := L^4(0,T;H^2_s) \cap \{ {\rm BCs} \} \, .
\end{equation}
Let $M_0$ and $M_1$ be upper bounds for $\| \Y \|_{\mathcal{Y}_0}$ and $\| \Y \|_{\mathcal{Y}_1}$, respectively.

% \da{Mention Kato, mention interpolation}

\emph{The map stabilizes a ball}. We record the tension estimate
% \begin{equation}
% \| \lambda \|_{H^1} \lesssim \| \Y_s \|_{H^2}^2 \, .
% \end{equation}
\begin{equation}
\| \lambda \|_{L^2_t H^1_s} \lesssim \| \Y_s \|_{\mathcal{Y}_1}^2
\end{equation}
and the contribution of the initial data,
\begin{equation}
\| e^{-t\p_s^4} \X_s^{\rm in} \|_{\mathcal{Y}_0} + \underbrace{\| e^{-t\p_s^4} \X_s^{\rm in} \|_{\mathcal{Y}_1}}_{M^{\rm in}_1(T)} \lesssim \underbrace{\| \X_s^{\rm in} \|_{H^1}}_{M^{\rm in}_0} \, ,
\end{equation}
keeping in mind that $M^{\rm in}_1(T) = o_{T \to 0^+}(1)$, since $\| \cdot \|_{L^4_t H^2_s(I \times (0,T))}$ is an integral quantity.
%\lo{[Perhaps we want to explain the little o notation...?]}

The main estimates are for the integral terms in~\eqref{eq:Duhamelforcurve}. By the energy estimates,
\begin{equation}
\| \X_s \|_{\mathcal{Y}_0} + \| \X_s \|_{\mathcal{Y}_1} \lesssim \| (\lambda \Y_s)_s + \gamma (\lambda_s \Y_s) \|_{L^2_t L^2_s} + \| (\Y_{sss} \cdot \Y_{ss}) \Y_s \|_{L^2_t L^2_s} \, .
\end{equation}
We may then bound
\begin{equation}
\begin{aligned}
\| (\lambda \Y_s)_s \|_{L^2_t L^2_s} + \| \lambda_s \Y_s \|_{L^2_t L^2_s} \lesssim \| \lambda \|_{L^2_t H^1_s} \| \Y_s \|_{L^\infty_t H^1_s} &\lesssim M_1^2 M_0 \\
\| (\Y_{sss} \cdot \Y_{ss}) \cdot \Y_s \|_{L^2_t L^2_s} \lesssim \| \Y_{sss} \|_{L^4_t L^2_s} \| \Y_{ss} \|_{L^4_t L^\infty_s} \| \Y_s \|_{L^\infty_t H^1_s} &\lesssim T^{1/8} M_1^{3/2} M_0^{3/2} \, ,
\end{aligned}
\end{equation}
where we have used that $\| \Y_{ss} \|_{L^8_t L^\infty_s} \lesssim \| \Y_{sss} \|_{L^4_t L^2_s} \| \Y_s \|_{L^\infty_t H^1_s}$.

In conclusion, we obtain
\begin{equation}\label{eq:ballstab}
\begin{aligned}
\| \X_s \|_{\mathcal{Y}_1} &\leq M^{\rm in}_1(T) + C (M_1^2 M_0 + T^{1/8} M_1^{3/2} M_0^{3/2}) \quad (\leq 3M_1/4)\\
\| \X_s \|_{\mathcal{Y}_0} &\leq C M^{\rm in}_0 + C (M_1^2 M_0 + T^{1/8} M_1^{3/2} M_0^{3/2}) \quad (\leq 3M_0/4) \, ,
\end{aligned}
\end{equation}
provided we make the following choices:
First, choose $M_0 \geq 4 C M^{\rm in}_0$. Then, choose $M_1 \leq \min( (4CM_0)^{-1}, (4C)^{-1/2} )$. %small enough to satisfy $M_1 M_0 \leq 1/(4C)$ and $M_1^2 \leq 1/(4C)$.
Finally, choose $T \leq 1$ small enough to satisfy $C T^{1/8} M_1^{3/2} M_0^{3/2} \leq \min(M_0,M_1)/4$ and $M^{\rm in}_1(T) \leq M_1/4$.
%To ensure $C M_1^2 M_0 \leq M_1/4$, we require 
%$C M_1^2 M_0 \leq M_0/4$, we require 
%We require $M_0/4 \geq C M^{\rm in}_0$.

\emph{Contraction}. We begin with estimates on the difference between two tensions. For two curves $\Y^{(1)}$, $\Y^{(2)}$, we define 
\begin{equation}\label{eq:Rlambdak}
R_\lambda^{(k)} := |\Y^{(k)}_{sss}|^2 - (3\gamma+4) (\Y_{sss}^{(k)}\cdot\Y_{ss}^{(k)})_s \, , \quad k=1,2 \, ,
\end{equation}
and the tensions
\begin{equation}
\lambda^{(k)} = \mc{T}[\Y_{s}^{(k)}]^{-1} (R_\lambda^{(k)}) \, , \quad k=1,2 \,.
\end{equation}
The difference $\lambda^{(1)}-\lambda^{(2)}$ may then be written
\begin{equation}\label{eq:lam_strat}
\lambda^{(1)}-\lambda^{(2)} = [ \mc{T}[\Y^{(1)}_{s}]^{-1} - \mc{T}[\Y^{(2)}_{s}]^{-1} ] (R_\lambda^{(1)}) + \mc{T}[\Y^{(2)}_{s}]^{-1} (R_\lambda^{(1)}- R_\lambda^{(2)}) \, .
\end{equation}

First, we estimate
\begin{equation}\label{eq:qdef}
\mc{T}[\Y^{(1)}_{s}]^{-1} (R_\lambda^{(1)}) - \mc{T}[\Y^{(2)}_{s}]^{-1} (R_\lambda^{(1)}) =: q^{(1)} - q^{(2)} \, .
\end{equation}
Notably, $q^{(1)} = \lambda^{(1)}$. Then
\begin{equation}
(1+\gamma) \big(q^{(2)} - q^{(1)}\big)_{ss} - |\Y_{ss}^{(2)}|^2 q^{(2)} + |\Y_{ss}^{(1)}|^2 q^{(1)} = 0 \, ,
\end{equation}
which we expand as
\begin{equation}
(1+\gamma) \big(q^{(2)} - q^{(1)}\big)_{ss} - |\Y_{ss}^{(2)}|^2 (q^{(2)} - q^{(1)}) - (|\Y_{ss}^{(2)}|^2 - |\Y_{ss}^{(1)}|^2) q^{(1)}  = 0
\end{equation}
and estimate as
\begin{equation}
\| q^{(2)} - q^{(1)} \|_{H^1_s} \lesssim (\| \Y_s^{(1)} \|_{H^1_s} + \| \Y_s^{(2)} \|_{H^1_s} ) \| \Y_s^{(2)} - \Y_s^{(1)} \|_{H^1_s} \| q^{(1)} \|_{H^1_s} \, .
\end{equation}

Second, we estimate
\begin{equation}
\begin{aligned}
&\| R_\lambda^{(1)} - R_\lambda^{(2)} \|_{(H^1_s)^*} \\
&\quad \lesssim \big(\| \Y^{(1)}_{sss} \|_{L^2_s} + \| \Y^{(2)}_{sss} \|_{L^2_s}\big) \| \Y^{(1)}_{sss} - \Y^{(2)}_{sss} \|_{L^2_s} + \| \Y^{(1)}_{sss} - \Y^{(2)}_{sss} \|_{L^2_s} \| \Y^{(1)}_{ss} \|^{1/2}_{L^2_s} \| \Y^{(1)}_{sss} \|^{1/2}_{L^2_s} \\
&\quad\quad + \| \Y^{(1)}_{ss} - \Y^{(2)}_{ss} \|_{L^2_s}^{1/2} \| \Y^{(1)}_{sss} - \Y^{(2)}_{sss} \|_{L^2_s}^{1/2} \| \Y^{(2)}_{sss} \|_{L^2_s}\,.
\end{aligned}
\end{equation}

Altogether, the difference between tensions is estimated by
\begin{equation}
\begin{aligned}
&\| \lambda^{(1)} - \lambda^{(2)} \|_{L^2_t H^1_s} \lesssim M_0 M_1^2 \| \Y_s^{(1)} - \Y_s^{(2)} \|_{\mathcal{Y}_0} +  M_1 \| \Y^{(1)}_s - \Y^{(2)}_s \|_{\mathcal{Y}_1} \\
&\quad + \| \Y^{(1)}_s - \Y^{(2)}_s \|_{\mathcal{Y}_1} T^{1/8}  M_1^{1/2} M_0^{1/2}
 + \| \Y^{(1)}_s - \Y^{(2)}_s \|_{\mathcal{Y}_0}^{1/2} \| \Y^{(1)}_s - \Y^{(2)}_s \|_{\mathcal{Y}_1}^{1/2}T^{1/4} M_1^{3/2} \, .
\end{aligned}
\end{equation}

We now turn to estimating the difference 
\begin{equation}
\begin{aligned}
\X^{(1)}_s - \X^{(2)}_s =
\int_0^t e^{-(t-t')\p_s^4} \p_s R_{\rm t} \, dt' + (3\gamma+4)  \int_0^t e^{-(t-t')\p_s^4}\p_s R_3 \, dt' \, ,
\end{aligned}
\end{equation}
where
\begin{equation}
\begin{aligned}
R_{\rm t} &= \big( \lambda^{(1)}  \Y_s^{(1)} \big)_s - \big( \lambda^{(2)} \Y_s^{(2)}\big)_s + \gamma \lambda_s^{(1)} \Y_s^{(1)}   - \gamma \lambda^{(2)}_s \Y_s^{(2)}\,,\\
%[ (\lambda^{(1)} - \lambda^{(2)}) \Y_s^{(1)} ]_s + [\lambda^{(2)} (\Y_s^{(1)} - \Y_s^{(2)})]_s + \gamma (\lambda_s^{(1)} - \lambda^{(2)}_s) \Y_s^{(1)} + \gamma \lambda^{(2)}_s (\Y_s^{(1)} - \Y_s^{(2)}) \, ,\\
%
R_3 &= (\Y_{sss}^{(1)} \cdot \Y_{ss}^{(1)}) \Y_{s}^{(1)} - (\Y_{sss}^{(2)} \cdot \Y_{ss}^{(2)}) \Y_{s}^{(2)} \, .
\end{aligned}
\end{equation}
These remainder terms may be estimated as 
\begin{equation}
\begin{aligned}
\| R_{\rm t} \|_{L^2_t L^2_s} &\lesssim \| \lambda^{(1)} - \lambda^{(2)} \|_{L^2_t H^1_s} M_0 + M_1^2 \| \Y_s^{(1)} - \Y_s^{(2)} \|_{\mathcal{Y}_0} \,,\\
\| R_3 \|_{L^2_t L^2_s} &\lesssim  T^{1/8} M_0 M_1 \| \Y_{ss}^{(1)} - \Y_{ss}^{(2)} \|_{L^\infty_t L^2_s}^{1/2} \| \Y_{sss}^{(1)} - \Y_{sss}^{(2)} \|_{L^4_t L^2_s}^{1/2} \\
& + \| \Y_{sss}^{(1)} - \Y_{sss}^{(2)} \|_{L^4_t L^2_s} T^{1/8} M_0^{3/2} M_1^{1/2} + T^{1/8} M_1^{3/2}M_0^{1/2} \| \Y_s^{(1)} - \Y_s^{(2)} \|_{L^\infty_t H^1_s} \, .
\end{aligned}
\end{equation}
% That is,
% \begin{equation}
% \begin{aligned}
% &\| R_3 \|_{L^2_t L^2_s} \lesssim t^{-3/8} \| \Y_{s}^{(1)} - \Y_{s}^{(2)} \|_{\mathcal{Y}_1} M_0^{3/2} M_1^{1/2} \\
% &\quad + t^{-3/8} \| \Y_{s}^{(1)} - \Y_{s}^{(2)} \|_{\mathcal{Y}_1}^{1/2} \| \Y_{s}^{(1)} - \Y_{s}^{(2)} \|_{\mathcal{Y}_0}^{1/2} M_0 M_1 \\
% &\quad + t^{-3/8} M_1^{3/2}M_0^{1/2}\| \Y_s^{(1)} - \Y_s^{(2)} \|_{\mathcal{Y}_0}
% \end{aligned}
% \end{equation}

By Duhamel's formula, similar to the estimates \eqref{eq:ballstab}, we obtain
\begin{equation}
  \| \X^{(1)}_{s} - \X^{(2)}_{s} \|_{\mathcal{Y}_1 \cap \mathcal{Y}_0}  \lesssim C (1+T^{1/8}) (1 + C(M_0)) o_{M_1 \to 0^+}(1)  \| \Y^{(1)}_{s} - \Y^{(2)}_{s} \|_{\mathcal{Y}_1 \cap \mathcal{Y}_0} \, ,
\end{equation}
where $C(M_0)$ is an increasing function of $M_0$, and we recall $T \leq 1$. A key component of the above argument is that $M_0$ does not appear without $M_1$. Finally, we can decrease $M_1$ (depending on $M_0$, more specifically, like an inverse power of $M_0$) as necessary to ensure that the system is contractive. Notably, $T$ can be chosen to depend continuously on $\X_s^{\rm in} \in H^1_s$.

\textbf{2. Continuation criterion}. The above construction gives the continuation criterion $\X_s \in C([0,T];H^1_s)$. However, the above estimates on Duhamel's formula yield that any solution $\X_s \in L^\infty_t H^1_s \cap L^4_t H^2_s(I \times (0,T))$ belongs to $C([0,T];H^1_s)$. 

\textbf{3. Higher regularity}. We will demonstrate that, on a possibly shorter time interval $(0,T)$, the solution satisfies
\begin{equation}
    \label{eq:myhigherregestinproof}
\| \X_s \|_{L^\infty_t H^3_s \cap L^2_t H^5_s(I \times (0,T))} \lesssim \| e^{-t\p_s^4} \X_s^{\rm in} \|_{L^\infty_t H^3_s \cap L^2_t H^5_s(I \times (0,T))} \, .
\end{equation}
The method is to prove the estimate~\eqref{eq:myhigherregestinproof} for the Picard iterates and pass it to the solution $\X_s$ by lower semi-continuity.\footnote{\emph{A posteriori}, the estimates below with $\Y_s = \X_s$ will demonstrate that the solution belongs to $C([0,T];H^3_s)$.} The relevant estimates on Duhamel's formula~\eqref{eq:Duhamelforcurve} will be derived from the linear estimate~\eqref{eq:higherlinearestimate}. The contribution from the initial data is
\begin{equation}
    \label{eq:initialdataestimate}
\| e^{-t\p_s^4} \X_s^{\rm in} \|_{L^\infty_t H^3_s \cap L^2_t H^5_s(I \times (0,T))} \leq C_0 \| \X_s^{\rm in} \|_{H^3_s} \, .
\end{equation}
%\in L^2_t H^1_s
The ``forces" inside Duhamel's formula should be estimated in $L^2_t H^1_s$. For these particular estimates, we divide the terms into ``new information" $L^\infty_t H^3_s \cap L^2_t H^5_s$ (of which we are allowed only one copy) and ``old information" $L^\infty_t H^1_s \cap L^2_t H^3_s$ (which we collect in square brackets):\footnote{Alternatively, we could close a contraction argument in $L^\infty_t H^3_s \cap L^2_t H^5_s$ on a time of existence depending on $\| \X_s^{\rm in} \|_{H^3_s}$.}
\begin{equation}
\begin{aligned}
\| (\Y_{sss} \cdot \Y_{ss}) \Y_{s})_s \|_{L^2_t H^1_s} &\lesssim \| \Y_{ssss} \|_{L^4_t H^1_s} \big[ \| \Y_{ss} \|_{L^4_t H^1_s} \| \Y_s \|_{L^\infty H^1_s} \| \big] \\
&\quad 
+ \| \Y_{sss} \|_{L^\infty_t H^1_s} \big[ \| \Y_{sss} \|_{L^2_t H^1_s} \| \Y_{s} \|_{L^\infty_t H^1_s} + \| \Y_{ss} \|_{L^4_t H^1_s}^2 \big]\,.
\end{aligned}
\end{equation}
By the tension estimates~\eqref{eq:higherregtensionestimatecurve}, we have % the first term doesn't require all new information
\begin{equation}
\begin{aligned}
\| \lambda[\Y_s] \|_{L^2_t H^2_s} &\lesssim \| \Y_{sss} \|_{L^\infty_t L^2_s} \big[\| \Y_{ss} \|_{L^\infty_t L^2_s} \| \Y_{sss} \|_{L^4_t L^2_s}^2\big] + \| \Y_{ssss} \|_{L^4_t L^2_s} \big[\| \Y_{sss} \|_{L^4_t L^2_s}\big]\,,\\
\| \lambda[\Y_s] \|_{L^4_t H^1_s} &\lesssim \| \Y_{sss} \|_{L^\infty_t L^2_s} \big[\| \Y_{sss} \|_{L^4_t L^2_s}\big] \, ,
\end{aligned}
\end{equation}
which yield
\begin{equation}
\| \lambda[\Y_s] \Y_s \|_{L^2_t H^2_s} \leq \| \Y_{sss} \|_{L^\infty_t L^2_s \cap L^2_t H^2_s(I \times (0,T))} C(M_0) o_{M_1 \to 0^+}(1) \, .
\end{equation}
In summary, we have
\begin{equation}
\| \X_s \|_{L^\infty_t H^3_s \cap L^2_t H^5_s(I \times (0,T))} \leq C_0 \| \X_s^{\rm in} \|_{H^3} + C(M_0) o_{M_1 \to 0^+}(1) \| \Y_{s} \|_{L^\infty_t L^2_s \cap L^2_t H^2_s(I \times (0,T))} \, .
\end{equation}
If we choose $T \ll 1$ to ensure that $C(M_0) o_{M_1 \to 0^+}(1) \leq 1/2$, then the estimate~\eqref{eq:myhigherregestinproof} with the constant $2 C_0$ will be propagated by the Picard iterates. Notably, $T$ can still be chosen to depend continuously on $\X_s^{\rm in} \in H^1_s$.

% To prove higher regularity, it is first convenient to get the solution out of the ``critical" space $C([0,T];H^1) \cap L^2_t H^3_s$ and into a subcritical space, for example, $C([0,T];H^2) \cap L^2_t H^4_s$. (Once the solution belongs to a sub-critical space, it can be considered a ``linear" equation, and the curve and tension can be bootstrapped separately -- see [Koiso].) These estimates can be propagated through the fixed point argument (it only matters ball-to-ball, then use weak-star compactness). Importantly, the necessary smallness only depends on $\mathcal{Y}_1$, not on $\mathcal{Y}_2$.

\textbf{4. Unit length}. \emph{A posteriori}, we may justify that the condition $|\X_s| = 1$ is propagated by solutions to the equation. Suppose that $\X_s^{\rm in} \in H^3_s$ with $|\X_s^{\rm in}|^2 = 1$. We multiply~\eqref{eq:section3xsequation} by $\X_s$ and simplify \emph{without} the condition $|\X_s|^2 = 1$:
\begin{equation}
\begin{aligned}
\frac{1}{2} (\p_t + \p_s^4) |\X_s|^2
&= %4 \X_{ssss} \cdot \X_{ss} + 3 |\X_{sss}|^2 +
(1+\gamma) \lambda_{ss} (|\X_s|^2-1) + \frac{1}{2} (1+\gamma) \lambda_s \p_s |\X_s|^2 + \frac{1}{2} \lambda\p_s^2 |\X_s|^2  \\
%- |\X_{ss}|^2 \lambda
%+ \gamma \lambda_{ss} (|\X_s|^2-1) + \gamma \lambda_s \frac{1}{2} \p_s |\X_s|^2
&\quad + 3 \gamma (\X_{sss} \cdot \X_{ss})_s (|\X_s|^2-1) + 3 \gamma (\X_{sss} \cdot \X_{ss}) \frac{1}{2} \p_s |\X_s|^2\,,
\end{aligned}
\end{equation}
where we have already applied the tension equation~\eqref{eq:tensionequationcurveforlemma}. This may be viewed as a parabolic equation for the quantity $|\X_s|^2 - 1$ with initial data $|\X_s^{\rm in}|^2 - 1=0$, whose unique solution is zero.\footnote{One may verify that the coefficients are smooth enough that the $L^2$ energy estimates make sense.}
Finally, when $\X_s^{\rm in} \in H^1_s$ with $|\X_s^{\rm in}| = 1$, we use an approximation argument.\footnote{For example, mollify $\X$ and rewrite in arclength parameterization.}
\end{proof}

% \da{What about for the rods? Do the same thing with the hyperviscous tension, I suppose.}

\subsection{Energetics and global well-posedness}\label{subsec:energetics}

In this section, we use the energy identity \eqref{eq:energy} satisfied by the curve evolution to prove global well-posedness. Defining 
\begin{equation}
E(t) = \frac{1}{2} \int_I |\X_{ss}|^2 \, ds\,,
\end{equation}
we recall from~\eqref{eq:energy} that $E$ satisfies 
\begin{equation}
\frac{d}{dt} E(t) = - \int_I |(\X_{sss} - \lambda \X_s)_s|^2 - \gamma \int_I |(\X_{sss} - \lambda \X_s)_s \cdot \X_s|^2 \, ds \, .
\end{equation}
The right-hand side energy dissipation appears in terms of a quantity 
\begin{equation}
 \Z := \X_{sss} - \lambda \X_s \, .
 \end{equation}
In particular, defining
\begin{equation}
D(t) = \int_I |\Z_s|^2 \, ds\,,
\end{equation}
we note that 
\begin{equation}
\frac{d}{dt} E(t) \leq - D(t) \, ,
\end{equation}
 and a key point will be to extract control on $\X_{ssss}$ from
 \begin{equation}
 \Z_s = \X_{ssss} - (\lambda \X_s)_s\,.
 \end{equation}

\begin{proposition}\label{prop:Xssss_ests_curve}
In the free end setting, $\X_{ssss}$ may be controlled by the energy $E(t)$ and dissipation $D(t)$ as
\begin{equation}\label{eq:Xssss_free0}
\| \X_{ssss} \|_{L^2_s} \lesssim D(t)^{1/2} (1+E(t)^{1/2}) + E(t)^{5/2} \, .
\end{equation}
In the closed loop setting, $\X_{ssss}$ may be controlled by $E(t)$ and $D(t)$ as
\begin{equation}\label{eq:Xssss_closed0}
\| \X_{ssss} \|_{L^2_s} \lesssim D^{1/2}(t) (1 + E^{1/2}(t)) + E(t) + E(t)^{5/2} \, .
\end{equation}
\end{proposition}

\begin{proof}
One useful observation (cf.~\cite[p. 406]{koiso1996motion}) is that it is often geometrically advantageous to work with the quantity $\Z$ instead of $\Z_s$.
We begin by noting the following tangential and non-tangential projections 
\begin{equation}
\begin{aligned}
P_{\X_s} \Z &= - (|\X_{ss}|^2 + \lambda) \X_s\,,\\
P_{\X_s}^\perp \Z &= \X_{sss} + |\X_{ss}|^2 \X_s\,,
\end{aligned}
\end{equation}
as well as 
\begin{equation}
	\label{eq:Zstan}
\begin{aligned}
P_{\X_s} \Z_s &= (\X_{ssss} \cdot \X_{s}) \X_s - \lambda_s \X_s 
= - \frac{3}{2} \big(|\X_{ss}|^2\big)_s \X_s -\lambda_s \X_s \, , \\
P_{\X_s}^\perp \Z_s &= \X_{ssss} + \frac{3}{2} \big(|\X_{ss}|^2\big)_s \X_s\,.
\end{aligned}
\end{equation}
We now consider the free end and closed loop cases separately.

\emph{Free ends}. To begin, we recall the boundary conditions $\X_{ss}\big|_{s=0,1} = \X_{sss}\big|_{s=0,1} = 0$. In particular, $\X_{ss} \in H^2_0$, so various homogeneous interpolation inequalities are valid, for example,
\begin{equation}
	\label{eq:interp2}
\| \X_{sss} \|_{L^2_s} \lesssim \| \X_{ss} \|_{L^2_s}^{1/2} \| \X_{ssss} \|_{L^2_s}^{1/2} \, ,
\end{equation}
\begin{equation}
	\label{eq:interp3}
\| \X_{ss} \|_{L^\infty_s} \lesssim \| \X_{ss} \|_{L^2_s}^{1/2} \| \X_{sss} \|_{L^2_s}^{1/2} \, .
\end{equation}

We follow the observation of~\cite{oelz2011curve}. Since $\Z$ vanishes at the filament endpoints, we may write
\begin{equation}
\Z(s) = \int_0^s \Z_{s'}(s') \, ds' \, ,
\end{equation}
and we obtain
\begin{equation}
\| \Z \|_{H^1_s} \lesssim D(t)^{1/2} \, .
\end{equation}
We estimate the tangential component of $\Z$ as
\begin{equation}
\| \lambda + |\X_{ss}|^2 \|_{H^1_s} = \| \Z \cdot \X_s \|_{H^1_s} \lesssim D(t) \| \X_s \|_{H^1_s} \lesssim D(t)^{1/2} (1+E(t)^{1/2}) \, ,
\end{equation}
so that, by the triangle inequality, the normal component satisfies
\begin{equation}
	\label{eq:normalcomponentpperp}
\| P^\perp_{\X_s} \X_{sss} \|_{H^1_s} \lesssim D(t)^{1/2} (1+E(t)^{1/2}) \, .
\end{equation}
To estimate the full $H^1_s$ norm of $\X_{sss}$, it remains to estimate the tangential contribution $P_{\X_s} \X_{sss} = - |\X_{ss}|^2 \X_s$ in $H^1_s$. First,
\begin{equation}
\| |\X_{ss}|^2 \X_s \|_{L^2_s} \leq \| \X_{ss} \|_{L^4_s}^2 \lesssim \| \X_{ss} \|_{L^2_s}^{7/8} \| \X_{ssss} \|_{L^2_s}^{1/8} \, ,
\end{equation}
where we use interpolation ($L^4$ is ``halfway" between $L^2$ and $L^\infty$, and then we apply~\eqref{eq:interp3} and~\eqref{eq:interp2}). Second, we may write
\begin{equation}
(|\X_{ss}|^2 \X_s)_s = 2 (\X_{sss} \cdot \X_{ss}) \X_s + \underbrace{|\X_{ss}|^2 \X_{ss}}_{= - (\X_{sss} \cdot \X_s) \X_{ss}} \, ,
\end{equation}
and therefore
\begin{equation}
\|(|\X_{ss}|^2 \X_s)_s \|_{L^2_s} \lesssim \| \X_{ss} \|_{L^2_s}^{1/2} \| \X_{sss} \|_{L^2_s}^{3/2} \lesssim \| \X_{ss} \|_{L^2_s}^{5/4} \| \X_{ssss} \|_{L^2_s}^{3/4} \, .
\end{equation}
Thus,
\begin{equation}
\begin{aligned}
\| \X_{ssss} \|_{L^2_s} &\leq \| (P_{\X_s}^\perp \X_{sss})_s \|_{L^2_s} + \| (|\X_{ss}|^2 \X_s)_s \|_{L^2_s} \\
&\leq \| (P_{\X_s}^\perp \X_{sss})_s \|_{L^2_s} + C \| \X_{ss} \|_{L^2_s}^{5/4} \| \X_{ssss} \|_{L^2_s}^{3/4} \, .
\end{aligned}
\end{equation}
Finally, we apply~\eqref{eq:normalcomponentpperp} to the first term and split the second term via Young's inequality to obtain~\eqref{eq:Xssss_free0}. This controls $\X_{ssss}$ in $L^2_t L^2_s$.

\emph{Closed loop}. We must now be careful with the mean of $\Z$, cf.~\cite[p. 428]{koiso1996motion}:
\begin{equation}
\int_\T \Z \, ds = \int_\T (\X_{sss} - \lambda \X_{s}) \, ds = \int_\T \lambda_s \big(\X-\X(0)\big) \, ds\,.
\end{equation}
Using~\eqref{eq:Zstan}, we may write
\begin{equation}
\int_\T \Z \, ds = - \int_\T \bigg(\Z_s\cdot\X_s + \frac{3}{2} \big(|\X_{ss}|^2\big)_s\bigg) \big(\X - \X(0)\big) \, ds \leq D(t)^{1/2} + E(t) \,,
\end{equation}
and hence
\begin{equation}
\| \Z \|_{H^1_s} \lesssim D(t)^{1/2} + E(t) \, .
\end{equation}
We can now continue with the same arguments as above. All the interpolation inequalities are valid because $\X_{ss}$ has zero mean, and we obtain the conclusion~\eqref{eq:Xssss_closed0}. \end{proof}

In conclusion, $\X_s$ belongs to $L^\infty_t H^1_s \cap L^2_t H^3_s$ on its maximal time of existence. Together with the continuation criterion, this proves global well-posedness.

%%%%%%%%%%%%%%%%%%%%%%%%%%%%%%%%%%%%%%%%%%%%%%%%%%%%%%%%%%%%
%%%%%%%%%%%%%%%%%%%%%%%%%%%%%%%%%%%%%%%%%%%%%%%%%%%%%%%%%%%%
%%%%%%%%%%%%%%%%%%%%%%%%%%%%%%%%%%%%%%%%%%%%%%%%%%%%%%%%%%%%

\subsection{Propagation estimates and decay to equilibrium}
\label{sec:propagationestimates}

\subsubsection{Poincar{\'e} inequality for free ends}
% Let
% \begin{equation}
% D_1(t) := \int_0^1 |(\X_{sss} - \lambda \X_s - \eta \kappa_3 \X_s \times \X_{ss})_s|^2 \, ds
% \end{equation}
% be the energy dissipation in the free end setting ($\eta = 0$ without the frame). Recall
% \begin{equation}
% \| \X_{sss} - \lambda \X_s - \eta \kappa_3 \X_s \times \X_{ss} \|_{H^1_s}^2 \lesssim D(t) \, .
% \end{equation}
In the free end setting, the dissipation $D(t)$ directly controls the energy through the following Poincar{\'e}-type inequality~\cite{oelz2011curve}:
\begin{equation}
    \label{eq:Poincareinequalitystep}
\begin{aligned}
\int_0^1 |\X_{ss}|^2\,ds &\overset{\eqref{eq:muckinaboutwiththecurvature}}{\leq} \frac{1}{\pi^2} \int_{ \{ |\X_{ss}|^2 > 0 \} } |P_{\e_{\rm n}} \X_{sss}|^2 \, ds \\
&\leq \frac{1}{\pi^2} \int_0^1 |\X_{sss} - \lambda \X_s|^2 \,ds \leq \frac{1}{\pi^4} D(t) \, ,
\end{aligned}
\end{equation}
obtained by applying the $H^1_0(0,1)$ Sobolev embedding twice.

The term $|P_{\e_{\rm n}} \X_{sss}|^2$ in~\eqref{eq:Poincareinequality} arises in the following way: Let $\kappa^2 := |\X_{ss}|^2$ and examine the (at most countably many) open intervals $I_j$ on which $\kappa^2 > 0$. Since $\kappa^2$ vanishes to at least second order at $\p I_j$, $\kappa := \sqrt{|\X_{ss}|^2} \in C^1(\overline{I_j})$ and is Lipschitz on $[0,1]$. Then
\begin{equation}
	\label{eq:muckinaboutwiththecurvature}
\int_0^1 \kappa^2 \, ds \leq \frac{1}{\pi^2} \int_0^1 |\kappa_s|^2 \, ds = \frac{1}{\pi^2} \sum_j \int_{I_j} |\X_{sss} \cdot \e_{\rm n}|^2 \, ds
\end{equation}
where $\e_{\rm n} := \X_{ss}/\kappa$ is well defined on each interval $I_j$.\footnote{The less involved computation
\begin{equation}
\int_0^1 |\X_{ss}|^2\,ds \leq \int_0^1 |\p_s |\X_{ss}|^2|\,ds \leq 2 \int_0^1 |\X_{sss} \cdot \X_{ss}|\,ds
\end{equation}
achieves a worse constant.}

% Old version:
% \begin{equation}
%     \label{eq:Poincareinequality}
% \begin{aligned}
% \int_0^1 |\X_{ss}|^2\,ds &\leq \int_0^1 |\p_s |\X_{ss}|^2|\,ds \leq 2 \int_0^1 |\X_{sss} \cdot \X_{ss}|\,ds \\
% &\leq 2 \int_0^1 |(\X_{sss} - \lambda \X_s) \cdot \X_{ss}|\,ds\\
% & \leq 2 \sqrt{D(t)} \| \X_{ss} \|_{L^2_s} \, ,
% \end{aligned}
% \end{equation}
In summary,
\begin{equation}
\label{eq:Poincareinequality}
E \leq \frac{1}{2\pi^4} D(t) \, .
\end{equation}

Let $c_0^{-1}$ be the best constant in the inequality~\eqref{eq:Poincareinequality}; in particular, $c_0 \geq 2\pi^4$.\footnote{As remarked in~\cite{oelz2011curve}, $\pi^4\approx 97$ is not the bottom eigenvalue of $-\p_s^4$ with clamped boundary conditions (the linearized operator around a straight rod), which is approximately $500$. However, $\pi^4$ \emph{is} the bottom eigenvalue of $-\p_s^4$ with boundary conditions $u,\,\p_s^2 u|_{s=0,1} = 0$.} Combining this with the energy balance, we obtain
\begin{equation}\label{eq:dtE_noforce}
\frac{dE}{dt} \leq -D(t) \leq - c_0 E \, .
\end{equation}
Therefore, Gr{\"o}nwall's inequality implies
\begin{equation}
    \label{eq:exponentialdecayfreeend}
E(t) \leq e^{-c_0 t} E^{\rm in} \, .
\end{equation}

%(In the rod setting, we can now plug this into the inequality which controls $\int_0^t \int |\Z_s|^2 \, ds \, dt'$.)

%Again, the above reasoning was in the free end setting only.

Additionally, we may obtain a time-integrated decay estimate on $D$. For $\theta\in(0,1)$, the bound \eqref{eq:dtE_noforce} may be rewritten as   
% \begin{equation}
% \frac{dE}{dt} \leq - D \, .
% \end{equation}
\begin{equation}
\frac{dE}{dt} + \theta D \leq -(1-\theta) D \leq - c_0 (1-\theta) D \, ,
\end{equation}
which, using an integrating factor, may be further rewritten as 
\begin{equation}
\frac{d}{dt} (Ee^{c_0 (1-\theta) t}) + \theta D e^{c_0 (1-\theta) t} \leq 0 \, .
\end{equation}
Integrating in time yields the estimate%the time-integrated decay estimate
\begin{equation}
\label{eq:E_D_free}
\theta \int_0^{\infty} e^{c_0 (1-\theta) t} D(t) \, dt \leq E^{\rm in} \, .
\end{equation}

\subsubsection{Propagation estimates}

In the free end setting, we already have the exponential decay estimate~\eqref{eq:exponentialdecayfreeend}, whereas in the periodic setting, we only have \emph{a priori} bounds on the energy and a time-integrated decay of the energy dissipation:
\begin{equation}\label{eq:E_D_closed}
\sup_t E(t) \leq E^{\rm in} \, , \quad \int_0^{\infty} D(t) \, dt \leq E^{\rm in} \, .
\end{equation}
In both the free-end and periodic settings, our goal is to prove pointwise-in-time bounds on the energy dissipation $D(t)$. %, which is a strong indication that the solution is settling to a steady state.
In the periodic setting, these will be upgraded to a full proof of convergence to Euler elasticae using the Lojasiewicz inequality from Section~\ref{sec:rod_loja}. %).a further tool from the theory of gradient flows, 
\begin{proposition}
	\label{pro:differentialineq}
In the closed loop setting,
\begin{equation}
\frac{1}{2} \frac{d}{dt} D(t) \leq - \frac{1}{2} \int_\T |\p_t \X_{ss}|^2\,ds + C(E) D(t) \times D(t) + C(E) D(t) \, .
\end{equation}
In the free end setting,
\begin{equation}
\frac{1}{2} \frac{d}{dt} D(t) \leq - \frac{1}{2} \int_0^1 | \p_t \X_{ss} |^2\,ds + e^{-c_0t} \big(C(E) D(t)^2 + C(E) D(t)\big) + C D(t) + C D(t)^2 \, .
\end{equation}
\end{proposition}

\begin{proof}
We begin with some calculations, valid for both the free and periodic cases. First, we have
\begin{equation}\label{eq:dtDt_eqn}
\begin{aligned}
\frac{1}{2} \frac{d}{dt} D(t) &= \frac{1}{2} \frac{d}{dt} \int_I |\Z_s|^2 + \gamma (\X_s \cdot \Z_s)^2 \, ds \\
&= \int_I [(\bI + \gamma \X_s \otimes \X_s) \Z_s] \cdot \p_t \Z_s \, ds + \int_I \gamma (\X_s \cdot \Z_s) (\p_t \X_s \cdot \Z_s) \, ds \, .
\end{aligned}
\end{equation}
Since $\p_t \Z|_{s=0,1}=0$ in the free end setting, we may write 
\begin{equation}
\int_I [\underbrace{(\bI + \gamma \X_s \otimes \X_s) \Z_s}_{- \p_t \X}] \cdot \p_t \Z_s \, ds = \int_I \p_t \X_s \cdot \p_t \Z = \int_I \p_t \X_s \cdot \p_t ( \X_{sss} - \lambda \X_s ) \, ds \, .
\end{equation}
The first term produces an advantageous dissipation-type term:
\begin{equation}
\int_I \p_t \X_s \cdot \p_t \X_{sss} \, ds = - \int_I |\p_t \X_{ss}|^2 \, ds \, .
\end{equation}
The second term becomes
\begin{equation}
- \int_I \lambda |\p_t \X_s|^2 \, ds \, ,
\end{equation}
as the term containing $\p_t \lambda$ vanishes due to $\p_t \X_s \cdot \X_s = 0$. Since $\lambda$ may not have an advantageous sign, it must be estimated as
\begin{equation}\label{eq:tensionterm}
\left| \int_I \lambda |\p_t \X_s|^2 \, ds \right| \lesssim \| \lambda \|_{L^\infty_s} \| \p_t \X_s \|_{L^2_s}^2\,. 
\end{equation} 
The remaining term in \eqref{eq:dtDt_eqn} satisfies
\begin{equation}
	\label{eq:gammaterm}
\left| \int_I \gamma (\X_s \cdot \Z_s) (\p_t \X_s \cdot \Z_s) \, ds \right| \lesssim \| \p_t \X_s \|_{L^\infty_s} \| \Z_s \|_{L^2_s}^2 \, .
\end{equation}
We proceed with estimating \eqref{eq:tensionterm} and \eqref{eq:gammaterm} by considering the periodic and free-end cases separately.

\textit{Closed loop}. 
The tension term \eqref{eq:tensionterm} may be estimated by using that the tension satisfies
\begin{equation}
\| \lambda \|_{H^1_s} \lesssim \| \X_{sss} \|_{L^2_s}^2 \lesssim \| \X_{ss} \|_{L^2_s} \| \X_{ssss} \|_{L^2_s} \, .
\end{equation}
We apply an interpolation inequality
\begin{equation}\label{eq:ptXs_interp0}
\| \p_t \X_s \|_{L^2_s}^2 \lesssim \| \p_t \X \|_{L^2_s} \| \p_t \X_{ss} \|_{L^2_s} \, ,
\end{equation}
since $\p_t \X_s$ has zero mean. Eventually, the term $\| \p_t \X_{ss} \|_{L^2_s}$ will be ``broken off" via Young's inequality, so it remains to estimate
\begin{equation}
\| \X_{ss} \|_{L^2_s}^2 \| \X_{ssss} \|_{L^2_s}^2 \| \Z_s \|_{L^2_s}^2 \, .
\end{equation}
Using the estimates of Proposition \ref{prop:Xssss_ests_curve} on $\X_{ssss}$, the above is estimated by
\begin{equation}\label{eq:Es_and_Ds}
C(E) D(t) \| \Z_s \|_{L^2_s}^2 + C(E) \| \Z_s \|_{L^2_s}^2 \leq C(E) D(t) \| \Z_s \|_{L^2_s}^2 + C(E) D(t) \, .
\end{equation}
We will use a Gr{\"o}nwall argument on the quantity $\| \Z_s \|_{L^2_s}^2$, which is equivalent to $D(t)$. We will crucially exploit that $D(t)$ is already known to be time-integrable to control the coefficients in the Gr{\"o}nwall argument and the ``forcing" $C(E) D(t)$.

We now estimate the term \eqref{eq:gammaterm} involving $\gamma$. Since $\| \p_t \X_s \|_{L^\infty_s} \lesssim \| \p_t \X_{ss} \|_{L^2_s}$, we can break off this term, and it remains to estimate $\| \Z_s \|_{L^2_s}^4$, namely, by $D(t)^2$.

Summarizing, we have
\begin{equation}
\frac{1}{2} \frac{d}{dt} D(t) \leq - \frac{1}{2} \int_\T |\p_t \X_{ss}|^2\,ds + C(E) D(t) \times D(t) + C(E) D(t) \, ,
\end{equation}
where the $1/2$ on the right-hand side comes from absorbing terms.

%%%%%%%%
\textit{Free ends}. 
In the free end setting, we may estimate the tension term \eqref{eq:tensionterm} by using that
\begin{equation}
\| \lambda \|_{H^1_s} \lesssim \| \X_{ss} \|_{L^2_s} \| \X_{ssss} \|_{L^2_s} \, ,
\end{equation}
due to the boundary conditions. Additionally, we have
\begin{equation}
\| \p_t \X_s \|_{L^2_s}^2 \lesssim \| \p_t \X \|_{L^2_s}^2 + \| \p_t \X \|_{L^2_s} \| \p_t \X_{ss} \|_{L^2_s} \lesssim \| \Z_s \|_{L^2_s}^2 + \| \Z_s \|_{L^2_s} \| \p_t \X_{ss} \|_{L^2_s} \, .
\end{equation}
After breaking off the $\p_t \X_{ss}$ term, it remains to estimate
\begin{equation}
\| \X_{ss} \|_{L^2_s}^2 \| \X_{ssss} \|_{L^2_s}^2 \| \Z_s \|_{L^2_s}^2 \, ,
\end{equation}
as before. However, we already know that $E(t)$ is exponentially decaying. Therefore, this term is estimated by
\begin{equation}
e^{-c_0t} (C(E) D(t)^2 + C(E) D(t)) \, .
\end{equation}

Next, we estimate the term \eqref{eq:gammaterm} involving $\gamma$. For this, we use
\begin{equation}
\| \p_t \X_s \|_{L^\infty_s} \lesssim \| \Z_s \|_{L^2_s}^{1/4} \| \p_t\X \|_{H^2_s}^{3/4} \lesssim \| \Z_s \|_{L^2_s} + \| \Z_s \|_{L^2_s}^{1/4} \| \p_t \X_{ss} \|_{L^2_s}^{3/4} \, .
\end{equation}
Then
\begin{equation}\label{eq:free_est1}
\| \p_t \X_s \|_{L^\infty_s} \| \Z_s \|_{L^2_s}^2 \lesssim \| \Z_s \|_{L^2_s}^3 + \| \Z_s \|_{L^2_s}^{9/4} \| \p_t\X_{ss} \|_{L^2_s}^{3/4} \, .
\end{equation}
We break off the $\p_t \X_{ss}$ term to obtain
\begin{equation}\label{eq:free_est2}
\| \Z_s \|_{L^2_s}^3 + \| \Z_s \|_{L^2_s}^{18/5} \lesssim D(t) + D(t)^2 \, .
\end{equation}

Finally, we have
\begin{equation}
\frac{1}{2} \frac{d}{dt} D(t) \leq - \frac{1}{2} \int_0^1 | \p_t \X_{ss} |^2\,ds + e^{-c_0t} \big(C(E) D(t)^2 + C(E) D(t)\big) + D(t) + D(t)^2 \, .
\end{equation}
\end{proof}

\subsubsection{Decay to equilibrium}
    \label{sec:decayimmersednoforcing}

We now leverage the differential inequalities in Proposition~\ref{pro:differentialineq} to prove convergence to equilibrum. First, we establish the following behavior of $E$ and $D$: %First, we prove that the energy stabilizes and the energy dissipation converges to zero:  and $D(t) \to 0$ as $t \to +\infty$.

\begin{proposition}
In the closed loop setting,
\begin{equation}
E(t) \leq E^{\rm in} \, , \quad E(t) \searrow \bar{E} =: \liminf_{t' \to +\infty} E(t') \text{ as } t \to +\infty
\end{equation}
\begin{equation}
	\label{eq:Dconvergestozero}
D(t) \to 0 \text{ as } t \to +\infty \, .
\end{equation}

In the free end setting, we have the refined convergence
\begin{equation}
E(t) \leq e^{-c_0t} E^{\rm in}
\end{equation}
\begin{equation}
D(t) \leq e^{C(E^{\rm in},\theta)} e^{-\theta c_0 t} [D^{\rm in} + C(E^{\rm in},\theta)] \, , \quad \forall \theta \in [0,1) \, .
\end{equation}
\end{proposition}

\begin{proof}
\emph{Closed loop}.
The inequalities at our disposal are
\begin{equation}
\frac{dE}{dt} = - D \leq 0
\end{equation}
\begin{equation}
\frac{dD}{dt} \leq AD + B
\end{equation}
where $A(t) = B(t) = C(E) D(t)$ and $\int_0^{+\infty} A(t) \, dt \leq C(E^{\rm in})$ (the same inequality holds for $B$). Gr{\"o}nwall's inequality yields
\begin{equation}
D(t) \leq e^{\int_{t_0}^t A(t') \, dt'} D(t_0) + \int_{t_0}^t e^{\int_{t'}^t A(\tau) \, d\tau} B(t') \, dt' \, , \quad \forall t \geq t_0 \geq 0 \, .
\end{equation}
In particular,
\begin{equation}
	\label{eq:exponentialupperbdsonD}
D(t) \leq e^{C(E^{\rm in})} D(t_0) + e^{C(E^{\rm in})} C(E^{\rm in}) \int_{t_0}^t D(t') \, dt' \, .
\end{equation}
Since $D(t) \in L^1_t(\R_+)$, we can choose a sequence of initial times $t_{0,k} \to +\infty$ such that $D(t_{0,k}) \to 0$, thereby demonstrating~\eqref{eq:Dconvergestozero}.

Notably, if we choose $t_0 = 0$ in~\eqref{eq:exponentialupperbdsonD}, then we obtain upper bounds on $D$ which are exponential in $E^{\rm in}$ (more specifically, of the type $e^{C \times (E^{\rm in})^p}$ for some $p \geq 1$).
% \begin{equation}
% \frac{1}{2} \frac{d}{dt} D(t) \leq - \frac{1}{2} \int_\T |\p_t \X_{ss}|^2\,ds + C(E) D(t) \times D(t) + C(E) D(t) \, .
% \end{equation}

\emph{Free end}. In this case, we have the refined inequalities
\begin{equation}
\frac{dE}{dt} = - D \leq - c_0 E
\end{equation}
\begin{equation}
\frac{dD}{dt} \leq -c_0D + e^{-c_0 t} C(E^{\rm in}) (D^2 + D) + C (D^2 + D) + c_0 D \, ,
\end{equation}
where have added and subtracted $c_0 D$. We already have integrated exponential decay~\eqref{eq:E_D_free} for $D$, which we will upgrade to pointwise exponential decay. For short times, $D$ is controlled by the exponential-in-$E^{\rm in}$ bound~\eqref{eq:exponentialupperbdsonD} with $t_0=0$. However, by the integral estimates~\eqref{eq:E_D_free} on $D$, we have the following: For any $\theta > 0$, there exists $T = T(E^{\rm in},\theta) >0$ such that 
\begin{equation}
\int_{T}^{+\infty} (e^{-c_0 t} C(E^{\rm in}) + C) D(t) \, dt \leq (1-\theta) c_0 \, .
\end{equation}
Then the differential inequality becomes
\begin{equation}
\frac{dD}{dt} \leq - \theta c_0 D + B  \qquad \forall t \geq T \, ,
\end{equation}
where $B = e^{-c_0t} C(E^{\rm in}) D + (C+c_0) D$. Then Gr{\"o}nwall's inequality will yield exponential convergence with rate $e^{-\theta c_0 t}$, since the contribution of the $B$ term in Duhamel's formula is estimated by
\begin{equation}
\int_{T}^t e^{-\theta c_0(t-t')} C(E^{\rm in}) e^{-\theta' c_0 t'} \underbrace{D(t') e^{\theta' c_0 t'}}_{\in L^1_{t'}(\R_+)} \, dt'  \qquad \forall t \geq T \, ,
\end{equation}
where we choose $\theta' \in (\theta,1)$. \end{proof}
% \begin{equation}
% \frac{1}{2} \frac{d}{dt} D(t) \leq - \frac{1}{2} \int_0^1 | \p_t \X_{ss} |^2\,ds + e^{-t/2} \big(C(E) D(t)^2 + C(E) D(t)\big) + D(t) + D(t)^2 \, .
% \end{equation}

% Then we use Gr{\"o}nwall's inequality and the fact that $D(t)$ is decaying to obtain\footnote{For example, if you already know $x$ is decaying and you have the inequality $\dot x \leq f$, then you can write $\dot x \leq -c x + cx + f$.}
% \begin{equation}
% D(t) \leq e^{-c_0t} [C(E^{\rm in}) D^{\rm in} + C(E^{\rm in})]
% \end{equation}
% for all $c_0 < \frac{1}{2}$. 

Finally, we conclude the convergence to equilibrium in Theorem~\ref{thm:GWP_classical}.

\begin{proposition}
If $\X$ is a finite-energy solution to~\eqref{eq:classical}, then $\X$ converges to an Euler elastica $\bar{\X}$ in $H^4_s$ as $t \to +\infty$.
\end{proposition}
\begin{proof}
\emph{Free end}. We have already demonstrated that $\X_{ss}$ is converging exponentially to zero in $H^2_s$. We should moreover demonstrate that $\X$ is stabilizing as $t \to +\infty$. We integrate $\p_t \X$ from $t=t_0$ to $t_1$ and estimate in $L^2_s$:
\begin{equation}
    \label{eq:followmeforstabilizing}
\| \X(\cdot,t_0) - \X(\cdot,t_1) \|_{L^2_s} \lesssim \int_{t_0}^{t_1} D^{1/2}(t) \, dt \, , \quad \forall t_1 \geq t_0 \geq 0 \, .
\end{equation}
Since $D(t)$ has integrated exponential decay~\eqref{eq:E_D_free}, we 
% \begin{equation}
% \frac{d}{dt} \int_0^1 \X(s,t)\phi(s) \, ds = -\int_0^1 (\mathbf{I} + \gamma \X_s \otimes \X_s) \Z_s\phi(s) \, ds \leq D(t)^{1/2} \| \phi \|_{L^2_s} \, ,
% \end{equation}
obtain that $\X$ is stabilizing exponentially in $L^2_s$ as $t \to +\infty$. 
% \begin{equation}
% \| \X(\cdot,t) \|_{L^2} = \sup_{\| \phi \|_{L^2} \leq 1} \int \X(s,t)\phi(s) \,ds
% \end{equation}
Hence, $\X$ is converging in $H^4_s$ as $t \to +\infty$.

\emph{Closed loop}. In the periodic setting, we require the Lojasiewicz inequality in Proposition~\ref{pro:curvelojascewicz}: For every equilibrium $\bar{\X}$, there exists $\delta > 0$, $\beta \in (0,1/2)$, and $A > 0$ such that if $\X \in H^4_s$ is an inextensible curve satisfying $\| \X - \bar{\X} \|_{H^4_s} \leq \delta$, then
\begin{equation}
    \label{eq:curvelojarestate}
(E[\X] - E[\bar{\X}])^{1-\beta} \leq A \sqrt{D[\X]} \, .
\end{equation}
The application of Lojasiewicz inequalities to geometric flows has a rich history starting with~\cite{SimonAnnals}. Our presentation is inspired by~\cite[Chapter 2]{kohout2021a}. 

%and an analogous argument to that for the immersed rod in Section~\ref{subsec:GWP_rodframe}.

By the energy and propagation estimates,
\begin{equation}
E(t) \leq E^{\rm in} \, , \quad \lim_{t \to +\infty} D(t) = 0 \, ,
\end{equation}
which control $\X_{ss}$ but not the center of mass $\int_\T \X(s,t) \,ds$. Define
\begin{equation}
\wt{\X}(s,t) = \X(s,t) - \int_\T \X(s,t) \, ds
\end{equation}
and observe that $E[\wt{\X}] = E[\X]$ and $D[\wt{\X}] = D[\X]$.

%The solution satisfies\dacomment{This is basically by propagation estimates.}
% \begin{equation}
% \sup_{t \in [0,+\infty)} D(t) < +\infty \, ,
% \end{equation}
% (actually, by the propagation estimates, we even have $D(t) \to 0$ as $t \to +\infty$)

First, we demonstrate that \emph{any sequence $\wt{\X}(s,t_k)$ with $t_k \to +\infty$ contains a subsequence (not relabeled) converging in $H^4_s$}. By weak compactness in $H^4_s$ and the Rellich-Kondrachov theorem, there exists a subsequence satisfying
\begin{equation}
\wt{\X}(t_k) \to \bar{\X} \text{ in } H^{4-\varepsilon}_s \, , \; \forall \varepsilon \in (0,4] \, , \qquad \wt{\X}(t_k) \rightharpoonup \bar{\X} \text{ in } H^4_s
\end{equation}
\begin{equation}
\lambda[\wt{\X}] \to \lambda[\bar{\X}] \text{ in } H^1_s \, .
\end{equation}
Since $(\mathbf{I} +\gamma \X_s \otimes \X_s) \Z_s$ is weakly continuous in $L^2$ with respect to the above convergences, $\bar{\X}$ is a steady state:
\begin{equation}
D[\bar\X] = 0 \, .
\end{equation}
% and therefore, by compactness, there exists $(\bar\X,\bar\e)$ and a sequence $t_k \to +\infty$ such that
% \begin{equation}
% (\X,\e)(t_k) \to (\bar\X,\bar\e) \text{ as } k \to +\infty \, .
% \end{equation}
Once we demonstrate that
\begin{equation}
    \label{eq:IholdIhold}
\| \wt{\X}_{ssss}(\cdot,t_k) \|_{L^2_s} \to \| \bar{\X}_{ssss} \|_{L^2_s} \, ,
\end{equation}
then we automatically have $\wt{\X}_{ssss}(\cdot,t_k) \to \bar{\X}$ strongly in $L^2_s$ by the properties of weak convergence. For this, we observe that since $D[\wt{\X}] \to 0 = D[\bar\X]$ and $(\lambda[\wt{\X}_s] \wt{\X}_s + \kappa_3 \wt{\X}_s \times \wt{\X}_{ss})_s$ converges strongly in $L^2$, we must have~\eqref{eq:IholdIhold}; ergo, $\wt{\X}(\cdot,t_k) \to \bar\X$ strongly in $H^4_s$. %$\| \wt{\X}_{ssss} \|_{L^2_s} \to \| \bar{\X}_{ssss} \|_{L^2_s}$

Fix a sequence $t_k$ and steady state $\bar{\X}$ satisfying the above strong convergence. Let $\bar{E} := E[\bar\X]$. Let $\delta > 0$ be the size of an $H^4_s$-neighborhood of $\bar\X$ in which~\eqref{eq:curvelojarestate} holds. By continuity, for each $k$, there exists a closed interval $I_k$ containing $t_k$ such that $\wt{\X} \in B_\delta(\bar{\X})$ in the $H^4_s$ topology whenever $t \in I_k$. Choose the intervals to be maximal forward-in-time. For the sake of contradiction, suppose that each $I_k$ is finite, so we may assume they are of the form $[t_k,T_k]$ with $T_k < +\infty$. By maximality, $\| \wt{\X}(T_k) - \bar\X \|_{H^4_s} = \delta$.  Then
\begin{equation}
    \label{eq:mydiffeq}
\frac{d}{dt} (E(t) - \bar{E})^\beta = - \beta (E - \bar{E})^{\beta - 1} D(t) \overset{\eqref{eq:curvelojarestate}}{\leq} -\beta  C^{-1} \sqrt{D(t)} \, .
\end{equation}
% , since via the Lojasiewicz inequality in Proposition~\ref{pro:curvelojascewicz},
% \begin{equation}
% (E - \bar{E})^{1-\beta} \leq C \sqrt{D(t)} \, .
% \end{equation}
We integrate the differential inequality~\eqref{eq:mydiffeq} from $t_k$ to $T_k$:
\begin{equation}
    \label{eq:mytimederivest}
\int_{t_k}^{T_k} \| \p_t \X \|_{L^2_s} \, dt = \int_{t_k}^{T_k} \sqrt{D(t)} \, dt \lesssim (E(t_k) - \bar{E})^\beta - (E(T_k) - \bar{E})^\beta = o_{k \to +\infty}(1) \, .
\end{equation}
By the triangle inequality,
\begin{equation}
\begin{aligned}
\| \wt{\X}(\cdot,T_k) - \bar{\X} \|_{L^2_s} &\leq \underbrace{\| \X(\cdot,T_k) - \X(\cdot,t_k) \|_{L^2_s}}_{= o_{k\to+\infty}(1) \text{ by \eqref{eq:mytimederivest}}} + \| \wt{\X}(\cdot,t_k) - \bar\X \|_{L^2_s} \\
&= o_{k \to +\infty}(1) \, ,
\end{aligned}
\end{equation}
which implies that $\wt{\X} \to \bar{\X}$ in $L^2_s$ as $k \to +\infty$. By the above analysis on sequences, there exists a subsequence such that $\wt{\X} \to \X$ in $H^4_s$; this contradicts the maximality of $T_k$.

We have demonstrated that, for every $0 < \delta \ll 1$, there exists $T(\delta) \geq 0$ such that $\| \wt{\X}(\cdot,t) - \bar\X \|_{H^4_s} \leq \delta$ whenever $t \geq T$, i.e., $\wt{\X} \to \bar\X$ in $H^4_s$.

To establish the convergence of $\X$ (compared to $\wt{\X}$), we revisit the inequality~\eqref{eq:mytimederivest}, which is now valid on $[T,+\infty)$ for $T \gg 1$ and demonstrate that $\int_T^{+\infty} \| \p_t \X \|_{L^2_s} \, dt < +\infty$. In particular, $\X$ is stabilizing in $L^2_s$ and, since $\wt{\X}$ is a translation of $\X$, $\X$ is also stabilizing in $H^4_s$. 
\end{proof}

%!TEX root = RFT GWP.tex

\section{Forcing by intrinsic curvature in 2D: large periodic solutions}

In this section, we prove Theorem~\ref{thm:periodic}. We consider the planar curve evolution in the free end setting $I = [0,1]$ with preferred curvature $\zeta(s,t)$:
\begin{equation}\label{eq:2D_forced_repeat}
\p_t \X = -\big({\bf I}+\gamma\X_s\otimes\X_s\big)\big(\X_{sss}-\lambda(s,t)\X_s-(\zeta \e_{\rm n})_s \big)_s\,, \quad \abs{\X_s}^2=1\,,
\end{equation}
where $\e_{\rm n}(s,t) = \X_s^\perp$. With  forcing, the relevant boundary conditions become
\begin{equation}
	\label{eq:forcedboundaryconditions}
(\X_{ss} - \zeta \e_{\rm n})\big|_{s=0,1} = 0 \, , \quad ((\X_{ss} - \zeta \e_{\rm n})_s - \lambda \X_s)\big|_{s=0,1} = 0 \, ,
\end{equation}
equivalent to the curvature conditions
\begin{equation}
(\kappa - \zeta)\big|_{s=0,1} = 0 \, , \quad (\kappa - \zeta)_s\big|_{s=0,1} = 0 \, , \quad \lambda = 0 \, .
\end{equation}
The tension equation with forcing is given by
\begin{equation}\label{eq:lambda00repeat}
\begin{aligned}
    (1+\gamma)\lambda_{ss}  -\abs{\X_{ss}}^2\lambda = 
    &-(4+3\gamma)(\X_{sss}\cdot\X_{ss})_s +\abs{\X_{sss}}^2\\
    &- (1+\gamma) (\zeta \e_{\rm n})_{sss} \cdot \X_s - \gamma (\zeta \e_{\rm n})_{ss} \cdot \X_{ss} \, .
\end{aligned}
\end{equation}

Due to planarity, the forced system \eqref{eq:2D_forced_repeat},~\eqref{eq:forcedboundaryconditions}, and~\eqref{eq:lambda00repeat} may be recast as an evolution for the curvature difference $\wt\kappa=\kappa-\zeta$ by the system
\begin{align}
\p_t\wt\kappa  &=  
-\wt\kappa_{ssss} -\p_t\zeta
+ \big[(2+\gamma)(\wt\kappa+\zeta)\lambda_s 
+ \lambda(\wt\kappa+\zeta)_s \nonumber \\
&\qquad 
+ (3+2\gamma)\wt\kappa_s(\wt\kappa+\zeta)^2+ (3+\gamma)\wt\kappa(\wt\kappa+\zeta)(\wt\kappa+\zeta)_s \big]_s  \label{eq:barkappa}\\
(1+\gamma)\lambda_{ss} - \lambda(\wt\kappa+\zeta)^2 &=  -(3+2\gamma)\wt\kappa_{ss}(\wt\kappa+\zeta) + \wt\kappa(\wt\kappa+\zeta)^3 \nonumber \\
&\qquad - 3(1+\gamma)\wt\kappa_s(\wt\kappa+\zeta)_s
- (1+\gamma)\wt\kappa(\wt\kappa+\zeta)_{ss} \label{eq:barlambda} \\
 \wt\kappa\big|_{s=0,1}&=0\,,\quad \wt\kappa_s\big|_{s=0,1}=0\,, \quad \lambda\big|_{s=0,1}=0\,. \label{eq:homoBCs}
\end{align}
Given a unique solution $(\wt\kappa,\lambda)$ to the curvature formulation~\eqref{eq:barkappa}-\eqref{eq:homoBCs}, the curve may be recovered by solving
\begin{equation}
\begin{aligned}
\p_t\X &= ({\bf I}+\gamma\e_{\rm t}\otimes\e_{\rm t})((\wt\kappa\e_{\rm n})_s-\lambda\e_{\rm t})_s\,, \\
\p_t\e_{\rm t} &= F(\wt\kappa,\zeta,\lambda)\e_{\rm n}\,, \qquad 
\p_t\e_{\rm n} = -F(\wt\kappa,\zeta,\lambda)\e_{\rm t}\,,
\end{aligned}
\end{equation}
where $\e_{\rm t}=\X_s$ and the expression for $F$ is given explicitly in terms of $\wt\kappa, \zeta,\lambda$: % Technically enough to solve at $s=0$
\begin{equation}
\begin{aligned}
F(\wt\kappa,\zeta,\lambda) := \p_t\X_s\cdot\e_{\rm n} &=  
-\wt\kappa_{sss} 
+ (3+\gamma)\wt\kappa(\wt\kappa+\zeta)(\wt\kappa+\zeta)_s 
+ (3+2\gamma)\wt\kappa_s(\wt\kappa+\zeta)^2 \\
&\qquad +\lambda(\wt\kappa+\zeta)_s+ (2+\gamma)(\wt\kappa+\zeta)\lambda_s\,.
\end{aligned}
\end{equation}

The formulation \eqref{eq:barkappa}-\eqref{eq:homoBCs} was considered in~\cite{mori2023well},\footnote{up to a redefinition of the filament tension} where the authors investigate time-periodic solutions driven by small
\begin{equation}
	\label{eq:timeperiodiczeta}
\zeta \in C(\T_T;H^1_s) \, , \quad \p_t \zeta \in C(\T_T;L^2_s)\,.
\end{equation}
Here $\T_T := \R/T\mathbb{Z}$ denotes the $T$-periodic torus. Therein, a solution theory was developed in the time-weighted space $\wt\kappa \in C([0,T];L^2_s) \cap \big\{ \| t^{1/4} \| \wt\kappa(\cdot,t) \|_{H^1_s} \|_{L^\infty(0,T)} < +\infty \big\}$.

% Recall the center-of-mass velocity
% \begin{equation}\label{eq:swimming_repeat}
% \bm{V}(t) := \int_0^1\frac{\p\X}{\p t}(s,t)\,ds = -\gamma\int_0^1\X_s\big(\X_s\cdot(\X_{sss}-(\zeta \e_{\rm n})_s)_s-\lambda_s \big)\,ds\,.
% \end{equation}

In this section, we construct time-periodic solutions \emph{without} size restrictions on $\zeta$. It will be convenient to work with solutions in the energy space $\wt{\kappa} \in C([0,T];L^2_s) \cap L^2_t H^2_s(I \times (0,T))$. The analysis is as in~\cite{mori2023well} and Section~\ref{sec:immersednoforcing}, except
\begin{itemize}
\item the time-weighted function space in~\cite{mori2023well} is replaced by $L^4_t H^1_s(I \times (0,T))$ in the fixed point argument,
\item we consider short-time theory for arbitrarily large $\zeta$, not only long-time theory for small $\zeta$, and
\item we record the continuous dependence on the initial data.
\end{itemize}

With these adjustments, the relevant local well-posedness statement is:
\begin{proposition}%[LWP with preferred curvature]
	\label{prop:timeperiodicLWP} %$\wt{T} > 0$ and
Let $\zeta$ satisfy~\eqref{eq:timeperiodiczeta}. Let $\wt{\kappa}^{\rm in} \in L^2(0,1)$. Then there exists $T^* \in (0,+\infty]$ and a unique maximal solution
\begin{equation}
\wt{\kappa} \in C([0,S];L^2_s) \cap L^2_t H^2_s((0,1) \times (0,S)) \, , \quad \forall S < T^* \, ,
\end{equation}
to the curvature formulation~\eqref{eq:barkappa}-\eqref{eq:homoBCs} with initial data $\wt{\kappa}^{\rm in}$. % for all $T \in (0,T^*)$.

If $T^* < +\infty$, then
\begin{equation}
	\label{eq:continuationcriterionperiodic}
\| \wt{\kappa} \|_{L^\infty_t L^2_s \cap L^2_t H^2_s((0,1) \times (0,S))} \to +\infty \text{ as } S \to T^*_- \, .
\end{equation}

The solution $\wt{\kappa}$ depends continuously on the initial data $\wt{\kappa}^{\rm in}$ in the following sense:

Let $\wt{\kappa}_1$ be a solution on $[0,T^*_1)$. Let $S \in (0,T^*_1)$. Then, for all $\wt{\kappa}_2^{\rm in}$ sufficiently close to $\wt{\kappa}_1^{\rm in}$, the corresponding solution lives on $[0,S]$, and
\begin{equation}
\wt{\kappa}_2 \to \wt{\kappa}_1 \text{ in } C([0,S];L^2_s) \cap L^4_t H^1_s((0,1) \times (0,S)) \text{ as } \wt{\kappa}_2^{\rm in} \to \wt{\kappa}_1^{\rm in} \text{ in } L^2(0,1) \, .
\end{equation}
\end{proposition}

% (The proof of the fixed point argument encodes short-time uniqueness, which can be extended to long-time uniqueness by iterating the argument. The same is true for continuous dependence.)

% In this section, we demonstrate the existence of arbitrarily large time-periodic solutions.

The proof of Theorem~\ref{thm:periodic} will also demonstrate that $T^* = +\infty$ in Proposition~\ref{prop:timeperiodicLWP}.

First, we consider the tension with preferred curvature, analogous to Lemma~\ref{lem:tension}. 
\begin{lemma}[Tension estimate]
Let $\X_{ss} \in H^1(0,1)$ and $\zeta \e_{\rm n}- \X_{ss} \in H^1_0(0,1)$ with $|\X_s|^2 = 1$. Then there exists a unique solution $\lambda \in H^1_0(0,1)$ to~\eqref{eq:lambda00repeat}. The solution satisfies
\begin{equation}
	\label{eq:tensionestimateforced}
\| \lambda \|_{H^1_s} \lesssim \| \X_{sss} \|_{L^2_s}^2 + (1 + \| \X_{sss} \|_{L^2_s}^2) \| \zeta \|_{H^1_s} \, .
\end{equation}
\end{lemma}

\begin{proof}
The relevant energy identity for the tension is
\begin{equation}
	\label{eq:needlastterm}
\begin{aligned}
(1+\gamma) &\int_0^1 |\lambda_s|^2\,ds + \int_0^1 \abs{\X_{ss}}^2\lambda^2 \,ds \\
&= -(4+3\gamma)\int_0^1 (\X_{sss}\cdot\X_{ss}) \lambda_s\,ds - \int_0^1 \abs{\X_{sss}}^2 \lambda\,ds \\
&\quad- (1+\gamma) \int_0^1 (\zeta \e_{\rm n})_{ss} \cdot (\lambda \X_s)_s\,ds + \gamma \int_0^1 (\zeta \e_{\rm n})_{ss} \cdot \X_{ss} \lambda\,ds \, .
\end{aligned}
\end{equation}
We analyze the new terms not present in Lemma~\ref{lem:tension}. Since
\begin{equation}
\zeta_{ss} \e_{\rm n} \cdot \big(\lambda_s \X_s\big) = 0 \, ,
\end{equation}
we rewrite\footnote{Technically, this should be justified by writing $\big((\zeta \e_{\rm n})_{ss} \cdot \X_s\big)_s = \big(2\zeta_{s} (\e_{\rm n})_s + \zeta (\e_{\rm n})_{ss}\big)_s$ at the level of the tension equation.}
\begin{equation}
- \int_0^1 (\zeta \e_{\rm n})_{ss} \cdot (\lambda \X_s)_s \,ds = - \int_0^1 (2 \zeta_s (\e_{\rm n})_s + \zeta (\e_{\rm n})_{ss}) \cdot (\lambda_s \X_s) \, ds + \int_0^1 (\zeta \e_{\rm n})_s (\lambda \X_{ss})_s \, ds \, ,
\end{equation}
which, along with the final term of~\eqref{eq:needlastterm}, can be estimated by $(1 + \| \X_{sss} \|_{L^2_s}^2) \| \zeta \|_{H^1_s} \| \lambda \|_{H^1_s}$. \end{proof}
% We also integrate by parts in the last term of~\eqref{eq:needlastterm}:
% \begin{equation}
% \int_0^1 (\zeta \e_n)_{ss} \cdot \X_{ss} \lambda \, ds = - \int_0^1 (\zeta \e_n)_{s} \cdot (\X_{ss} \lambda)_s \, ds \, .
% \end{equation}
% After IBP and Young's inequality,  contribute the (rough) estimate
% \begin{equation}
% \begin{aligned}
% %&\| (\zeta \X_s)_{ss} \cdot \X_{ss} \|_{L^1_s}^2 + \| (\zeta \X_s)_{ss} \|_{L^2_s}^2 \\
% %&\quad \lesssim
% %\left( 1 + \| \X_{ss} \|_{L^2_s}^2 \right) \| (\zeta \X_s)_{ss} \|_{L^2_s}^2 \\
% \| (\zeta \X_s)_{ss} \|_{L^2_s}^2 \lesssim \| \zeta \|_{H^2_s}^2 (1 + \| \X_{sss} \|_{L^2_s}^2) \, ,
% \end{aligned}
% \end{equation}
% which yields~\eqref{eq:tensionestimateforced}.

\begin{proof}[Proof of Theorem~\ref{thm:periodic}]
Fix $T > 0$ and $\zeta$ satisfying \eqref{eq:timeperiodiczeta}. We consider the \emph{time-$T$ map} (Poincar{\'e} map)
\begin{equation}
\Phi^T : L^2(0,1) \to L^2(0,1) \, ,
\end{equation}
\begin{equation}
\Phi^T(\wt{\kappa}^{\rm in}) = \wt{\kappa}(\cdot,T) \, ,
\end{equation}
where $\wt{\kappa}$ is the solution, granted by Proposition~\ref{prop:timeperiodicLWP}, to the curvature formulation with initial data $\wt{\kappa}^{\rm in}$.

For $\Phi^T$ to be well defined, we should demonstrate that every such solution $\wt{\kappa}$ exists up to time $T$ (actually, $T^* = +\infty$). Then, moreover, $\Phi^T : L^2_s \to L^2_s$ will be continuous. Let $\wt{\kappa}$ be a solution. We have the energy identity
\begin{equation}
\begin{aligned}
&\frac{1}{2} \frac{d}{dt} \int_0^1 |\X_{ss} - \zeta \e_{\rm n}|^2\,ds \\
&\quad = - \int_0^1 \big[({\bf I} + \gamma \X_s \otimes \X_s) (\Z - (\zeta \e_{\rm n})_s)_s \big]\cdot \big(\Z - (\zeta \e_{\rm n})_s\big)_s \, ds - \int_0^1 \p_t \zeta (\kappa - \zeta) \, ds \, .
\end{aligned}
\end{equation}
Let
\begin{equation}
\begin{aligned}
E_\zeta &:= \frac{1}{2}\int_0^1 \underbrace{|\X_{ss} - \zeta \e_{\rm n}|^2}_{\wt{\kappa}^2} \,ds \\
 D_\zeta &:= \int_0^1 \big[({\bf I} + \gamma \X_s \otimes \X_s) (\Z - (\zeta \e_{\rm n})_s)_s \big]\cdot \big(\Z - (\zeta \e_{\rm n})_s\big)_s \, ds \, .
\end{aligned}
\end{equation}
Then we have the Poincar{\'e} inequality
\begin{equation}
E_\zeta \leq c_0 D_\zeta \, ,
\end{equation}
proved as in~\eqref{eq:Poincareinequality}, and the differential inequality
\begin{equation}
\dot E_\zeta \leq - D_\zeta - \int_0^1 \p_t \zeta (\kappa - \zeta) \, ds \leq - \frac{c_0}{2} E_\zeta + \frac{1}{2c_0} \| \p_t \zeta \|_{L^2_s}^2 \, .
\end{equation}
We remark that this contractivity is the reason to consider free ends instead of the torus.
By Gr{\"o}nwall's inequality, we have
\begin{equation}
	\label{eq:Gronwallperiodic}
E_\zeta(t) \leq e^{-\frac{c_0 t}{2}} E_\zeta^{\rm in} + \frac{1}{2c_0} \int_0^t e^{-\frac{c_0(t-t')}{2}} \| \p_t \zeta(\cdot,t') \|_{L^2_s}^2 \, dt' \, , \quad \forall t \in (0,T^*) \, ,
\end{equation}
and additionally
\begin{equation}
\| D_\zeta \|_{L^1(0,T')} \leq C(M,T')
\end{equation}
on any finite time interval on which the solution is defined. According to the continuation criterion~\eqref{eq:continuationcriterionperiodic}, $T^* = +\infty$.

Next, we demonstrate that $\Phi^T$ stabilizes a ball. Define
\begin{equation}
N := \sup_{t \in [0,T]} \| \p_t \zeta(\cdot,t) \|_{L^2_s}^2 \, , \quad N' = N + \sup_{t \in [0,T]} \| \zeta(\cdot,t) \|_{H^1_s}^2 \, .
\end{equation}
From~\eqref{eq:Gronwallperiodic}, we have
\begin{equation}
E_\zeta(T) \leq e^{-\frac{c_0T}{2}} E_\zeta^{\rm in} + \frac{N}{c_0^2} \, .
\end{equation}
Choose $M$ satisfying $e^{-\frac{c_0T}{2}} M + N/c_0^2 \leq M$,
for example,
\begin{equation}
M = \frac{N}{(1 - e^{-\frac{c_0T}{2}}) c_0^2} \, .
\end{equation}
Then
\begin{equation}
	\label{eq:PhiTballtoitself}
\Phi^T : \overline{B_M^{L^2}} \to \overline{B_M^{L^2}} \, .
\end{equation}

Finally, we demonstrate that $\Phi^T$ as in~\eqref{eq:PhiTballtoitself} is \emph{compact}. This is done via a \emph{smoothing estimate}: We will demonstrate that
\begin{equation}
\| \wt{\kappa}(\cdot,T) \|_{H^2_s} \leq M'(M,T) \, ,
\end{equation}
that is,
\begin{equation}
\Phi^T : \overline{B_M^{L^2}} \to \overline{B_{M'}^{H^2}} \, .
\end{equation}
The compactness of the embedding $H^2(0,1) \into L^2(0,1)$ guarantees the compactness of $\Phi^T$.

Recall from Section~\ref{subsec:energetics} that $\sqrt{D_\zeta}$ controls $\| \X_{ss} - \zeta \e_{\rm n} \|_{H^2_s}$, which in turn controls $\| \wt{\kappa} \|_{H^2_s}$. There are two (roughly equivalent) ways to prove the smoothing estimate, both of which go through $D_\zeta$. One way is to write
\begin{equation}
\frac{1}{2} \frac{d}{dt} (tD_\zeta(t)) = \frac{1}{2} D_\zeta(t) + t \frac{1}{2} \frac{d}{dt} D_\zeta(t)\,,
\end{equation}
substitute an expression for $\frac{d}{dt}D_\zeta$, as in Section~\ref{sec:propagationestimates}, and close a differential inequality for $tD_\zeta(t)$. A second way, which we take, is the following: Let $\wt{t} \in [0,T/2]$. By Chebyshev's inequality, there exists a time $t_0 \in [0,\wt{t}]$ such that
\begin{equation}
	\label{eq:Dzetaboundedtimet0}
D_\zeta(t_0) \leq \wt{t}^{-1} C(M,T) \, .
\end{equation}
Below, we propagate the quantity $D_\zeta(t)$ from time $t_0$, where it is bounded by~\eqref{eq:Dzetaboundedtimet0}, to time $T$. The computations are like those in Section~\ref{sec:propagationestimates}.\footnote{Technically, one should further justify the computations below, since \emph{a priori} the solution $\wt{\kappa}$ is not known to belong to $C([t_0,T];H^2(0,1)) \cap L^2_t H^4_s((0,1) \times (t_0,T))$. One way would be to prove short-time well-posedness from initial data in $H^2(0,1)$, as was done in Section~\ref{sec:immersednoforcing}.}

Following those computations, we have
\begin{equation}
\begin{aligned}
\frac{1}{2} \frac{d}{dt} D_\zeta(t) &= \underbrace{\int_0^1 \p_t \X_s \cdot \p_t (\Z - (\zeta \e_{\rm n})_s)\,ds}_{=I_1} \\
&\quad + \underbrace{\gamma \int_0^1 (\X_s \cdot (\Z - (\zeta \e_{\rm n})_s)_s) (\p_t \X_s \cdot (\Z - (\zeta \e_{\rm n})_s))\,ds}_{=I_2} \, .
\end{aligned}
\end{equation}

In the term $I_1$, we substitute $\Z = \X_{sss} - \lambda \X_s$ to obtain
\begin{equation}
	\label{eq:Ihaveanadvantageousterm}
\begin{aligned}
\int_0^1 &\p_t \X_s \cdot \p_t (\X_{ss}- (\zeta \e_{\rm n}))_s\,ds - \int_0^1 \lambda |\p_t \X_s|^2\,ds \\
&= - \int_0^1 |\p_t \X_{ss}|^2\,ds + \int_0^1 \p_t \X_{ss} \cdot \p_t (\zeta \e_{\rm n})\,ds - \int_0^1 \lambda |\p_t \X_s|^2\,ds \, .
\end{aligned}
\end{equation}
We will split various terms by Cauchy's inequality and absorb $\p_t \X_{ss}$ contributions into the advantageous term on the left-hand side of~\eqref{eq:Ihaveanadvantageousterm}. 
%(We say that the remainder ``contributes".) 

First, we split the term $\int_0^1 \p_t \X_{ss} \cdot \p_t (\zeta \e_{\rm n})_s \, ds$, leaving a right-hand side contribution of (a constant multiple of)
\begin{equation}
	\label{eq:Icontribution}
\| \p_t \zeta \|_{L^2_s}^2 + \| \zeta \|_{L^\infty_s}^2 \| \p_t \X_s \|_{L^2_s}^2 \, .
\end{equation}
Observe the interpolation inequality
\begin{equation}
	\label{eq:usefulinterpolationwithDzeta}
\| \p_t \X_s \|_{L^2_s} \lesssim \| \p_t \X \|_{L^2_s} + \| \p_t \X \|_{L^2_s}^{1/2} \| \p_t \X_{ss} \|_{L^2_s}^{1/2} \, ,
\end{equation}
keeping in mind that $\| \p_t \X \|_{L^2_s} \lesssim D_\zeta^{1/2}(t)$. Therefore,~\eqref{eq:Icontribution} contributes
\begin{equation}
	\label{eq:contribution1}
N' + N' D_\zeta + (N')^2 D_\zeta 
\end{equation}
to the right-hand side remainder. Next,
\begin{equation}
\left| \int_0^1 \lambda |\p_t \X_s|^2 \, ds \right| \leq \| \lambda \|_{L^\infty_s} \| \p_t \X_s \|_{L^2_s}^2 \, ,
\end{equation}
which, after applying~\eqref{eq:usefulinterpolationwithDzeta} and splitting, contributes
\begin{equation}
	\label{eq:contribution2}
\| \lambda \|_{H^1_s} D_\zeta + \| \lambda \|_{H^1_s}^2 D_\zeta 
\end{equation}
to the right-hand side. Recalling the tension estimate
\begin{equation}
\| \lambda \|_{H^1_s} \lesssim \| \X_{sss} \|_{L^2_s}^2 + ( 1 + \| \X_{sss} \|_{L^2_s} )^2 N' \, ,
\end{equation}
we have
\begin{equation}
\| \lambda \|_{L^2_t H^1_s((0,1) \times (0,T))} \lesssim C(M,N',T) \, .
\end{equation}

Next, we estimate term $I_2$ by (a constant multiple of)
\begin{equation}
	\label{eq:IIcontribution}
D_\zeta^{1/2} \| \p_t \X_s \|_{L^2_s} \| \Z - (\zeta \e_{\rm n})_s \|_{L^\infty_s} \lesssim D_\zeta^{1/2} (D_\zeta + D_\zeta^{1/2} \| \p_t \X_{ss} \|_{L^2_s}^{1/2}) \| \Z - (\zeta \e_{\rm n})_s \|_{L^\infty_s} \, .
\end{equation}
Recall $\Z - (\zeta \e_{\rm n})_s = \X_{sss} - \lambda \X_s - (\zeta \e_{\rm n})_s$, whose terms we estimate separately as
\begin{equation}
\begin{aligned}
\| \X_{sss} - (\zeta \e_{\rm n})_s \|_{L^\infty_s} &\lesssim E_\zeta^{1/8} D_\zeta^{3/8} \,,\\
\| \lambda \X_s \|_{L^2_t H^1_s((0,1) \times (0,T))} &\leq C(M,N',T) \, .
\end{aligned}
\end{equation}
In summary,~\eqref{eq:IIcontribution} contributes
\begin{equation}
\label{eq:contribution3}
D_\zeta^{3/2} (D_\zeta^{3/8} E_\zeta^{1/8} + \| \lambda \|_{H^1_s})
\end{equation}
to the right-hand side, and, after splitting, we have
\begin{equation}
	\label{eq:contribution4}
[D_\zeta (D_\zeta^{3/8} E_\zeta^{1/8} + \| \lambda \|_{H^1_s})]^{4/3} \lesssim D_\zeta^{11/6} E_\zeta^{1/6} + D_\zeta^{4/3} \| \lambda \|_{H^1_s}^{4/3} \, .
\end{equation}
Notice that $\big\| D_\zeta^{1/3}(t) \| \lambda(\cdot,t) \|_{H^1_s}^{4/3} \big\|_{L^1_t(0,T)} \leq C(M,N',T)$.

In conclusion, combining~\eqref{eq:contribution1},~\eqref{eq:contribution2},~\eqref{eq:contribution3}, and~\eqref{eq:contribution4}, we have
\begin{equation}
\frac{d}{dt} D_\zeta(t) \leq A(t) D_\zeta(t) + B(t) \, ,
\end{equation}
where $\|A(t) \|_{L^1(0,T)}, \|B(t) \|_{L^1(0,T)} \leq C(M,N',T)$. Evolving this differential inequality from time $t_0$, we obtain the desired bound on $D_\zeta(T)$.
%(Technically, we must also give an argument that the quantities in question are finite. This can be done at the level of the fixed point argument.)

Finally, the Schauder fixed point theorem (Proposition~\ref{pro:schauder}) yields the existence of a fixed point and a $T$-periodic solution. \end{proof}

\begin{proposition}[Schauder fixed point theorem]
	\label{pro:schauder}
Let $X$ be a Banach space and $K \subset X$ be a non-empty bounded, closed, and convex set. Suppose that $\mc{A} : K \to K$ is compact, in the sense that it is continuous and maps bounded sequences to relatively compact sequences. Then $\mc{A}$ has a fixed point.
\end{proposition}

See, e.g.,~\cite[Corollary 11.2, p. 280]{gilbarg1977elliptic}.

% \da{Finish estimates on size of swimming speed}

%!TEX root = RFT GWP.tex

\section{Evolution of immersed rod with material frame in 3D}\label{sec:rod3D}

%%%%%%%%%%%%%%%%%%%%%%%%%%%%%%%%%

In this section, we prove Theorem~\ref{thm:GWP_curveframe}. %As a reminder, 
We consider the coupled curve-frame evolution
\begin{align}
    \p_t\X 
    &=  -({\bf I}+\gamma\X_s\otimes\X_s)\big(\X_{sss}-\lambda\X_s -\eta\kappa_3\X_s\times\X_{ss} \big)_s \,, \quad \abs{\X_s}^2=1\,,
    \label{eq:Xss_0-repeat}\\
    \p_t\kappa_3   
    &= \bigg(\frac{\eta}{\alpha}+\eta\abs{\X_{ss}}^2\bigg)(\kappa_3)_{ss}+ \mc{R}^{\rm f}[\X,\kappa_3]\,, \label{eq:kap3_0-repeat}
\end{align}
where 
\begin{align}
    \mc{R}^{\rm f}[\X,\kappa_3] &= \eta\big(\abs{\X_{ss}}^2\big)_s(\kappa_3)_s + \eta\kappa_3\X_{ssss}\cdot\X_{ss} + \eta\kappa_3\abs{\X_{ss}}^4  \label{eq:Rf0-repeat} \\
    &\qquad 
      -\X_{sssss} \cdot (\X_s\times\X_{ss}) +\lambda\X_{sss}\cdot (\X_s\times\X_{ss})  \nonumber
\end{align}
and the tension satisfies
\begin{equation}\label{eq:tension_0-repeat}
    (1+\gamma)\lambda_{ss}  -\abs{\X_{ss}}^2\lambda = 
    -(4+3\gamma)(\X_{sss}\cdot\X_{ss})_s +\abs{\X_{sss}}^2  - \eta\kappa_3\X_s\cdot(\X_{ss}\times\X_{sss})\,.
\end{equation}
The equations are considered on $I=[0,1]$ and $\T$. When $I=[0,1]$, we additionally impose the boundary conditions
\begin{equation}
    \label{eq:rodboundaryconditions}
\X_{ss}\,,\, \X_{sss}\big|_{s=0,1} = 0 \, , \quad \kappa_3,\lambda\big|_{s=0,1} = 0 \, .
\end{equation}
We again consider~\eqref{eq:Xss_0-repeat} as an evolution for $\X_s$. Note that $\int_I \X \, ds$ may be recovered from integrating~\eqref{eq:Xss_0-repeat} in space and time. 

Throughout this section, we assume that $\eta \in (0,\eta_{\rm max}]$ and $\alpha \in (0,\alpha_{\rm max}]$, and the constants $C$ and implied constants in $\lesssim$ may depend on $\eta_{\rm max}$, $\alpha_{\rm max}$.

Given the discussion surrounding the curvature criterion~\eqref{eq:Xss_criterion}, we make the following definition.

\begin{definition}[Strong solution]
Let $T>0$ and $I = [0,1]$ or $\T$. We say that $(\X_{s},\kappa_3)$ is a \emph{strong solution} on $I \times [0,T]$ if $|\X_s|^2 = 1$ and
\begin{equation}
    \label{eq:regreqofstrong}
(\X_{s},\kappa_3) \in (L^\infty_t H^3_s \cap L^2_t H^5_s) \times (L^\infty_t H^1_s \cap L^2_t H^2_s) \text{ on } I \times (0,T) \, ,
\end{equation}
\begin{equation}
    \label{eq:esssup}
{\rm ess}\,\sup_{t \in (0,T)} \| \X_{ss} \|_{L^\infty_s(I)}^2 < \frac{1}{\alpha} \, ,
\end{equation}
\begin{equation}
\p_t \X_{ss} \in L^2_t L^2_s(I \times (0,T)) \, ,
\end{equation}
and there exists
\begin{equation}
\lambda \in L^2_t H^2_s(I \times (0,T))
\end{equation}
such that $(\X_{s},\kappa_3)$ satisfy the equations~\eqref{eq:Xss_0-repeat}-\eqref{eq:tension_0-repeat} in $\X_s$ formulation\footnote{in the sense of distributions} with tension $\lambda$. When $I=[0,1]$, we further require the boundary conditions~\eqref{eq:rodboundaryconditions}, whereas when $I=\T$, we require that $\X_s$ has zero mean.
% \begin{equation}
% \X_{ss}, \X_{sss}, \kappa_3, \lambda|_{s=0,1} = 0 \, .
% \end{equation}
\end{definition}

%\begin{remark}
The energy estimates below can be justified for strong solutions, and moreover $(\X_s,\kappa_3) \in C([0,T];H^3_s \times H^1_s)$.\footnote{Sometimes we speak of strong solutions on $[0,T]$ instead of on $I \times [0,T]$. We say strong solution on $[0,T)$ to mean strong solution on $[0,S]$ for all $S \in (0,T)$.}
%\da{Consequences, e.g., energy estimate makes sense, continuity in time, initial data, things like that}
%Afterward, $\lambda$ satisfies the equations.
%\end{remark}

The definition of strong solution encodes information coming from the structure of the energy estimates which is not evidently available from a na{\"i}ve fixed point argument. For example, $\p_t \X_{ss}$ involves $(\kappa_3)_{sss}$, which is not bounded in $L^2_{t,s}$ by the regularity requirements~\eqref{eq:regreqofstrong}. Therefore, even our local well-posedness theory is much more subtle than the curve-only case.

% A starting place for the solution theory might be weak solutions for which the energy dissipation (in)equality sense, analogous to Leray-Hopf solutions in the Navier-Stokes theory. We immediately run into two issues: (i) the energy dissipation is not evidently smoothing in the curve, and (ii) solutions are not evidently unique. A condition ameliorating (i) is $|\X_{ss}| < 1/\alpha$. Regarding (ii), this is handled by working with solutions more regular than the energy class. Our choice of solution class is somewhat subtle and relies on structural properties of the equation.

\begin{proposition}[Local well-posedness]
    \label{thm:wellposednesstheoryrod}
Let
\begin{equation}
(\X_{s}^{\rm in},\kappa_3^{\rm in}) \in H^3 \times H^1(I) \, , \quad |\X_s^{\rm in}|^2 = 1 \, ,
\end{equation}
\begin{equation}
|\X_{ss}^{\rm in}|^2 < \frac{1}{\alpha} \, ,
\end{equation}
with the relevant boundary conditions. Then there exists a unique maximally defined strong solution $(\X_{ss},\kappa_3)$ on $[0,T^*)$ with initial data $(\X_s^{\rm in},\kappa_3^{\rm in})$.

(Continuation criterion) If $T^* < +\infty$, then
\begin{equation}
\limsup_{T \to T^*_-} \| \X_{ss} \|_{L^\infty_t L^\infty_s(I \times (0,T))}^2 = \frac{1}{\alpha} \, .
\end{equation}
\end{proposition}

Given the previous proposition, we may obtain the following global-in-time consequences:
\begin{proposition}[Conditional global well-posedness]
    \label{pro:GWProdthing}
In the notation of Proposition~\ref{thm:wellposednesstheoryrod}, there exist $C_1, C_2 > 0$ such that if
\begin{equation}
    \label{eq:amisatisfied}
\| \X_{ss}^{\rm in} \|_{H^2_s(I)} + \| \kappa_3^{\rm in} \|_{H^1_s(I)} \leq \frac{1}{C_1} \log \frac{1}{C_2 \alpha} \, ,
\end{equation}
then $T^* = +\infty$.

(Global solutions converge) Suppose that $T^* = +\infty$ and
\begin{equation}
\sup_{t \in [0,+\infty)} \| \X_{ss} \|_{L^\infty_s}^2 < \frac{1}{\alpha} \, .
\end{equation}
Then there exists a steady curve-frame pair $(\bar{\X},\bar{\e}) \in H^4_s \times H^2_s$, where $\bar{\e}(s)$ is a frame vector, such that
\begin{equation}
    \label{eq:rodconvergence}
(\X,\e)(t) \to (\bar{\X},\bar{\e}) \text{ in } H^4_s \times H^2_s \text{ as } t \to +\infty \, .
\end{equation}
\end{proposition}

\begin{remark} %, and in practice they are not exorbitant. Therefore, 
We do not calculate explicit values for $C_1$ and $C_2$ (in principle, they can be estimated from the proof), so we do not determine whether the condition~\eqref{eq:amisatisfied} is satisfied in physical parameter regimes. We merely observe that given a value of the slenderness parameter $\epsilon \sim 10^{-2}$--$10^{-3}$, we have $\alpha \sim 10^{-4}$--$10^{-6}$ according to~\eqref{eq:alphadef}.
\end{remark}

%In the free end setting,~\eqref{eq:rodconvergence} follows from a Poincar{\'e} inequ

The key point in both Propositions~\ref{thm:wellposednesstheoryrod} and~\ref{pro:GWProdthing} is to prove two \emph{a priori} estimates for strong solutions which propagate the energy $E$ and energy dissipation $D$ defined in~\eqref{eq:Et_and_Dt}. These are shown in Section~\ref{sec:energeticsrod}. In particular, the propagation of $D$ estimate is responsible for $\p_t \X_{ss} \in L^2_{t,s}$ -- see Section~\ref{subsec:propagation_estimates}. We then proceed directly to the proof of Proposition~\ref{pro:GWProdthing} in Section~\ref{subsec:GWP_rodframe}, where we exploit the Lojasiewicz inequality from Proposition~\ref{pro:rodlojascewicz} in the periodic setting.
%In the periodic setting,~\eqref{eq:rodconvergence} is based on t

Even with the \emph{a priori} estimates in hand, the existence of strong solutions goes through a hyperviscous regularization~\eqref{eq:Xs_evol}-\eqref{eq:kap3_evol}, in which $\nu\p_s^4 \kappa_3$, $\nu>0$, is incorporated into the $\kappa_3$ equation. The hyperviscous equations are amenable to fixed point arguments, which we carry out in Appendix~\ref{sec:hyperviscous}. We then obtain a strong solution in the limit as $\nu\to 0$ in Section \ref{subsec:LWP_rodframe}. These steps are relegated to the end of the section.

%Equilibria are as described in Langer-Singer Kirchhoff rod paper.

\subsection{Energetics}
    \label{sec:energeticsrod}

In analogy to the curve-only setting, we define  
\begin{equation}
\wt{\bm Z}= \X_{sss}-\lambda\X_s-\eta\kappa_3\X_s\times\X_{ss}
\end{equation}
as well as the energy and dissipation quantities
\begin{equation}\label{eq:Et_and_Dt}
\begin{aligned}
E(t) &= \frac{1}{2}\int_{I}\big(\abs{\X_{ss}}^2 + \eta\kappa_3^2\big)\,ds \,,\\
D(t) &= \int_{I}\big(({\bf I}+\gamma\X_s\otimes\X_s)\wt{\bm{Z}}_s\big)\cdot\wt{\bm{Z}}_s\, ds + \frac{\eta^2}{\alpha}\int_{I}(\kappa_3)_s^2\,ds\,.
\end{aligned}
\end{equation}
We may write the energy identity~\eqref{eq:E_ID} as 
\begin{equation}\label{eq:for_decay1}
\p_tE(t) = -D(t) \,.
\end{equation}
This computation can be justified for strong solutions on their domain of definition. Immediately, we have that $E(t)\le E^{\rm in}$ and $\int_0^\infty D(t)\,dt\le E^{\rm in}$, as in~\eqref{eq:E_D_closed}.

For the free-ended rod we can further show that $E(t)$ decays exponentially-in-time and $D(t)$ has time-integrated exponential decay, as in~\eqref{eq:E_D_free}. We first note that, by an analogous argument to \eqref{eq:muckinaboutwiththecurvature} in the curve-only setting, we have the following Poincar\'e-type inequality:
% \begin{equation}
% \begin{aligned}
% \int_0^1\abs{\X_{ss}}^2\,ds &\le \int_0^1\abs{\p_s\abs{\X_{ss}}^2}\,ds = 2\int_0^1\abs{\X_{sss}\cdot\X_{ss}}\,ds\\
% &= 2\int_0^1\abs{(\X_{sss}-\lambda\X_s-\eta\kappa_3\X_s\times\X_{ss})\cdot\X_{ss}}\,ds\\
% &\leq \frac{2}{\pi} \sqrt{D(t)}\norm{\X_{ss}}_{L^2_s}\,.
% \end{aligned}
% \end{equation}
%
% In particular, $\sqrt{D(t)}$ controls $\norm{\X_{ss}}_{L^2_s}$:
% \begin{equation}
% \norm{\X_{ss}}_{L^2_s} \leq \frac{2}{\pi} \sqrt{D(t)}\,.
% \end{equation}
\begin{equation}\label{eq:rod_Poincare}
E(t) \leq \frac{1}{2\pi^4} D(t) \, .
\end{equation}
In addition, $\eta\int_0^1\abs{\kappa_3}^2\,ds\leq \pi^2 \eta \int_0^1\abs{(\kappa_3)_s}^2\,ds$ by the 1D Poincar\'e inequality. Using \eqref{eq:for_decay1}, we thus have
\begin{equation}\label{eq:for_decay2}
\p_tE(t) \le -D(t) \le -c_0E(t)\,,
\end{equation}
where, as in \eqref{eq:dtE_noforce}, we let $c_0^{-1}$ denote the best constant in \eqref{eq:rod_Poincare}. By Gr\"onwall's inequality, the free-ended rod energy satisfies 
\begin{equation}\label{eq:E_decay}
E(t) \le e^{-c_0t}E^{\rm in}\,.
\end{equation}
Additionally, from \eqref{eq:for_decay2} we obtain a time-integrated decay estimate on $D(t)$. For any $\theta\in(0,1)$, we have
\begin{equation}
\p_tE(t) + \theta D(t) \leq -(1-\theta) D(t) \leq - c_0 (1-\theta) E(t) \, ;
\end{equation}
in particular,
\begin{equation}
\p_t (Ee^{c_0 (1-\theta) t}) + \theta D e^{c_0 (1-\theta) t} \leq 0 \, .
\end{equation}
Integrating in time and using \eqref{eq:E_decay}, we have
\begin{equation}
\theta \int_0^{+\infty} e^{c_0 (1-\theta) t} D(t) \, dt \leq E^{\rm in} 
\end{equation}
for a global-in-time strong solution.

%%%%%%%%%%%%%%
\subsubsection{Relationships between quantities}
We need to establish that the quantities $E(t)$ and $D(t)$ control $\X_{ssss}$ in a useful way. Recalling the notation $\bm{Z}=\X_{sss}-\lambda\X_s$ from Section \ref{subsec:energetics}, we define 
\begin{equation}\label{eq:D1t}
\begin{aligned}
D_1(t) &= \int_{I}\bigg( \abs{\bm{Z}_s}^2  + \gamma \big(\X_s\cdot\bm{Z}_s\big)^2
    + \frac{\eta^2}{\alpha}(\kappa_3)_s^2 
      \bigg) \,ds\,.
\end{aligned}
\end{equation}
Immediately from \eqref{eq:Xssss_free0} and \eqref{eq:Xssss_closed0}, we have
\begin{equation}\label{eq:Xssss_free}
\norm{\X_{ssss}}_{L^2_s}\lesssim D_1(t)^{1/2}(1+E(t)^{1/2})+E(t)^{5/2}
\end{equation}
in the free end setting and 
\begin{equation}\label{eq:Xssss_closed}
\norm{\X_{ssss}}_{L^2_s}\lesssim D_1(t)^{1/2}(1+E(t)^{1/2})+ E(t)+E(t)^{5/2}
\end{equation}
in the closed loop setting. 
We need to show that, under the condition $\abs{\X_{ss}}^2 \leq \frac{1- \delta}{\alpha}$ for some $\delta \in (0,1)$ (the quantitative variant of~\eqref{eq:esssup}), the full dissipation quantity $D(t)$ controls $D_1(t)$.

\begin{proposition}
    \label{pro:rodrelationship}
Suppose that $(\X_s,\kappa_3)$ satisfies $|\X_s|^2 = 1$ and $\| \X_{ss} \|_{L^\infty_s}^2 \leq \frac{1-\delta}{\alpha}$ for some $\delta \in (0,1)$.

In the free end setting, the dissipation $D$ and energy $E$ bound $D_1$ and $\X_{ssss}$ as follows:
\begin{equation}\label{eq:Xssss_free2}
\begin{aligned}
D_1(t) &\lesssim_\delta D(t)(1+E(t)^3) + E(t)^5 \,,\\
\norm{\X_{ssss}}_{L^2_s} &\lesssim_\delta D(t)^{1/2}(1+E(t)^2)+E(t)^{5/2} +E(t)^3\,.
\end{aligned}
\end{equation}

In the closed loop setting, $D$ and $E$ bound $D_1$ and $\X_{ssss}$ as:
\begin{equation}\label{eq:Xssss_closed2}
\begin{aligned}
D_1(t) &\lesssim_\delta D(t)(1+E(t)^3) +  E(t)^2 + E(t)^5 \,, \\
\norm{\X_{ssss}}_{L^2_s}  &\lesssim_\delta  D(t)^{1/2}(1+E(t)^2) + E(t)+E(t)^3\,.
\end{aligned}
\end{equation}
\end{proposition}
\begin{proof}
We may write
\begin{equation}
\begin{aligned}
D(t) &= \int_{I}\bigg(\big|\wt{\bm{Z}}_s\big|^2 +\gamma(\X_s\cdot\wt{\bm{Z}}_s)^2 +\frac{\eta^2}{\alpha} (\kappa_3)_s^2\bigg)\,ds\\
&= \int_{I}\bigg(\big|\bm{Z}_s-\eta\big(\kappa_3\X_s\times\X_{ss}\big)_s\big|^2 +\gamma(\X_s\cdot\bm{Z}_s)^2 +\frac{\eta^2}{\alpha} (\kappa_3)_s^2\bigg)\,ds\\
&= D_1(t) + \eta^2\int_{I}\abs{\big(\kappa_3\X_s\times\X_{ss} \big)_s}^2\,ds - 2\eta\int_{I}
    \bm{Z}_s\cdot\big(\kappa_3\X_s\times\X_{ss} \big)_s  \,ds\,,
\end{aligned}
\end{equation}
so that
\begin{equation}\label{eq:D1est}
D_1(t) + \eta^2\int_{I}\abs{\big(\kappa_3\X_s\times\X_{ss} \big)_s}^2\,ds = D(t) + 2\eta\int_{I}
    \bm{Z}_s\cdot\big(\kappa_3\X_s\times\X_{ss} \big)_s  \,ds\,.
\end{equation}
To estimate the right-hand side cross term, we first write 
\begin{equation}\label{eq:divvy}
    2\eta\abs{\bm{Z}_s\cdot\big(\kappa_3\X_s\times\X_{ss} \big)_s}
    \le 2\eta\abs{\bm{Z}_s\cdot(\X_s\times\X_{ss})(\kappa_3)_s}+ 2\eta\abs{\bm{Z}_s\cdot(\X_s\times\X_{sss})\kappa_3}\,.
\end{equation}
We rely on the curvature condition \eqref{eq:esssup} to estimate the first term in \eqref{eq:divvy} as 
\begin{equation}
    2\abs{\bm{Z}_s}\abs{\X_{ss}}\eta\abs{(\kappa_3)_s} \le 2\abs{\bm{Z}_s}(1-\delta)^{1/2}\frac{\eta}{\alpha^{1/2}}\abs{(\kappa_3)_s}
    \le \varepsilon\abs{\bm{Z}_s}^2 + \frac{1-\delta}{\varepsilon}\frac{\eta^2}{\alpha}\abs{(\kappa_3)_s}^2
\end{equation}
for $\varepsilon>0$. Taking $\varepsilon=1-\frac{\delta}{2}$ yields 
\begin{equation}
    2\eta\big|\bm{Z}_s\cdot(\X_s\times\X_{ss})(\kappa_3)_s\big| \le \bigg(1-\frac{\delta}{2}\bigg)\abs{\bm{Z}_s}^2 + \frac{1-\delta}{1-\frac{\delta}{2}}\frac{\eta^2}{\alpha}\abs{(\kappa_3)_s}^2\,;
\end{equation}
in particular, both terms may be absorbed into the left-hand side $D_1(t)$ in \eqref{eq:D1est}.

We split the second term in \eqref{eq:divvy} as
\begin{equation}
    2\eta\big|\bm{Z}_s\cdot(\X_s\times\X_{sss})\kappa_3\big|\le \varepsilon\abs{\bm{Z}_s}^2 + \frac{\eta^2}{\varepsilon}\abs{\kappa_3}\abs{\X_{sss}}\,,
\end{equation}
and for $\varepsilon$ sufficiently small (depending on $\delta$), the $\bm{Z}_s$ term may also be absorbed into the left-hand side $D_1(t)$. 
%
% Splitting the right-hand side cross term above as
% \begin{equation}\label{eq:divvy}
%     \abs{2\eta\bm{Z}_s\cdot\big(\kappa_3\X_s\times\X_{ss} \big)_s}
%     \le \frac{1}{2}\abs{\bm{Z}_s}^2 + 2\eta^2\abs{\big(\kappa_3\X_s\times\X_{ss} \big)_s}^2\,,
% \end{equation}
% we may absorb $\frac{1}{2}\abs{\bm{Z}_s}^2 + \eta^2\abs{\big(\kappa_3 \X_s\times\X_{ss} \big)_s}^2$ into the left-hand side terms. The remaining factor of $\eta^2\abs{\big(\kappa_3\X_s\times\X_{ss} \big)_s}^2$ may be estimated as 
% \begin{equation}\label{eq:first_term}
%     \eta^2\abs{\big(\kappa_3\X_s\times\X_{ss} \big)_s}^2 \le \eta^2\big(\abs{(\kappa_3)_s}^2\abs{\X_{ss}}^2+\abs{\kappa_3}^2\abs{\X_{sss}}^2 \big)\,.
% \end{equation}
% Here we will need to use the condition \eqref{eq:Xss_criterion}, $\abs{\X_{ss}}^2\le \frac{(1-\delta)}{\alpha}$, in order to absorb the first term of \eqref{eq:first_term} into $D_1(t)$. So far, we have
% \begin{equation}
% \frac{1}{2}D_1(t) \le D(t) + \eta^2\int_{I}\abs{\kappa_3}^2\abs{\X_{sss}}^2\,ds \,.
% \end{equation}
To estimate the remainder, we consider the free end and closed loop cases separately.

\emph{Free ends}. In the free end setting, using the interpolation inequalities 
\begin{equation}\label{eq:interp}
    \norm{\kappa_3}_{L^\infty_s}^2 \lesssim \norm{\kappa_3}_{L^2_s}\norm{(\kappa_3)_s}_{L^2_s}\,, \quad 
    \norm{\X_{sss}}_{L^2_s}^2 \lesssim \norm{\X_{ss}}_{L^2_s}\norm{\X_{ssss}}_{L^2_s}\,,
\end{equation}
along with the bound \eqref{eq:Xssss_free}, we have
\begin{equation}
\begin{aligned}
    \int_0^1\abs{\kappa_3}^2\abs{\X_{sss}}^2\,ds &\le \norm{\kappa_3}_{L^\infty_s}^2\norm{\X_{sss}}_{L^2_s}^2  \\
    &\lesssim \norm{\kappa_3}_{L^2_s}\norm{(\kappa_3)_s}_{L^2_s} \norm{\X_{ss}}_{L^2_s}\norm{\X_{ssss}}_{L^2_s} \\
    &\lesssim E(t) D(t)^{1/2}\big(D_1(t)^{1/2}(1+E(t)^{1/2}) + E(t)^{5/2}  \big) \\
    &\le \varepsilon D_1(t) + C(\varepsilon)\big( E(t)^2 D(t)(1+E(t)) +  E(t)^5\big)\,.
\end{aligned}
\end{equation}
Taking $\varepsilon>0$ sufficiently small (again depending on $\delta$), we obtain~\eqref{eq:Xssss_free2}.

\emph{Closed loop}.
For the closed filament, the $\X_{sss}$ interpolation inequality \eqref{eq:interp} still holds. However, the interpolation inequality for $\kappa_3$ now requires a full $H^1_s$ norm on the right-hand side. Using \eqref{eq:Xssss_closed}, we may then estimate
\begin{equation}
\begin{aligned}
    \int_\T\abs{\kappa_3}^2\abs{\X_{sss}}^2\,ds &\le \norm{\kappa_3}_{L^\infty_s}^2\norm{\X_{sss}}_{L^2_s}^2  \\
    &\lesssim \norm{\kappa_3}_{L^2_s}(\norm{\kappa_3}_{L^2_s}+ \norm{(\kappa_3)_s}_{L^2_s})\norm{\X_{ss}}_{L^2_s}\norm{\X_{ssss}}_{L^2_s} \\
    &\lesssim  E(t)^{3/2} \big(D_1(t)^{1/2}(1+E(t)^{1/2}) + E(t)+ E(t)^{5/2}  \big) \\
    &\quad +  E(t) D(t)^{1/2}\big(D_1(t)^{1/2}(1+E(t)^{1/2}) + E(t)+ E(t)^{5/2}  \big)\\
    &\le  \varepsilon D_1(t) +  C(\varepsilon)\big(E(t)^2 D(t)(1+E(t)) + E(t)^2 + E(t)^5\big)  
\end{aligned}
\end{equation}
for any $\varepsilon>0$. Again taking $\varepsilon$ sufficiently small, we obtain~\eqref{eq:Xssss_closed2}.
\end{proof}

We must next estimate the tension $\lambda$ in terms of $\X_{ss}$ and $\kappa_3$.

\begin{lemma}[Tension estimates]\label{lem:tension_rod}
The tension equation~\eqref{eq:tension_0-repeat} is uniquely solvable in $H^1_0(I)$, and its solution $\lambda$ satisfies
\begin{equation}\label{eq:H1_lam}
    %\norm{\lambda}_{H^1_s} \lesssim \norm{\X_s}_{H^2_s}^2(1 + \norm{\kappa_3}_{H^1_s})\,,
    \norm{\lambda}_{H^1_s} \lesssim \norm{\X_{ss}}_{L^2_s}\norm{\X_{ssss}}_{L^2_s}+ \norm{\kappa_3}_{L^2_s}\norm{\X_{ss}}_{L^2_s}^{5/4}\norm{\X_{ssss}}_{L^2_s}^{3/4}\,,
\end{equation}
% \begin{equation}\label{eq:H2_lam}
%     \norm{\lambda}_{H^2_s} 
%     \lesssim (1+\norm{\X_{ss}}_{L^2_s}^2)\big(\norm{\X_{ss}}_{L^2_s}^{3/4}\norm{\X_{ssss}}_{L^2_s}^{5/4} + \norm{\kappa_3}_{L^2_s}\norm{\X_{ss}}_{L^2_s}\norm{\X_{ssss}}_{L^2_s}\big)
%     %\lesssim (\norm{\X_s}_{H^4_s}^{1/2}\norm{\X_s}_{H^2_s}^{3/2} + \norm{\X_s}_{H^2_s}^4)(1 + \norm{\kappa_3}_{H^1_s}) \, ,
% \end{equation}
provided that the right-hand side of~\eqref{eq:H1_lam} is finite.

In addition, given two curve-frame pairs $(\X_s^{(1)},\kappa_3^{(1)})$, $(\X_s^{(2)},\kappa_3^{(2)})$, let $\lambda^{(1)}$, $\lambda^{(2)}$ denote the corresponding tensions. The following difference bound holds:
\begin{equation}\label{eq:lamdiff_H1}
\begin{aligned}
    \norm{\lambda^{(1)}-\lambda^{(2)}}_{H^1_s} 
    &\lesssim \big(\norm{\X_{s}^{(1)}}_{H^2_s} +\norm{\X_{s}^{(2)}}_{H^2_s}\big)\times\\
    &\quad \times\big(1 + \norm{\X_{s}^{(1)}}_{H^2_s}^2\big) \big(1+ \|\kappa_3^{(1)}\|_{H^1_s} +\|\kappa_3^{(2)}\|_{H^1_s}\big)\norm{\X_{s}^{(1)}-\X_{s}^{(2)}}_{H^2_s} \\
    &\quad + \norm{\X_{s}^{(1)}}_{H^2_s}\norm{\X_{s}^{(1)}}_{H^3_s}\|\kappa_3^{(1)}- \kappa_3^{(2)}\|_{L^2_s} \,,
\end{aligned}
\end{equation}
provided that the right-hand side of~\eqref{eq:lamdiff_H1} is finite.
\end{lemma}

\begin{proof}
Estimate~\eqref{eq:H1_lam} follows from analogous arguments to Lemma~\ref{lem:tension} using the form of \eqref{eq:tension_0-repeat}. 
%We then obtain the $H^2_s$ bound~\eqref{eq:H2_lam} directly from the equation \eqref{eq:tension_0-repeat}.
To show the difference estimate \eqref{eq:lamdiff_H1}, as in the curve-only setting (see equation \eqref{eq:Rlambdak}), we let $R_\lambda^{(k)}$ denote the right-hand side of the tension equation \eqref{eq:tension_0-repeat} for curve-frame pair $(\X_s^{(k)},\kappa_3^{(k)})$, $k=1,2$. We may estimate the difference  
\begin{equation}\label{eq:Rlam_diff}
\begin{aligned}
    \norm{R_\lambda^{(1)}- R_\lambda^{(2)}}_{(H^1)^*} &\lesssim  \norm{\X_{s}^{(1)}-\X_{s}^{(2)}}_{H^2_s} \big(1 + (1+ \norm{\X_{s}^{(1)}}_{H^2_s}) \|\kappa_3^{(2)}\|_{H^1_s} \big)\times  \\ 
    &\quad 
    \times \big(\norm{\X_{s}^{(1)}}_{H^2_s} +\norm{\X_{s}^{(2)}}_{H^2_s}\big)
    + \norm{\X_{s}^{(1)}}_{H^2_s}\norm{\X_{s}^{(1)}}_{H^3_s}\|\kappa_3^{(1)}- \kappa_3^{(2)}\|_{L^2_s} \,.
\end{aligned}
\end{equation}
Defining $q^{(1)}$, $q^{(2)}$ as in \eqref{eq:qdef}, we may obtain the analogous estimate
\begin{equation}\label{eq:q_diff}
\begin{aligned}
    \norm{q^{(1)}-q^{(2)}}_{H^1_s}&\lesssim \big(\norm{\X_{s}^{(1)}}_{H^2_s} +\norm{\X_{s}^{(2)}}_{H^2_s}\big)\times\\
    &\quad \times\norm{\X_{s}^{(1)}-\X_{s}^{(2)}}_{H^2_s}
    \norm{\X_{s}^{(1)}}_{H^2_s}^2(1 + \|\kappa_3^{(1)}\|_{H^1_s})\,.
\end{aligned}
\end{equation}
From \eqref{eq:Rlam_diff} and \eqref{eq:q_diff}, we obtain \eqref{eq:lamdiff_H1} by the same arguments as in \eqref{eq:lam_strat}.  
\end{proof}

%%%%%%%%%%%%%%%%%%%%%%%%%%%%%%%%%%%%
%%%%%%%%%%%%%%%%%%%%%%%%%%%%%%%%%%%%
%%%%%%%%%%%%%%%%%%%%%%%%%%%%%%%%%%%%
%%%%%%%%%%%%%%%%%%%%%%%%%%%%%%%%%%%%
%%%%%%%%%%%%%%%%%%%%%%%%%%%%%%%%%%%%
%%%%%%%%%%%%%%%%%%%%%%%%%%%%%%%%%%%%
%%%%%%%%%%%%%%%%%%%%%%%%%%%%%%%%%%%%

\subsubsection{Propagation estimates}\label{subsec:propagation_estimates}

Equipped with the bounds \eqref{eq:Xssss_free2} and \eqref{eq:Xssss_closed2} on $\X_{ssss}$, we may now turn to estimating $D(t)$.

\begin{proposition}
	\label{prop:propagationproprod}
Let $(\X_s,\kappa_3)$ be a strong solution on $[0,T]$ satisfying
\begin{equation}
\sup_{t \in [0,T]} \| \X_{ss} \|_{L^\infty}^2 \leq \frac{1-\delta}{\alpha} \, .
\end{equation}
Then, in the closed loop setting, 
\begin{equation}
    \label{eq:rodpropagationperiodic}
\begin{aligned}
\frac{1}{2}\frac{d}{dt} D(t) &\le -\frac{1}{2}\int_\T\abs{\p_t\X_{ss}}^2\,ds  
-\eta \int_\T \big(\p_t\X_s\cdot(\X_s\times\X_{ss})\big)^2\,ds - \frac{\eta^3}{2\alpha^2}\int_\T (\kappa_3)_{ss}^2\,ds\\
& \qquad + C(E,\delta) D(t)^2 + C(E,\delta) D(t)\,.
\end{aligned}
\end{equation}
In the free end setting,
\begin{equation}
    \label{eq:rodpropagationfreeend}
\begin{aligned}
\frac{1}{2}\frac{d}{dt} D(t) &\le -\frac{1}{2}\int_0^1\abs{\p_t\X_{ss}}^2\,ds  
-\eta \int_0^1 \big(\p_t\X_s\cdot(\X_s\times\X_{ss})\big)^2\,ds - \frac{\eta^3}{2\alpha^2}\int_0^1 (\kappa_3)_{ss}^2\,ds\\
& \qquad + e^{-c_0t}C(E^{\rm in},\delta)\big(D(t)^2 + D(t)\big) + C(\delta)(D(t) +D(t)^2)\,.
\end{aligned}
\end{equation}
\end{proposition}

\begin{proof}
Differentiating $D(t)$ in time, the following identity holds and can be justified for strong solutions:
 \begin{equation}
\begin{aligned}
\frac{1}{2}\frac{d}{dt} D(t) = \frac{1}{2}\frac{d}{dt}\int_{I}&\bigg(\big|\wt{\bm{Z}}_s\big|^2 +\gamma(\X_s\cdot\wt{\bm{Z}}_s)^2 +\frac{\eta^2}{\alpha} (\kappa_3)_s^2\bigg)\,ds = J_1+J_2+J_3\,,\\
J_1&= \int_{I} \big(({\bf I}+\gamma\X_s\otimes\X_s)\wt{\bm{Z}}_s\big)\cdot\p_t\wt{\bm{Z}}_s\,ds\,, \\
J_2&= \gamma\int_{I} (\X_s\cdot \wt{\bm{Z}}_s)(\p_t\X_s\cdot \wt{\bm{Z}}_s)\,ds\,, \\
J_3&= \frac{\eta^2}{\alpha}\int_{I}(\kappa_3)_s\p_t(\kappa_3)_s\,ds\,.
\end{aligned}
\end{equation}
We have
\begin{equation}
\begin{aligned}
J_1 &= - \int_{I} \big(({\bf I}+\gamma\X_s\otimes\X_s)\wt{\bm{Z}}_s\big)_s\cdot\p_t\wt{\bm{Z}}\,ds 
= \int_{I} \p_t\X_s\cdot\p_t\wt{\bm{Z}}\,ds\\
&= \int_{I} \p_t\X_s\cdot(\p_t\X_{sss} - \lambda\p_t\X_s - \eta \p_t\kappa_3\X_s\times\X_{ss}-\eta\kappa_3\X_s\times\p_t\X_{ss})\,ds\,,
\end{aligned}
\end{equation}
where we have used that $\p_t\wt{\bm{Z}}\big|_{s=0,1}=0$ in the free end setting and $\X_s\cdot\p_t\X_s=0$ for both free ends and the closed loop. Upon integrating by parts and replacing $\p_t\kappa_3$ by \eqref{eq:kap3_0-repeat}, we obtain 
\begin{equation}\label{eq:J1}
\begin{aligned}
J_1 &=  -\int_{I}\abs{\p_t\X_{ss}}^2\,ds  
-\eta \int_{I} \big(\p_t\X_s\cdot(\X_s\times\X_{ss})\big)^2\,ds - \int_{I}\lambda\abs{\p_t\X_s}^2\,ds\\
&\qquad  - \frac{\eta^2}{\alpha}\int_{I}(\kappa_3)_{ss}\p_t\X_s\cdot(\X_s\times\X_{ss})\,ds
-\eta \int_{I} \kappa_3\p_t\X_s\cdot(\X_s\times\p_t\X_{ss})\,ds\,.
\end{aligned}
\end{equation}
The first two terms in \eqref{eq:J1} have a favorable sign, but we will need to estimate the remaining three terms. As in the curve-only setting, we estimate
\begin{equation}
\abs{\int_{I}\lambda\abs{\p_t\X_s}^2\,ds} \leq \norm{\lambda}_{L^\infty_s}\norm{\p_t\X_s}_{L^2_s}^2\,,
\end{equation}
where $\norm{\lambda}_{L^\infty_s}$ may be estimated using the $H^1_s$ bound \eqref{eq:H1_lam} of Lemma \ref{lem:tension_rod}.
% \begin{equation}\label{eq:lambdaH1}
% \begin{aligned}
% \norm{\lambda}_{H^1_s} &\lesssim \norm{\X_{sss}}_{L^2_s}^2+ \norm{\kappa_3}_{L^2_s}\norm{\X_{ss}}_{L^2_s}^{1/2}\norm{\X_{sss}}_{L^2_s}^{3/2}\\
% &\lesssim \norm{\X_{ss}}_{L^2_s}\norm{\X_{ssss}}_{L^2_s}+ \norm{\kappa_3}_{L^2_s}\norm{\X_{ss}}_{L^2_s}^{5/4}\norm{\X_{ssss}}_{L^2_s}^{3/4} \,,
% \end{aligned}
% \end{equation}
% where we are using a more refined version of the tension estimate \eqref{eq:H1_lam} as well as the interpolation inequality \eqref{eq:interp}. 

Again following the curve-only setting, we now consider the closed loop and free end filaments separately. 

%%%%%%%%%%%%%%%
\emph{Closed loop}.
We begin by noting that, as in \eqref{eq:ptXs_interp0}, the following interpolation inequality holds for $\norm{\p_t\X_s}_{L^2_s}^2$: 
\begin{equation}\label{eq:ptXs_interp}
\norm{\p_t\X_s}_{L^2_s}^2 \lesssim \norm{\wt{\bm{Z}}_s}_{L^2_s}\norm{\p_t\X_{ss}}_{L^2_s}\,,
\end{equation}
where we have used that $\norm{\p_t\X}_{L^2_s}\lesssim \norm{\wt{\bm{Z}}_s}_{L^2_s}$. Combining \eqref{eq:H1_lam} and \eqref{eq:ptXs_interp}, we obtain 
\begin{equation}
\begin{aligned}
\abs{\int_\T\lambda\abs{\p_t\X_s}^2\,ds} &\le \varepsilon\norm{\p_t\X_{ss}}_{L^2_s}^2 + C(\varepsilon)\,\norm{\X_{ss}}_{L^2_s}^2\norm{\X_{ssss}}_{L^2_s}^2\norm{\wt{\bm{Z}}_s}_{L^2_s}^2 \\
&\qquad + C(\varepsilon)\,\norm{\kappa_3}_{L^2_s}^2\norm{\X_{ss}}_{L^2_s}^{5/2}\norm{\X_{ssss}}_{L^2_s}^{3/2}\norm{\wt{\bm{Z}}_s}_{L^2_s}^2
\end{aligned}
\end{equation}
for any $\varepsilon>0$. As before, the first term above will be absorbed into the signed first term of $J_1$, while the second term above may be estimated as in \eqref{eq:Es_and_Ds}, using \eqref{eq:Xssss_closed2} to obtain
\begin{equation}\label{eq:J1bd1}
 \norm{\X_{ss}}_{L^2_s}^2\norm{\X_{ssss}}_{L^2_s}^2\norm{\wt{\bm{Z}}_s}_{L^2_s}^2 \le
 C(E,\delta)D(t)\norm{\wt{\bm{Z}}_s}_{L^2_s}^2 + C(E,\delta)\norm{\wt{\bm{Z}}_s}_{L^2_s}^2\,.
\end{equation}
For the third term, again using \eqref{eq:Xssss_closed2}, we have 
\begin{equation}\label{eq:J1bd2}
\begin{aligned}
\norm{\kappa_3}_{L^2_s}^2\norm{\X_{ss}}_{L^2_s}^{5/2}\norm{\X_{ssss}}_{L^2_s}^{3/2}\norm{\wt{\bm{Z}}_s}_{L^2_s}^2 
&\lesssim E(t)^{9/4}\norm{\X_{ssss}}_{L^2_s}^{3/2}\norm{\wt{\bm{Z}}_s}_{L^2_s}^2 \\
&\leq C(E,\delta)D(t)^{7/4}+ C(E,\delta)D(t)\,.
\end{aligned}
\end{equation}
We next turn to the fourth term of $J_1$. In the periodic setting, we have
\begin{equation}\label{eq:kap3ss_1}
\begin{aligned}
\abs{\frac{\eta^2}{\alpha}\int_{\T}(\kappa_3)_{ss}\p_t\X_s\cdot(\X_s\times\X_{ss})\,ds} &\leq \frac{\eta^2}{\alpha}\norm{(\kappa_3)_{ss}}_{L^2_s} \norm{\p_t\X_s}_{L^\infty_s}\norm{\X_{ss}}_{L^2_s}\\
&\leq \varepsilon \frac{\eta^3}{\alpha^2}\norm{(\kappa_3)_{ss}}_{L^2_s}^2  + C(\varepsilon) \norm{\p_t\X_s}_{L^\infty_s}^2\norm{\X_{ss}}_{L^2_s}^2\,,
\end{aligned}
\end{equation}
where the $(\kappa_3)_{ss}$ will later be absorbed into a coercive term $\frac{\eta^3}{\alpha^2}\norm{(\kappa_3)_{ss}}_{L^2_s}^2$ arising from $J_3$. The remaining term satisfies 
\begin{equation}\label{eq:kap3ss_2}
\begin{aligned}
\norm{\p_t\X_s}_{L^\infty_s}^2\norm{\X_{ss}}_{L^2_s}^2 &\lesssim 
\norm{\wt{\bm{Z}}_s}_{L^2_s}^{1/2}\norm{\p_t\X_{ss}}_{L^2_s}^{3/2}\norm{\X_{ss}}_{L^2_s}^2 \\
&\le \varepsilon\norm{\p_t\X_{ss}}_{L^2_s}^2 + C(\varepsilon)\,\norm{\wt{\bm{Z}}_s}_{L^2_s}^2\norm{\X_{ss}}_{L^2_s}^8\\
&\le
\varepsilon\norm{\p_t\X_{ss}}_{L^2_s}^2 + C(\varepsilon,\delta)\,D(t) E(t)^4
\end{aligned}
\end{equation}
for any $\varepsilon>0$.
We may estimate the final term in $J_1$ as 
\begin{equation}\label{eq:J1bd3}
\begin{aligned}
\abs{\eta\int_\T \kappa_3\p_t\X_s\cdot(\X_s\times\p_t\X_{ss})\,ds} &\lesssim \norm{\kappa_3}_{L^2_s}\norm{\p_t\X_s}_{L^\infty_s}\norm{\p_t\X_{ss}}_{L^2_s}\\
&\lesssim 
\norm{\kappa_3}_{L^2_s}\norm{\wt{\bm{Z}}_s}_{L^2_s}^{1/4}\norm{\p_t\X_{ss}}_{L^2_s}^{7/4} \\
&\le 
C(\varepsilon)\,\norm{\kappa_3}_{L^2_s}^8\norm{\wt{\bm{Z}}_s}_{L^2_s}^2 + \varepsilon\norm{\p_t\X_{ss}}_{L^2_s}^2\\
&\le C(\varepsilon,\delta)\,E(t)^4D(t) + \varepsilon\norm{\p_t\X_{ss}}_{L^2_s}^2
\end{aligned}
\end{equation}
for any $\varepsilon>0$, where we have used an interpolation inequality for $\p_t\X_s$ and that $\norm{\p_t\X}_{L^2_s} \lesssim \norm{\wt{\bm{Z}}_s}_{L^2_s}$. 
In total, for the closed filament, we may obtain the $J_1$ bound
\begin{equation}\label{eq:J1bound}
\abs{J_1} \le \frac{1}{4} \norm{\p_t\X_{ss}}_{L^2_s}^2 + \frac{\eta^3}{4\alpha^2} \norm{(\kappa_3)_{ss}}_{L^2_s}^2 + C(E,\delta) D(t)^2 + C(E,\delta) D(t)\,.
\end{equation}

The term $J_2$ may be estimated exactly as in the curve-only setting. For the periodic filament, we have
\begin{equation}
\abs{J_2}\lesssim \norm{\p_t\X_s}_{L^\infty_s}\norm{\bm{Z}_s}_{L^2_s}^2
\le \varepsilon\norm{\p_t\X_{ss}}_{L^2_s}^2 + C(\varepsilon)\,\norm{\bm{Z}_s}_{L^2_s}^4
\le \varepsilon\norm{\p_t\X_{ss}}_{L^2_s}^2 + C(\varepsilon,\delta)\,D(t)^2
\end{equation}
for any $\varepsilon>0$. 

To estimate $J_3$, we observe that, upon differentiating~\eqref{eq:kap3_0-repeat}, we have
\begin{equation}\label{eq:kap3_s}
\p_t(\kappa_3)_s = \big(\p_t\X_s\cdot(\X_s\times\X_{ss})\big)_s + \frac{\eta}{\alpha}(\kappa_3)_{sss}\,.
\end{equation}
We thus rewrite $J_3$ as 
\begin{equation}\label{eq:J3}
\begin{aligned}
J_3 &= \frac{\eta^2}{\alpha}\int_\T(\kappa_3)_s\bigg( \big(\p_t\X_s\cdot(\X_s\times\X_{ss})\big)_s + \frac{\eta}{\alpha}(\kappa_3)_{sss}\bigg)\,ds\\
&= - \frac{\eta^3}{\alpha^2}\int_\T (\kappa_3)_{ss}^2\,ds - \frac{\eta^2}{\alpha}\int_\T(\kappa_3)_{ss}\p_t\X_s\cdot(\X_s\times\X_{ss})\,ds\,.
\end{aligned}
\end{equation}
The first term is coercive, while the second term has already been estimated in \eqref{eq:kap3ss_1} and \eqref{eq:kap3ss_2}. Note from \eqref{eq:J1} that these terms appear with the same sign and do not cancel.

Combining the estimates for $J_1$, $J_2$, and $J_3$ for the periodic filament, we obtain~\eqref{eq:rodpropagationperiodic}.
% \begin{equation}
% \begin{aligned}
% \frac{1}{2}\frac{d}{dt} D(t) &\le -\frac{1}{2}\int_\T\abs{\p_t\X_{ss}}^2\,ds  
% -\eta \int_\T \big(\p_t\X_s\cdot(\X_s\times\X_{ss})\big)^2\,ds - \frac{\eta^3}{2\alpha^2}\int_\T (\kappa_3)_{ss}^2\,ds\\
% & \qquad + C(E^{\rm in}) D(t)^2 + C(E^{\rm in}) D(t)\,.
% \end{aligned}
% \end{equation}

%%%%%%%%%%%%%%
\emph{Free ends}.
In the free end setting, the same type of bounds \eqref{eq:J1bd1} and \eqref{eq:J1bd2} hold for the third term of $J_1$, except that we may now build in the exponential decay \eqref{eq:E_decay} of $E(t)$ to obtain 
\begin{equation}
\begin{aligned}
\abs{\int_0^1\lambda\abs{\p_t\X_s}^2\,ds} &\le \varepsilon\norm{\p_t\X_{ss}}_{L^2_s}^2 + e^{-c_0t}C(E^{\rm in},\varepsilon,\delta)\big(D(t)^2 + D(t)\big)\\ 
&\qquad + C(\varepsilon,\delta)(D(t) +D(t)^2)\,.
\end{aligned} 
\end{equation}
Here the additional $D(t) +D(t)^2$ terms come from \eqref{eq:Xssss_free2}.

For the fourth term of $J_1$, instead of \eqref{eq:kap3ss_2}, we need to include lower order terms in the $\p_t\X_s$ interpolation inequality:
\begin{equation}\label{eq:twice}
\begin{aligned}
\norm{\p_t\X_s}_{L^\infty_s}^2\norm{\X_{ss}}_{L^2_s}^2 &\lesssim \norm{\wt{\bm{Z}}_s}_{L^2_s}^{1/2}\norm{\p_t\X}_{H^2_s}^{3/2}\norm{\X_{ss}}_{L^2_s}^2\\
&\lesssim \big( \norm{\wt{\bm{Z}}_s}_{L^2_s}^2 + \norm{\wt{\bm{Z}}_s}_{L^2_s}^{1/2}\norm{\p_t\X_{ss}}_{L^2_s}^{3/2}\big)\norm{\X_{ss}}_{L^2_s}^2 \\
&\le \varepsilon\norm{\p_t\X_{ss}}_{L^2_s}^2 + e^{-c_0t}C(E^{\rm in},\varepsilon)\norm{\wt{\bm{Z}}_s}_{L^2_s}^2
\end{aligned}
\end{equation}
for any $\varepsilon>0$. Similarly, for the fifth term of $J_1$, we must replace \eqref{eq:J1bd3} with
\begin{equation}
\begin{aligned}
\norm{\kappa_3}_{L^2_s}\norm{\p_t\X_s}_{L^\infty_s}\norm{\p_t\X_{ss}}_{L^2_s}
&\lesssim 
\norm{\kappa_3}_{L^2_s}\norm{\wt{\bm{Z}}_s}_{L^2_s}^{1/4}\norm{\p_t\X}_{H^2_s}^{7/4} \\
&\lesssim 
\norm{\kappa_3}_{L^2_s}(\norm{\wt{\bm{Z}}_s}_{L^2_s}^2+ \norm{\wt{\bm{Z}}_s}_{L^2_s}^{1/4}\norm{\p_t\X_{ss}}_{L^2_s}^{7/4}) \\
&\le 
\varepsilon\norm{\p_t\X_{ss}}_{L^2_s}^2 + e^{-c_0t}C(E^{\rm in},\varepsilon)\norm{\wt{\bm{Z}}_s}_{L^2_s}^2 
\end{aligned}
\end{equation}
for any $\varepsilon>0$.
In total, for the free-ended filament, we may bound
\begin{equation}
\begin{aligned}
\abs{J_1} &\le \frac{1}{4}\norm{\p_t\X_{ss}}_{L^2_s}^2 + \frac{\eta^3}{4\alpha^2}\norm{(\kappa_3)_{ss}}_{L^2_s}^2 \\
&\quad+ e^{-c_0t}C(E^{\rm in},\delta)\big(D(t)^2 + D(t)\big) + C(\delta)(D(t) +D(t)^2)\,.
\end{aligned}
\end{equation}

The term $J_2$ may again be estimated as in the curve-only setting \eqref{eq:free_est1}, \eqref{eq:free_est2}, yielding
\begin{equation}
\abs{J_2} \le \varepsilon\norm{\p_t\X_{ss}}_{L^2_s}^2 + C(\varepsilon,\delta)(D(t) +D(t)^2)
\end{equation}
for $\varepsilon>0$. The term $J_3$ is again given by \eqref{eq:J3}, and the second term in the free end setting may be estimated as in \eqref{eq:twice}. 

Combining the free end estimates for $J_1$, $J_2$, and $J_3$, we have~\eqref{eq:rodpropagationfreeend}.
\end{proof}

\subsection{Conditional GWP and convergence to equilibrium}\label{subsec:GWP_rodframe}

\subsubsection{Conditional GWP}

We now demonstrate Theorem \ref{thm:GWP_curveframe} that, under certain conditions, the local-in-time solution $(\X_s,\kappa_3)$ to \eqref{eq:Xss_0}-\eqref{eq:kap3_0} may be continued to a global solution. %, we must guarantee that the bound \eqref{eq:Xss_criterion} on $\norm{\X_{ss}}_{L^\infty_s}$ continues to be satisfied for all time.

We begin by summarizing the consequences of Proposition~\ref{prop:propagationproprod}. Namely, by Gr{\"o}nwall's inequality, we have the \emph{a priori} estimate
\begin{equation}
    \label{eq:aprioriestwithmyd}
D(t) \leq e^{C(E^{\rm in},\delta)} [D^{\rm in} + C(E^{\rm in},\delta)] \, ,
\end{equation}
valid for a strong solution on a time interval $[0,T]$ in both free end and closed loop settings. To prove that certain strong solutions can be continued indefinitely, we should leverage the above estimate to propagate the condition $\| \X_{ss} \|_{L^\infty_s}^2 \leq \frac{1-\delta}{\alpha}$; we choose $\delta = \frac{1}{2}$.

To propagate this condition, we use the Sobolev embedding
\begin{equation}\label{eq:Xss_Xssss_embedding}
\norm{\X_{ss}}_{L^\infty_s}\lesssim \norm{\X_{ssss}}_{L^2_s}
\end{equation}
for both the free ended filament (due to the homogeneous boundary conditions) and the closed loop (since $\X_{ss}$ has zero mean). Recall that $\| \X_{ssss} \|_{L^2_s}^2$ is controlled in Proposition~\ref{pro:rodrelationship}, which we combine with~\eqref{eq:aprioriestwithmyd} and $E(t) \leq E^{\rm in}$ to obtain
\begin{equation}
\| \X_{ss} \|_{L^\infty_s}^2 \leq C e^{C(E^{\rm in})} (D^{\rm in} + 1) \, .
\end{equation}
The right-hand side is bounded by $\frac{1}{2\alpha}$ provided
\begin{equation}
\| \X_{ssss} \|_{L^2_s} \leq \frac{1}{C} \log \frac{1}{C\alpha} \, .
\end{equation}

\subsubsection{Convergence to equilibrium}
Suppose that $(\X_s,\kappa_3)$ is a global-in-time strong solution  satisfying, additionally,
\begin{equation}
\sup_{t \in [0,+\infty)} \| \X_{ss} \|_{L^\infty_s}^2 < \frac{1}{\alpha} \, .
\end{equation}

\emph{Free end}. Following the arguments in Section~\ref{sec:decayimmersednoforcing} (which rely only on the differential inequalities for $E$ and $D$), we can demonstrate in the free end setting that both $E$ and $D$ converge to zero exponentially as $t \to +\infty$ and, from~\eqref{eq:rodpropagationfreeend}, that $\| \p_t \X_{ss} \|_{L^2_s}$ has time-integrated exponential decay. From here, the main point is to demonstrate that the curve $\X$ and frame $\e$ themselves are stabilizing. We demonstrate that $\X$ is stabilizing in $L^2_s$ using~\eqref{eq:followmeforstabilizing}. Since
\begin{equation}
    \label{eq:sinceframe}
\p_s \e = (\X_{ss} \cdot \e) \X_s + \kappa_3 \X_s \times \e \text{ and } |\e|^2 = 1 \, ,
\end{equation}
we further obtain that $\p_s \e$ is converging to zero exponentially in $L^2_s$. Finally, we prove that $\e$ is stabilizing in $L^2_s$ by time-integrating
\begin{equation}
\p_t \e = - (\p_t \X_s \cdot \e)\X_s + \frac{\eta}{\alpha} (\kappa_3)_s \, ,
\end{equation}
 as we did for the curve equation in~\eqref{eq:followmeforstabilizing}, and exploiting the time-integrated exponential decay of $\p_t \X_s$ obtained by interpolating between $D^{1/2}$ and $\| \p_t \X_{ss} \|_{L^2_s}$. This completes the proof for the free-ended rod.

\emph{Closed loop}. We now focus on the closed loop. By the propagation estimates in Proposition~\ref{prop:propagationproprod} and arguments in Section~\ref{sec:decayimmersednoforcing}, we have
\begin{equation}
E(t) \leq E^{\rm in} \, , \quad \lim_{t \to +\infty} D(t) = 0 \, .
\end{equation}
These control $\X_{ss}$ and $\kappa_3$ but not immediately $\X$ and $\e$. As in the curve-only setting, we define $\wt{\X} = \X - \int_\T \X(s,t) \, ds$. By~\eqref{eq:sinceframe}, we have
% \begin{equation}
% \p_t \e = - (\p_t \X_s \cdot \e) \X_s + \da{ \p_s \kappa_3 }
% \end{equation}
%\begin{equation}
%\end{equation}
\begin{equation}
\limsup_{t \to +\infty} \| \e \|_{H^2_s}^2 < +\infty \, .
\end{equation}
% \begin{equation}
% \p_t \kappa_3 = \p_t \X_s \cdot (\X_s \times \X_{ss}) - \frac{\eta}{\alpha} \p_s^2 \kappa_3
% \end{equation}

First, we demonstrate that \emph{any sequence $(\wt{\X},\e)(s,t_k)$ with $t_k \to +\infty$ contains a subsequence (not relabeled) converging in $H^4_s \times H^2_s$ to a steady state.} By weak compactness in $H^4_s \times H^2_s$ and the Rellich-Kondrachov theorem, there exists a subsequence satisfying
\begin{equation}
\wt{\X}(s,t_k) \to \bar{\X} \text{ in } H^{4-\varepsilon}_s \, , \; \forall \varepsilon \in (0,4] \, , \qquad \wt{\X}(s,t_k) \rightharpoonup \bar{\X} \text{ in } H^4_s
\end{equation}
\begin{equation}
\e(s,t_k) \to \bar{\e} \text{ in } H^{2-\varepsilon}_s \, , \; \forall \varepsilon \in (0,2] \, , \qquad \e(s,t_k) \rightharpoonup \bar{\e} \text{ in } H^{2}_s 
\end{equation}
\begin{equation}
\lambda[\wt{\X}(s,t_k),\e(s,t_k)] \to \lambda[\bar{\X},\bar\e] \text{ in } H^1_s \, .
\end{equation}
The gradient %\lo{[This $\nabla_g$ notation has not been introduced--maybe do as in section 3???]}
\begin{equation}
\nabla_g E = \big((\nabla_g E)_1, (\nabla_g E)_2\big) = \big( (\mathbf{I}+\gamma \X_s \otimes \X_s) \wt{\Z}_s , - ((\nabla_g E)_{1,s} \cdot \e) \X_s - \frac{\eta}{\alpha} \p_s \kappa_3 \X_s \times \e\big)
\end{equation}
is weakly continuous in $H^{-1}_s$ with respect to the above convergences, so $(\bar\X,\bar\e)$ is a steady state:
\begin{equation}
D[\bar\X,\bar\e] = 0 \, .
\end{equation}
% and therefore, by compactness, there exists $(\bar\X,\bar\e)$ and a sequence $t_k \to +\infty$ such that
% \begin{equation}
% (\X,\e)(t_k) \to (\bar\X,\bar\e) \text{ as } k \to +\infty \, .
% \end{equation}
The weak convergence of $\wt{\X}_{ssss}$ in $L^2_s$ becomes strong convergence once it is known that $\| \wt{\X}_{ssss} \|_{L^2_s} \to \| \bar{\X}_{ssss} \|_{L^2_s}$; this holds, since $D[\X,\e] \to 0 = D[\bar\X,\bar\e]$ and $(\lambda[\X_s] \X_s + \kappa_3 \X_s \times \X_{ss})_s$ converges strongly in $L^2_s$. %, we do have $\| \X_{ssss} \|_{L^2} \to \| \bar{\X}_{ssss} \|_{L^2}$, ergo, $\X(t_k) \to \bar\X$ strongly in $H^4$.
Next, since $\p_s \kappa_3 \to 0$ strongly in $L^2_s$, $\e$ converges strongly in~$H^2_s$, which proves the claim.

The remainder of the analysis is an application of the Lojasiewicz inequality in Proposition~\ref{pro:rodlojascewicz}, as in the curve-only setting, with straightforward modifications.

\qed

\subsection{Local well-posedness}\label{subsec:LWP_rodframe}
Finally, in this section, we prove Proposition~\ref{thm:wellposednesstheoryrod}. As a technical tool, we consider a curve-frame evolution with \emph{hyperviscosity} $\nu \p_s^4 \kappa_3$ ($\nu > 0$) in the frame variable:
\begin{equation}
\p_t\kappa_3 + \nu(\kappa_3)_{ssss} = \bigg(\frac{\eta}{\alpha}+\eta\abs{\X_{ss}}^2\bigg)(\kappa_3)_{ss} +
     \mc{R}^{\rm f}[\X_s,\kappa_3] \, ,
\end{equation}
where $\mc{R}^{\rm f}$ is as in~\eqref{eq:kap3_0-repeat}, augmented with the boundary conditions
\begin{equation}
\kappa\,,\, \p_s^2 \kappa_3\big|_{s=0,1} = 0 \, .
\end{equation}

Local well-posedness of the hyperviscous system is shown in Appendix~\ref{sec:hyperviscous}. The computation of $\frac{1}{2} \frac{dD}{dt}$ from Section~\ref{sec:energeticsrod} can be justified for solutions to the hyperviscous equations, which yields quantitative estimates on $D(t)$ pointwise-in-time. Our strong solutions will be obtained in the limit $\nu \to 0^+$.

We first prove the following existence statement and establish uniqueness afterward.

\begin{proposition}[Short-time existence]
    \label{pro:shorttimeexistence}
Let
\begin{equation}
(\X_{s}^{\rm in},\kappa_3^{\rm in}) \in H^3 \times H^1(I)
\end{equation}
\begin{equation}
    \label{eq:initialdata1overalpha}
|\X_{ss}^{\rm in}|^2 \leq \frac{1-2\delta}{\alpha} \, .
\end{equation}
Then there exists $T > 0$, depending only on $\| \X_{ss}^{\rm in} \|_{H^2_s}$, $\| \kappa_3^{\rm in} \|_{H^1_s}$, and $\delta$, such that a strong solution exists on $[0,T]$ with initial data $(\X_s^{\rm in},\kappa_3^{\rm in})$.
\end{proposition}

\begin{proof}%[Proof of Proposition~\ref{pro:shorttimeexistence}]
Let $\nu > 0$ and $(\X_s,\kappa_3)$ be the short-time solution from Lemma~\ref{lem:qualitativeexistencehyperviscosity}. 
% The quantitative version of the assumption~\eqref{eq:initialdata1overalpha} is
% \begin{equation}
% |\X_{ss}^{\rm in}| \leq (1-2\delta)/\alpha \, .
% \end{equation}
Recall the energy estimate
\begin{equation}
E(t) + \int_0^t D(t') \, dt' \leq E^{\rm in} \, ,
\end{equation}
and the propagation estimate from Proposition~\ref{prop:propagationproprod},
\begin{equation}
    \label{eq:energyestimateDshorttime}
D(t) \leq C(\delta,E^{\rm in}) (D^{\rm in} + 1) \, ,
\end{equation}
which is valid on any time interval $[0,T]$ over which $|\X_{ss}|^2 \leq \frac{1-\delta}{\alpha}$.

We now use Duhamel's formula to estimate $\| \X_{ss} \|_{L^\infty_s}$ and demonstrate that both the condition $|\X_{ss}|^2 \leq \frac{1-\delta}{\alpha}$ and the estimate~\eqref{eq:energyestimateDshorttime} can be propagated up to some time~$T$:\footnote{A rigorous way to phrase this is as a continuity/bootstrapping argument. Recall that $\| \X_{ss} \|_{L^\infty_s}$ varies continuously. We demonstrate that there exists $T$, as in the statement of the lemma, such that the interval on which $\| \X_{ss} \|_{L^\infty_s}^2 \leq \frac{1-\delta}{\alpha}$ is relatively open in $[0,T]$, so it must be the whole $[0,T]$.}
\begin{equation}
    \label{eq:DuhamelforLinfty}
\X_{ss} = e^{- t \p_s^4} \X_{ss}^{\rm in} + R \, , 
\end{equation}
\begin{equation}
    \label{eq:RdefforDuhamel}
\begin{aligned}
R(\cdot,t) &:= \int_0^t e^{- (t-t') \p_s^4} \p_s^2 Q \, dt' \, , \\
Q&:= (1+\gamma)\lambda_{s}\X_s + \lambda\X_{ss} 
    + \eta(\kappa_3\X_s\times\X_{ss})_s + 3\gamma\X_{s}(\X_{ss}\cdot\X_{sss})\,,
\end{aligned}
\end{equation}
where the semigroup is understood with the appropriate boundary conditions.

First, with $N := \| \X_{ss}^{\rm in} \|_{H^2_s}$, we have
\begin{equation}
    \label{eq:myH2semigroupbound}
\sup_{t \in \R_+} \| e^{- t \p_s^4} \X_{ss}^{\rm in} \|_{H^2_s}^2 + \| \p_t e^{- t \p_s^4} \X_{ss}^{\rm in} \|_{L^2_t L^2_s(I \times \R_+)}^2 \leq CN^2 \, ,
\end{equation}
as follows from the two standard energy estimates (multiply the $\X_{ss}$ equation by $\X_{ss}$ and $\p_t \X_{ss}$, respectively) in Section~\ref{sec:lineartheory}. Hence,
\begin{equation}
    \label{eq:myl2semigroupbound}
\| e^{- t \p_s^4} \X_{ss}^{\rm in} - \X_{ss}^{\rm in} \|_{L^2_s}^2 \leq C t N^2 \, .
\end{equation}
In particular, we can interpolate between $L^2_s$ in~\eqref{eq:myl2semigroupbound} and $H^2_s$ in~\eqref{eq:myH2semigroupbound} to obtain
\begin{equation}
\| e^{- t \p_s^4} \X_{ss}^{\rm in} - \X_{ss}^{\rm in} \|_{L^\infty_s}^2 \leq C t^{1/2} N^2 \, .
\end{equation}
Therefore,
\begin{equation}
	\label{eq:Linftybd1}
\| e^{- t \p_s^4} \X_{ss}^{\rm in} \|_{L^\infty_s} \leq \| \X_{ss}^{\rm in} \|_{L^\infty_s} + C t^{1/4} N \leq \sqrt{\frac{1-2\delta}{\alpha}} + Ct^{1/4} N \, .
\end{equation}

Next, we estimate $R$ defined in~\eqref{eq:RdefforDuhamel}. We have
\begin{equation}
\begin{aligned}
\| Q \|_{L^2_t L^2_s (I \times (0,T))} &\leq T^{1/2} \| R \|_{L^\infty_t L^2_s (I \times (0,T))} \leq C(\delta,E^{\rm in},D^{\rm in}) T^{1/2} \, , \\
\| R \|_{L^\infty_t L^2_s(I \times (0,T))} &\leq  C(\delta,E^{\rm in},D^{\rm in}) T^{1/2}\,,\\
\| Q \|_{L^2_t H^1_s (I \times (0,T))} &\leq C(\delta,E^{\rm in},D^{\rm in}) \, ,\\
\| R \|_{L^\infty_t H^1_s(I \times (0,T))} &\leq C(\delta,E^{\rm in},D^{\rm in}) \, ,
\end{aligned}
\end{equation}
following the estimate of Lemma~\ref{lem:semigroupestimates} proved in Appendix~\ref{sec:hyperviscous}. 
By interpolating between these estimates, we obtain
\begin{equation}
	\label{eq:Linftybd2}
\| R(\cdot,t) \|_{L^\infty_s} \leq C(\delta,E^{\rm in},D^{\rm in}) T^{1/4} \, .
\end{equation}
Combining~\eqref{eq:Linftybd1} and~\eqref{eq:Linftybd2} and taking $T \ll 1$ depending only on $\delta$, $E^{\rm in}$, and $D^{\rm in}$, we demonstrate the curvature condition can be propagated to time $T$. Therefore, we have estimates on $E, D \in L^\infty_t(0,T)$ and $\p_t \X_{ss} \in L^2_t L^2_s(I \times (0,T))$ uniformly in the hyperviscosity coefficient $\nu$. 
Finally, we apply standard compactness arguments involving the Aubin-Lions lemma, as in the construction of weak solutions to the Navier-Stokes equations~\cite[Chapter 4]{RobinsonRodrigoSadowskiBook}, ~\cite[Chapter 12]{LemarieRieusset21stCentury} \cite[Chapter 3]{TemamBook} to take $\nu \to 0^+$. \end{proof}

% We will estimate the terms inside of the square brackets (call them $A$) in $L^2_t H^{1/2+\varepsilon}_s$, so that the solution will belong to $L^\infty_t H^{1/2+\varepsilon/2}_s \hookrightarrow L^\infty_t L^\infty_s$ with a gain of a small but positive power of $t$:
% \begin{equation}
% \| R \|_{L^\infty_s(I)} \leq C(\delta,E^{\rm in},D^{\rm in}) t^{\frac{\varepsilon}{8}} \, .
% \end{equation}
% by the semigroup estimate
% \begin{equation}
% \| e^{-t\p_s^4} \p_s^2 f \|_{H^{\sigma'}_s(I)} \lesssim t^{-\frac{1}{2}-\frac{1}{4} (\sigma'-\sigma)} \| f \|_{H^{\sigma}_s(I)} \, ,
% \end{equation}

% \da{Semigroup estimate -- currently fixing this}

% $e^{-t\p_s^4} \p_s^2 : L^2 \to...$ by duality

% $(A+I)^{-1} \p_s^2 e^{-t\p_s^4} (A+I): L^2 \to L^2$

% $(A+I) e^{-t\p_s^4} \p_s^2 (A+I)^{-1} f : L^2 \to L^2$. $H^1$ to $H^1$.

%(I think there's something a little non-trivial here; the $L^\infty$ estimate for the biharmonic equation has a constant $C$ in front, probably. The biharmonic heat kernel isn't positive. The $L^\infty$ norm, if the initial function isn't continuous, could jump up a little at the initial time, no? On the other hand, if the data is smoother, then there should be strong convergence in $L^\infty$ at the initial time, with a rate. That could be what's necessary here. Basically, $u_0$ in $H^2$, so the heat kernel converges with a rate in $H^1$, which gives a rate in $L^\infty$, and then you can compare with the $L^\infty$ norm of the initial data. Nice.)

With this in hand, we can complete the proof of Proposition~\ref{thm:wellposednesstheoryrod} (local well-posedness) by establishing uniqueness.

\begin{proof}[Proof of Proposition~\ref{thm:wellposednesstheoryrod}]
%%%%%%%%%%%%%%%%%%%%%%%%%%%%%%%%%%%%%%%%%%%%%%%%%%%%%%%

Let $(\X_s^{(1)},\kappa_3^{(1)})$, $(\X_s^{(2)},\kappa_3^{(2)})$ be two strong solutions on $[0,T]$ with identical initial data, and let $\wt{E} = \frac{1}{2} \int_I\big( |\X_{ss}^{(1)} - \X_{ss}^{(2)}|^2 + \eta |\kappa_3^{(1)} - \kappa_3^{(2)}|^2\big)\,ds$ denote the relative energy. We may calculate the relative energy identity
\begin{equation}\label{eq:relE_ID}
\begin{aligned}
\frac{d\wt E}{dt} &+ \int_I |\p_s^2 (\X_{ss}^{(1)} - \X_{ss}^{(2)})|^2\,ds + \int_I \left( \frac{\eta^2}{\alpha} + \frac{|\X_{ss}^{(1)}|^2}{\eta} \right) |\p_s (\kappa_3^{(1)} - \kappa_3^{(2)})|^2\,ds = \sum_{i=1}^7J_i\,,  \\
J_1&= 2\eta \int_I \big[\p_s \kappa_3^{(1)} \X_s^{(1)} \times \X_{ss}^{(1)} - \p_s \kappa_3^{(2)} \X_s^{(2)} \times \X_{ss}^{(2)}\big] \cdot \p_s^2 (\X_{ss}^{(1)} - \X_{ss}^{(2)})\,ds \\
J_2&= \int_I \big[(1+\gamma) (\lambda_s^{(1)} \X_s^{(1)} - \lambda_s^{(2)} \X_s^{(2)}) + \lambda^{(1)} \X_{ss}^{(1)} - \lambda^{(2)} \X_{ss}^{(2)}\big] \cdot \p_s^2 (\X_{ss}^{(1)} - \X_{ss}^{(2)})\,ds \\
J_3&= 2\eta \int_I (\kappa_3^{(1)} \X_s^{(1)} \times \X_{sss}^{(1)} - \kappa_3^{(2)} \X_s^{(2)} \times \X_{sss}^{(2)}) \cdot \p_s^2 (\X_{ss}^{(1)} - \X_{ss}^{(2)})\,ds \\
J_4&= 3\gamma  \int_I  (\X_s^{(1)} (\X_{ss}^{(1)} \cdot \X_{sss}^{(1)}) - \X_s^{(2)} (\X_{ss}^{(2)} \cdot \X_{sss}^{(2)})) \cdot (\X_{ss}^{(1)} - \X_{ss}^{(2)})\,ds  \\
J_5&= \eta^2 \int_I (|\X_{ss}^{(1)}|^2 - |\X_{ss}^{(2)}|^2) (\kappa_3^{(2)})_{ss} (\kappa_3^{(1)} - \kappa_3^{(2)})\,ds \\
J_6&= \eta^2 \int_I \big[\kappa_3^{(1)} (\X_{ssss}^{(1)} \cdot \X_{ss}^{(1)} + |\X_{ss}^{(1)}|^4) \\
&\qquad\qquad - \kappa_3^{(2)} (\X_{ssss}^{(2)} \cdot \X_{ss}^{(2)} + |\X_{ss}^{(2)}|^4)\big] (\kappa_3^{(1)} - \kappa_3^{(2)})\,ds \\
J_7&= \eta \int_I \big[\lambda^{(1)} \X_{sss}^{(1)} \cdot (\X_s^{(1)} \times \X_{ss}^{(1)}) - \lambda^{(2)}\X_{sss}^{(2)} \cdot (\X_s^{(2)} \times \X_{ss}^{(2)})\big] (\kappa_3^{(1)} - \kappa_3^{(2)})\,ds \, .
\end{aligned}
\end{equation}
We will derive the differential inequality
\begin{equation}\label{eq:wtEdiff}
\frac{d\wt{E}}{dt} \leq C \wt{E} \, ,
\end{equation}
where the constant $C$ may depend on the particulars of the solutions, including the time $T$, $\norm{\X_{ssss}^{(j)}}_{L^\infty_tL^2_s}$, $\norm{(\kappa_3^{(j)})_s}_{L^\infty_tL^2_s}$, and $\delta$.

We begin with the only term which requires $|\X_{ss}^{(1)}|^2 < \frac{1}{\alpha}$, which is $J_1$. We rewrite $J_1$ as
\begin{equation}
\begin{aligned}
    J_1 &= \underbrace{2\eta \int_I \p_s (\kappa_3^{(1)}-\kappa_3^{(2)}) \big(\X_s^{(1)} \times \X_{ss}^{(1)} \big) \cdot \p_s^2 (\X_{ss}^{(1)} - \X_{ss}^{(2)})\,ds }_{J_{1a}}\\
    &\qquad + \underbrace{2\eta \int_I \p_s \kappa_3^{(2)}\big( \X_s^{(1)} \times \X_{ss}^{(1)} - \X_s^{(2)} \times \X_{ss}^{(2)}\big) \cdot \p_s^2 (\X_{ss}^{(1)} - \X_{ss}^{(2)})\,ds}_{J_{1b}}\,.
\end{aligned}   
\end{equation}
As long as $|\X_{ss}^{(1)}|^2 \leq \frac{1-2\delta}{\alpha}$, the term $J_{1a}$ may be estimated by
\begin{equation}
\abs{J_{1a}}\le 
(1-\delta) \frac{\eta^2}{\alpha} \norm{\p_s (\kappa_3^{(1)} - \kappa_3^{(2)})}_{L^2_s}^2 + \underbrace{\frac{\alpha}{1-\delta} \sup |\X_{ss}^{(1)}|^2}_{< 1} \norm{\p_s^2 (\X_{ss}^{(1)} - \X_{ss}^{(2)})}_{L^2_s}^2 \, ,
\end{equation}
and therefore both terms may be absorbed into the left-hand side with room to spare.
The remaining term $J_{1b}$ may be estimated as 
\begin{equation}
\begin{aligned}
    \abs{J_{1b}} &\lesssim  \norm{\p_s \kappa_3^{(2)}}_{L^2_s}
    \big( \norm{\X_s^{(1)}- \X_s^{(2)}}_{L^\infty_s}\norm{\X_{ss}^{(2)}}_{L^\infty_s} \\
    &\qquad
    + \norm{\X_{ss}^{(1)} - \X_{ss}^{(2)}}_{L^\infty_s}\big) 
    \norm{\p_s^2 (\X_{ss}^{(1)} - \X_{ss}^{(2)})}_{L^2_s} \\
    &\lesssim \wt E^{1/2}\norm{\p_s^2 (\X_{ss}^{(1)} - \X_{ss}^{(2)})}_{L^2_s} + \wt E^{3/8}\norm{\p_s^2 (\X_{ss}^{(1)} - \X_{ss}^{(2)})}_{L^2_s}^{5/4}\\
    &\le \varepsilon\norm{\p_s^2 (\X_{ss}^{(1)} - \X_{ss}^{(2)})}_{L^2_s}^2 + C(\varepsilon)\wt E\,,
\end{aligned}
\end{equation}
for any choice of $\varepsilon>0$, and we note that $C(\varepsilon)$ also depends on $\norm{\p_s \kappa_3^{(2)}}_{L^2_s}$ and $\norm{\X_{ssss}^{(2)}}_{L^2_s}$.

We continue term-by-term. To estimate $J_2$, we will require the following tension estimate, which is a direct consequence of \eqref{eq:lamdiff_H1} from Lemma \ref{lem:tension_rod}:  
\begin{equation}\label{eq:lam1lam2}
\begin{aligned}
    \norm{\lambda^{(1)}-\lambda^{(2)}}_{H^1_s} 
    % &\lesssim \norm{\X_{s}^{(1)}-\X_{s}^{(2)}}_{H^2_s} +\|\kappa_3^{(1)}- \kappa_3^{(2)}\|_{L^2_s}\\
    %&\lesssim \big(\norm{\X_s^{(1)}}_{H^2_s}+ \norm{\X_s^{(2)}}_{H^2_s}\big)\norm{\X_s^{(1)}}_{H^2_s}^2\big(1+\norm{\kappa_3^{(1)}}_{H^1_s} \big)\norm{\X_s^{(1)}-\X_s^{(2)}}_{H^2_s}\\
    &\lesssim \wt E^{1/2} + \wt E^{1/4}\norm{\p_s^2(\X_{ss}^{(1)}-\X_{ss}^{(2)})}_{L^2_s}^{1/2}\,.
\end{aligned}   
\end{equation}
Using \eqref{eq:lam1lam2} along with Lemma \ref{lem:tension_rod}, we may estimate
\begin{equation}
\begin{aligned}
    \abs{J_2} 
    &\lesssim \big(\norm{\lambda^{(1)}-\lambda^{(2)}}_{H^1_s}(1+\norm{\X_{ss}^{(1)}}_{L^2_s})\\
    &\qquad + \norm{\lambda^{(2)}}_{H^1_s}\norm{\X_{ss}^{(1)}-\X_{ss}^{(2)}}_{L^2_s}\big)\norm{\p_s^2(\X_{ss}^{(1)}-\X_{ss}^{(2)})}_{L^2_s} \\
    &\lesssim \wt E^{1/2}\norm{\p_s^2(\X_{ss}^{(1)}-\X_{ss}^{(2)})}_{L^2_s} + \wt E^{1/4}\norm{\p_s^2(\X_{ss}^{(1)}-\X_{ss}^{(2)})}_{L^2_s}^{3/2} \\
    &\le \varepsilon\norm{\p_s^2 (\X_{ss}^{(1)} - \X_{ss}^{(2)})}_{L^2_s}^2 + C(\varepsilon)\wt E \,,
\end{aligned}
\end{equation}
again for any choice of $\varepsilon>0$.
We next estimate $J_3$ as 
\begin{equation}
\begin{aligned}
    \abs{J_3}&\lesssim 
    \bigg(\norm{\kappa_3^{(1)}-\kappa_3^{(2)}}_{L^2_s} \norm{\X_{sss}^{(2)}}_{L^\infty_s} +\norm{\kappa_3^{(1)}}_{L^\infty_s}\norm{\X_s^{(1)}- \X_s^{(2)}}_{L^\infty_s}\norm{\X_{sss}^{(2)}}_{L^2_s}\\
    &\qquad +\norm{\kappa_3^{(1)}}_{L^\infty_s} \norm{\p_s(\X_{ss}^{(1)}-\X_{ss}^{(2)})}_{L^2_s} \bigg)\norm{\p_s^2 (\X_{ss}^{(1)} - \X_{ss}^{(2)})}_{L^2_s} \\
    &\lesssim \big(\wt E^{1/2} + \wt E^{1/4}\norm{\p_{s}^2(\X_{ss}^{(1)}-\X_{ss}^{(2)})}_{L^2_s}^{1/2}\big)\norm{\p_s^2 (\X_{ss}^{(1)} - \X_{ss}^{(2)})}_{L^2_s}\\
    &\le \varepsilon\norm{\p_s^2 (\X_{ss}^{(1)} - \X_{ss}^{(2)})}_{L^2_s}^2 + C(\varepsilon)\wt E\,,
\end{aligned}
\end{equation}
and similarly estimate $J_4$ as
\begin{equation}
\begin{aligned}
    \abs{J_4} &\lesssim \bigg(\norm{\X_s^{(1)}-\X_s^{(2)}}_{L^\infty_s} \norm{\X_{ss}^{(2)}}_{L^\infty_s} \norm{\X_{sss}^{(2)}}_{L^2_s} + \norm{\X_{ss}^{(1)}-\X_{ss}^{(2)}}_{L^2_s} \norm{\X_{sss}^{(2)}}_{L^\infty_s}\\
    &\qquad + \norm{\X_{ss}^{(1)}}_{L^\infty_s} \norm{\p_s(\X_{ss}^{(1)}-\X_{ss}^{(2)})}_{L^2_s}\bigg)\norm{\X_{ss}^{(1)} - \X_{ss}^{(2)}}_{L^2_s} \\
    &\lesssim \wt E + \wt E^{3/4}\norm{\p_s^2(\X_{ss}^{(1)}-\X_{ss}^{(2)})}_{L^2_s}^{1/2} \\
    &\le \varepsilon\norm{\p_s^2(\X_{ss}^{(1)}-\X_{ss}^{(2)})}_{L^2_s}^2 + C(\varepsilon)\wt E\,.
\end{aligned}
\end{equation}

For $J_5$, we integrate by parts once before estimating. We may then obtain the bound 
\begin{equation}
\begin{aligned}
    \abs{J_5}&= \bigg|\eta^2 \int_I (|\X_{ss}^{(1)}|^2 - |\X_{ss}^{(2)}|^2) (\kappa_3^{(2)})_{s} (\kappa_3^{(1)} - \kappa_3^{(2)})_s\,ds\\
    &\qquad +2\eta^2 \int_I (\X_{ss}^{(1)}\cdot\X_{sss}^{(1)} - \X_{ss}^{(2)}\cdot\X_{sss}^{(2)}) (\kappa_3^{(2)})_{s} (\kappa_3^{(1)} - \kappa_3^{(2)})\,ds\bigg|\\
    &\lesssim \norm{(\kappa_3^{(2)})_s}_{L^2_s}\bigg(\norm{\X_{ss}^{(1)}-\X_{ss}^{(2)}}_{L^\infty_s}\norm{\X_{ss}^{(1)}+\X_{ss}^{(2)}}_{L^\infty_s}\norm{\p_s(\kappa_3^{(1)} - \kappa_3^{(2)})}_{L^2_s}\\
    &+
    \big(\norm{\X_{ss}^{(1)}- \X_{ss}^{(2)}}_{L^2_s}\norm{\X_{sss}^{(2)}}_{L^\infty_s} 
    + \norm{\X_{ss}^{(1)}}_{L^\infty_s}\norm{\X_{sss}^{(1)} -\X_{sss}^{(2)}}_{L^2_s}\big) \norm{\kappa_3^{(1)} - \kappa_3^{(2)}}_{L^\infty_s}\bigg)\\
    &\lesssim \wt E^{3/8}\norm{\p_s^2(\X_{ss}^{(1)}- \X_{ss}^{(2)})}_{L^2_s}^{1/4}\norm{\p_s(\kappa_3^{(1)} - \kappa_3^{(2)})}_{L^2_s} +\wt E^{3/4}\norm{\p_s(\kappa_3^{(1)} - \kappa_3^{(2)})}_{L^2_s}^{1/2}\\
    &\qquad + \wt E^{1/2}\norm{\p_s(\kappa_3^{(1)} - \kappa_3^{(2)})}_{L^2_s}^{1/2}\norm{\p_s^2(\X_{ss}^{(1)}- \X_{ss}^{(2)})}_{L^2_s}^{1/2}\\
    &\le C(\varepsilon) \wt E + \varepsilon\norm{\p_s^2(\X_{ss}^{(1)}- \X_{ss}^{(2)})}_{L^2_s}^2 + \varepsilon\norm{\p_s(\kappa_3^{(1)} - \kappa_3^{(2)})}_{L^2_s}^2\,.
\end{aligned}
\end{equation}
Here $C(\varepsilon)$ again depends on the particulars of the solutions. We may next estimate $J_6$ as
\begin{equation}
\begin{aligned}
\abs{J_6} &\lesssim \bigg( \big(\norm{\kappa_3^{(1)}- \kappa_3^{(2)}}_{L^\infty_s} \norm{\X_{ssss}^{(2)}}_{L^2_s} + \norm{\kappa_3^{(1)}}_{L^\infty_s} \norm{\p_s^2(\X_{ss}^{(1)}-\X_{ss}^{(2)})}_{L^2_s} \big)\norm{\X_{ss}^{(2)}}_{L^\infty_s}\\
&\quad + \norm{\kappa_3^{(1)}}_{L^\infty_s} \norm{\X_{ssss}^{(1)}}_{L^2_s} \norm{\X_{ss}^{(1)}-\X_{ss}^{(2)}}_{L^\infty_s} + \norm{\kappa_3^{(1)}- \kappa_3^{(2)}}_{L^2_s}\norm{\X_{ss}^{(1)}}_{L^\infty_s}^4 \\
&\quad + \norm{\kappa_3^{(2)}}_{L^\infty_s}\norm{\X_{ss}^{(1)}-\X_{ss}^{(2)}}_{L^2_s}\big(\norm{\X_{ss}^{(1)}}_{L^\infty_s}^3+\norm{\X_{ss}^{(2)}}_{L^\infty_s}^3\big)\bigg)\norm{\kappa_3^{(1)} - \kappa_3^{(2)}}_{L^2_s} \\
&\lesssim \wt E^{3/4}\norm{\p_s(\kappa_3^{(1)}- \kappa_3^{(2)})}_{L^2_s}^{1/2} + \wt E^{1/2}\norm{\p_s^2(\X_{ss}^{(1)}-\X_{ss}^{(2)})}_{L^2_s} \\
&\qquad
+ \wt E^{5/8}\norm{\p_s^2(\X_{ss}^{(1)}-\X_{ss}^{(2)})}_{L^2_s}^{1/4} + \wt E\\
&\le C(\varepsilon) \wt E + \varepsilon\norm{\p_s^2(\X_{ss}^{(1)}- \X_{ss}^{(2)})}_{L^2_s}^2 + \varepsilon\norm{\p_s(\kappa_3^{(1)} - \kappa_3^{(2)})}_{L^2_s}^2\,.
\end{aligned}
\end{equation}

Finally, using \eqref{eq:lam1lam2} and Lemma \ref{lem:tension_rod}, we may estimate $J_7$ as 
\begin{equation}
\begin{aligned}
\abs{J_7}&\lesssim \bigg( \norm{\lambda^{(1)}- \lambda^{(2)}}_{L^\infty_s} \norm{\X_{sss}^{(2)}}_{L^2_s} \norm{\X_{ss}^{(2)}}_{L^\infty_s} \\
&\quad + \norm{\lambda^{(1)}}_{L^\infty_s} \big(\norm{\p_s(\X_{ss}^{(1)} -\X_{ss}^{(2)})}_{L^2_s}\norm{\X_{ss}^{(2)}}_{L^\infty_s} + \norm{\X_{sss}^{(1)}}_{L^\infty_s} \norm{\X_{ss}^{(1)}-\X_{ss}^{(2)}}_{L^2_s} \\
&\quad + \norm{\X_{sss}^{(1)}}_{L^2_s} \norm{\X_s^{(1)}-\X_s^{(2)}}_{L^\infty_s} \norm{\X_{ss}^{(2)}}_{L^\infty_s} \big) \bigg) \norm{\kappa_3^{(1)} - \kappa_3^{(2)}}_{L^2_s} \\
&\lesssim \wt E + \wt E^{3/4}\norm{\p_s^2(\X_{ss}^{(1)}-\X_{ss}^{(2)})}_{L^2_s}^{1/2} \\
&\le \varepsilon\norm{\p_s^2(\X_{ss}^{(1)}-\X_{ss}^{(2)})}_{L^2_s}^2 + C(\varepsilon)\wt E\,.
\end{aligned}
\end{equation}

Altogether, taking $\varepsilon$ sufficiently small, we may absorb the right-hand side terms $\norm{\p_s^2(\X_{ss}^{(1)}-\X_{ss}^{(2)})}_{L^2_s}^2$ and $\norm{\p_s(\kappa_3^{(1)} - \kappa_3^{(2)})}_{L^2_s}^2$ into the left-hand side of \eqref{eq:relE_ID} to obtain the relative energy inequality \eqref{eq:wtEdiff}.
\end{proof}

%%%%%%%%%%%%%%%%%%%%%%%%%%%%%%%%%%%%%%%%%%%%%%%%%%

%!TEX root = RFT GWP.tex

\section{Lojasiewicz inequalities}
\label{sec:rod_loja}

To demonstrate the convergence to equilibrium for~\eqref{eq:classical} and~\eqref{eq:Xss_0}-\eqref{eq:kap3_0} on the periodic interval $I = \T$, we require a powerful tool from the theory of gradient flows known as the \emph{Lojasiewicz inequality}. Even in finite dimensions, it is possible for solutions to the gradient flow $\dot x = - \nabla V$ of a smooth potential $V$ to asymptotically wind in such a way that different subsequences converge to different equilibria; see the ``goat tracks" in~\cite[p. 69]{SimonLectureNotes}. However, this is impossible for analytic $V$ due to the Lojasiewicz inequality, which may be regarded as a piece of algebraic geometry. In~\cite{SimonAnnals}, Simon discovered how to prove and exploit the Lojasiewicz inequality for certain analytic energies by reducing the inequality to finite dimensions using ellipticity. There is now an abstract framework in which to prove many such inequalities. We use this framework to prove a Lojasiewicz inequality first in the curve-only setting and then in the coupled curve-frame setting.

% \textbf{Q. Is there a way to do this without reference to $\lambda$?}

% Need analyticity of the map, so also *analyticity* of $\lambda$ with respect to change in the curve. That's okay -- you can *probably* do that with the implicit function theorem.

%\subsection{Proof of the Lojasiewicz inequality}

%Proposition~\ref{pro:rodlojascewicz}}
%\subsection{The tangent space at a rod configuration}

\subsection{Curves}

%Only doing this in the periodic setting, by the way, since we have a straight up Poincar{\'e} in the non-periodic setting

First, we consider the manifold structure of the space of curves. One may view this as a framework in which to make the gradient calculation in Section~\ref{sec:derivation} rigorous. Let
\begin{equation}
\mathbf{M} := \big\{ \X \in H^4(\T;\R^3) \,: \, |\X_s|^2 = 1 \; \forall s \in \T \big\} \, .
\end{equation}
We may locally coordinatize $\mathbf{M}$ in the following way. Let $\X \in \mathbf{M}$. Define the tangent space
\begin{equation}
T_{\X} \mathbf{M} := \big\{ \Y \in H^4(\T;\R^3) \,:\, \Y_s \cdot \X_s = 0 \; \forall s \in \T \big\} \, ,
\end{equation}
endowed with the $H^4$ topology. In a neighborhood $U$ of the origin in $T_{\X} \mathbf{M}$, we have the exponential map
\begin{equation}
{\rm Exp}_{\X} : U \subset T_{\X} \mathbf{M} \to \mathbf{M} \, ,
\end{equation}
defined pointwise through the exponential map on the sphere:
\begin{equation}
({\rm Exp}_{\X} \Y)_s(s) = \exp_{\X_s(s)}^{S^2} \Y_s(s) \, , \quad ({\rm Exp}_{\X} \Y)(0) = \X(0) + \Y(0) \, .
\end{equation}
 Then ${\rm Exp}_{\X}$ is analytic into the ambient space (since $S^2$ is analytic and the map is defined pointwise) and provides a local coordinate chart around $\X$.\footnote{We need local coordinates to apply Proposition~\ref{pro:abstractloja}. The particulars of the exponential map will not enter the computations except for property $d|_0 \rm{Exp}_{(\bar{\X},\bar{\e})} = {\rm Id}$.

 The above charts can be used to give $\mathbf{M}$ the structure of an analytic Hilbert manifold which is embedded in the ambient space $H^4(\T;\R^3)$. However, all of our calculations will take place in a single chart.}

For $\gamma \geq 0$, we endow $\mathbf{M}$ with the metric %(I + \gamma \X_s \otimes \X_s)^{-1}
\begin{equation}
\label{eq:metrica}
g^\gamma_{\X}(\Y,\Z) = \int_\T (\mathbf{I} + \gamma \X_s \otimes \X_s)^{-1} \Y \cdot \Z \, ds \, , \quad \forall \Y,\Z \in T_{\X} \mathbf{M}
\end{equation}
% \begin{equation}
% g_{\X} (\Y,\Z) = \int \Y \cdot \Z \, ds \, , \quad \forall \Y, \Z \in T_{\X} \mathbf{M}
% \end{equation}
and define the energy functional $E : \mathbf{M} \to \R$ by
\begin{equation}
E = \frac{1}{2} \int_\T |\X_{ss}|^2 \, ds \, ,
\end{equation}
both of which are analytic on the ambient space and hence also analytic when restricted to~$\mathbf{M}$.

\begin{proposition}[Curve Lojasiewicz inequality]
	\label{pro:curvelojascewicz}
Let $\X \in \mathbf{M}$ be a critical point for the energy: $d|_{\X} E = 0$. Then there exists $\delta(\X) > 0$ and $\beta \in (0,1/2)$ such that if $\Y \in \mathbf{M}$ with $\| \X - \Y \|_{H^4} \leq \delta$ and $\gamma \geq 0$, then
\begin{equation}
|E(\X) - E(\Y)|^{1-\beta} \lesssim_\gamma \sqrt{D_\gamma[\Y]} \, ,
\end{equation}
%where $\mathcal{M}(\Y)$ is the $L^2$-gradient of $E$ at $Y$, as computed in Section~..., computed in the metric with $\gamma = 0$ (by equivalence).
where, as in Section~\ref{subsec:energetics},
\begin{equation}
D_\gamma[\Y] = \int_{\T} ((\mathbf{I} + \gamma \X_s \otimes \X_s) \Z_s) \cdot \Z_s \, ds \, , \quad \Z = \Y_{sss} - \lambda[\Y] \Y_s \, ,
\end{equation}
with $\lambda[\Y]$ as in~\eqref{eq:lambdalojadef}.
\end{proposition}
One may compare with the statement of the $\gamma=0$ planar case in~\cite[Proposition 8.5, p. 434]{koiso1996motion}.

We rely on an abstract Lojasiewicz inequality which generalizes the approach in~\cite{SimonAnnals,SimonLectureNotes}, see, e.g.~\cite{HuangGradient,FeehanMaridakis}. The following variant is adapted from~\cite[Theorem 2.4.2]{kohout2021a}.

%See Theorem~2.4.2 in~\href{https://ora.ox.ac.uk/objects/uuid:a47991d1-541a-4ef4-86cd-4f600d21d4ac/files/d4t64gn731}.
\begin{proposition}
	\label{pro:abstractloja}
Suppose we are given the following data:
\begin{enumerate}
\item A Hilbert space $(H,\langle\cdot,\cdot\rangle_H)$ and a Banach space $X$ densely embedded into $H$.
% \begin{equation}
% X \subset Y \subset H \, , \quad X \subset Z \subset H \, ,
% \end{equation}
% such that $X, Y, Z$ are complete when equipped with norms $\| \cdot \|_X,\| \cdot \|_Y,\| \cdot \|_Z$, respectively,
%and such that the above inclusions are continuous.
\item An analytic functional $E : U \to \R$ which is defined on an open neighborhood $U \subset X$ of the origin.
\item An analytic metric $g : U \to {\rm Sym}_2(U;\R)$ with $g \approx \langle \cdot, \cdot \rangle_H$ as bilinear forms in $U$:
\begin{equation}
\frac{1}{C} g_u(v,v) \leq \langle v,v \rangle_H \leq C g_u(v,v) \, , \quad \forall u \in U \, , \; v \in H \, .
\end{equation}
\item There is an analytic map $\nabla_g E : U \to H$ so that for every $u \in U$ and $x \in X$
\begin{equation}
d|_uE(x) := \frac{d}{dt} E(u + tx)|_{t=0} = \langle \nabla_g E(u), x \rangle_g \, .
\end{equation}
\item Suppose that $\nabla_g E(0) = 0$ and $\mathcal{L} := d|_{0} \nabla_g E$ is a Fredholm operator of index zero from $X$ to $H$.
%For every $u \in U$, the element $d_u \nabla E \in B(X, Y )$ has
%an extension to $L_u \in B(Z, H)$ such that the map $U \ni u \to L_u \in B(Z, H)$ is continuous.
% Finally, $L_{\bar{u}} \in B(Z, H)$ (an extension of $L$) is Fredholm of index zero.
\end{enumerate}
Then there is $\beta \in (0,1/2)$, $A < \infty$, and a neighborhood $\mathcal{O} \subset U$ of the origin so that
\begin{equation}
	\label{eq:conclusiontoobtain}
|E(u) - E(0)|^{1-\beta} \leq A \| \nabla_g E(u) \|_g \, , \quad \forall u \in \mathcal{O} \, .
\end{equation}
\end{proposition}

The statement in~\cite{kohout2021a} is written for the special case when the metric $g = \langle \cdot, \cdot \rangle_H$ (that is, independent of $u$). We can obtain Proposition~\ref{pro:abstractloja} from this special case in the following way:

% , and we demonstrate below how to reduce Proposition~\ref{pro:abstractloja} to that special case:

\begin{proof}
According to~\cite{kohout2021a}, Proposition~\ref{pro:abstractloja} holds assuming that
\begin{equation}
	\label{eq:toproveloja}
\nabla_H E : U \to H \text{ is analytic, and } d|_0 \nabla_H E : X \to H \text{ is Fredholm index zero} \, ,
\end{equation}
 and the conclusion is that
\begin{equation}
	\label{eq:knownlojaconclusion}
|E(u) - E(0)|^{1-\beta} \leq A \| \nabla_H E(u) \|_H \, .
\end{equation}
Our goal is two-fold: Prove~\eqref{eq:toproveloja} from our assumptions on $\nabla_g E$, and obtain the conclusion~\eqref{eq:conclusiontoobtain} from~\eqref{eq:knownlojaconclusion}.

Suppose that the hypotheses are satisfied with the analytic metric $g$. The map $g|_u^\flat : H \to H^*$, $g|_u^\flat(v) = g(v,\cdot)$, and its inverse $g|_u^\sharp$ from Riesz representation are analytic. We also define $h^\sharp : H^* \to H$ to be the Riesz representation map corresponding to $\langle \cdot, \cdot \rangle_H$. Therefore, under the assumptions,
\begin{equation}
u \mapsto d|_uE = g|_u^\flat \nabla_g E(u) \in H^*
\end{equation}
is analytic, so $u \mapsto \nabla_H E(u) = h^\sharp d|_uE$ is analytic. 

Next, we verify the Fredholm property.
\begin{equation}
d|_{u} \nabla_H E = d|_u h^\sharp dE = d|_u h^\sharp g|_u^\flat (\nabla_g E) = h^\sharp (d|_u g^\flat) \nabla_g E + h^\sharp g|_u^\flat d|_u \nabla_g E \, ,
\end{equation}
and we evaluate at the origin, which is a critical point, thereby obtaining
\begin{equation}
d|_{u} \nabla_H E= h^\sharp g|_0^\flat d|_0 \nabla_g E \, .
\end{equation}
The right-hand side is a post-composition of the Fredholm index zero operator with an isomorphism of the Hilbert space $H$. Therefore, $d|_0 \nabla_H E$ is Fredholm index zero, and we have verified~\eqref{eq:toproveloja}.

The right-hand side of the conclusion~\eqref{eq:knownlojaconclusion} is
\begin{equation}
	\label{eq:resultwithjusttheh}
\| \nabla_H E(u) \|_H = \sup_{\| v \|_{H} \leq 1} \langle \nabla E(u), v \rangle_H = \sup_{\| v \|_{H} \leq 1} d|_uE(v) \, .
\end{equation}
We have the analogous characterization
\begin{equation}
	\label{eq:appealtomeforsameresults}
\| \nabla_g E(u) \|_g = \sup_{\| v \|_g \leq 1} \langle \nabla_g E(u), v \rangle_g = \sup_{\| v \|_g \leq 1} d|_uE(v) \, .
\end{equation}
Since $\| \cdot \|_H$ and $\| \cdot \|_g$ are equivalent in a neighborhood of $u=0$, the desired conclusion~\eqref{eq:conclusiontoobtain} follows from~\eqref{eq:knownlojaconclusion} and the above characterizations of the norms.
\end{proof}

\begin{proof}[Proof of Proposition~\ref{pro:curvelojascewicz}]
We now verify the hypotheses of Proposition~\ref{pro:rodlojascewicz}. We begin with the case $\gamma = 0$; we recover $\gamma \geq 0$ at the end as a corollary. Fix $\bar{\X} \in \mathbf{M}$. We are applying the proposition in coordinates at $\bar{\X}$, namely, in a neighborhood $U$ of the origin in $T_{\bar{\X}} \mathbf{M}$. (Although the metric~\eqref{eq:metrica} \emph{looks} independent of $\X$, it is variable in the coordinate patch.) We work in the Banach space $X = T_{\bar{\X}} \mathbf{M}$, and the Hilbert space
\begin{equation}
	\label{eq:myHilbertspace}
H = \big\{ \Y \in L^2(\T;\R^3) \,:\, \Y_s \cdot \X_s = 0 \text{ in } H^{-1} \big\} \, .
\end{equation}
At the origin in $T_{\bar{\X}} \mathbf{M}$, $g = g^0$ coincides with the $L^2$ inner product on the tangent space. By analyticity, it is therefore bounded above and below by the $L^2$ inner product on a neighborhood of the origin. The gradient computation in Section~\ref{sec:derivation}, specifically,~\eqref{eq:firstvariationofE1}-\eqref{eq:gradientofE2}, yields\footnote{Moreover, when there is no preferred curvature, the calculation can be greatly simplified.}
\begin{equation}
\nabla_g E = (\X_{sss} - \lambda[\X] \X_{s})_s \, ,
\end{equation}
which is analytic. The Lagrange multiplier $\lambda[\X] : H^4 \to H^2$ is the solution to
\begin{equation}
	\label{eq:lambdalojadef}
\lambda_{ss} - |\X_{ss}|^2 \lambda + 4 (\X_{sss} \cdot \X_{ss})_s - |\X_{sss}|^2 = 0 \, ,
\end{equation}
whose analyticity can be demonstrated via the implicit function theorem.

Let $\wt{E} = E \circ {\rm Exp}_{\bar{\X}}$ and $\wt{g}$ be the metric in coordinates. Then
\begin{equation}
\nabla_g E \circ {\rm Exp}_{\bar{\X}} = d{\rm Exp}_{\bar{\X}} \nabla_{\wt{g}} \wt{E} \, , \quad \text{ in } U \, .
\end{equation}
We calculate $d|_{0}$ of the above expression:
\begin{equation}
d|_{\bar{\X}} \nabla_g E \underbrace{d|_{0} {\rm Exp}}_{= {\rm Id}} = d|_{0} d{\rm Exp}_{\bar{\X}} \underbrace{(\nabla_{\wt{g}} \wt{E})(0)}_{=0} + \underbrace{d|_{0} {\rm Exp}_{\bar{\X}}}_{= {\rm Id}} d|_{0} \nabla_{\wt{g}} \wt{E} \, ,
\end{equation}
that is,
\begin{equation}
d|_{\bar{\X}} \nabla_g E = d|_{0} \nabla_{\wt{g}} \wt{E} : T_{\bar{\X}} \mathbf{M} \to H \, ,
\end{equation}
so we can study the linearized operator without reference to coordinates.
We compute
\begin{equation}
	\label{eq:LinearizedOperatorL}
\mathcal{L}[\Y] = d|_{\bar{\X}} \nabla_g E[\Y] = \big(\Y_{sss} - (d|_{\bar{\X}}\lambda)[\Y] \X_{0,s} - \lambda[\bar{\X}] \Y_s\big)_s \, .
\end{equation}
% where $(d|_{\bar{\X}}\lambda)[\Y]$ is the solution $\lambda$ to
% \begin{equation}
% \lambda_{ss} - |\X_{0ss}|^2 \lambda - 2\lambda[\X_{0s}] \X_{0ss} \cdot \Y_{ss} + 4 (\X_{0sss} \cdot \Y_{ss} + \Y_{sss} \cdot \X_{0ss})_s - 2 \X_{0sss} \cdot \Y_{sss} = 0 \, .
% \end{equation}
%It is more convenient to take a derivative of this operator and consider it as an operator on $\Y_s$.

To demonstrate the Fredholm property, we choose a smooth frame $\e_1(s),\e_2(s)$ on $\bar{\X}$, perpendicular to $\X_{0,s}$, and make the identification $T_{\bar{\X}} \mathbf{M} \cong H^3(\T;\R)^2 \times \R^2$ under the map
\begin{equation}
	\label{eq:curveidentification}
\Y \mapsto (c_1,c_2,\vec{c}_0) := \left( \Y_s \cdot \e_1 , \Y_s \cdot \e_2 , \int_\T \Y \, ds \right) \, .
\end{equation}
Likewise, we identify the Hilbert space $H$ in~\eqref{prop:propagationproprod} with $H^{-1}(\T;\R)^2 \times \R^2$. To calculate $\mathcal{L}$ under this identification, we write $\Y_s = c_1 \e_1 + c_2 \e_2$ and compute $\mathcal{L}_s \cdot \e_1$ and $\mathcal{L}_s \cdot \e_2$. (The action of the equation on the center of mass $\vec{c}_0$ is trivial.) Since
\begin{equation}
\Y_{sssss} = c_{1,ssss} \e_1 + c_{2,ssss} \e_2 + \cdots \, ,
\end{equation}
where the remainder belongs to $L^2$, the linearized operator $\mathcal{L}$ in~\eqref{eq:LinearizedOperatorL} is realized as
\begin{equation}
	\label{eq:firstidentificationlinearized}
	\wt{\mathcal{L}} = (c_1,c_2,\vec{c}_0) \mapsto
\begin{bmatrix}
\p_s^4 + \wt{A} & B &  \\
C & \p_s^4 + \wt{D}& \\
& & 0
\end{bmatrix} : H^3(\T;\R)^2 \times \R^2 \to H^{-1}(\T,\R)^2 \times \R^2 \, ,
\end{equation}
where $\wt{A}, B,C,\wt{D}$ map into $L^2(\T;\R)$. Let
\begin{equation}
P = \begin{bmatrix}
(1 + \p_s^4)^{-1} & & \\
& (1 + \p_s^4)^{-1} & \\
& & 1
\end{bmatrix} \, .
\end{equation}
Then
\begin{equation}
 P \wt{\mathcal{L}} = {\rm Id} + P \begin{bmatrix}
-1 + \wt{A} & B &  \\
C & -1 + \wt{D}& \\
& & -1
\end{bmatrix} : H^3(\T;\R)^2 \times \R^2 \circlearrowleft
\end{equation}
is a compact perturbation of the identity and hence Fredholm index zero.

% Next, we add and subtract a copy of the identity and post-compose the $c_1$ and $c_2$ equations with the Fourier multiplier \da{$(1 + \p_s^3)^{-1}$} to transform the system into ${\rm Id} + \text{ compact}$, which is Fredholm index zero. This concludes the proof for $\gamma=0$.

%we wish to invert $I - \p_s^4$ and realize the operator as $I + \text{ compact}$. However, $\p_s^4$ does not preserve the property $\Y_s \cdot \X_{0s} = 0$. We 
%The problem becomes a system of equations for $c_1$ and $c_2$. 

 %\dacomment{should clean. Is there a better formulation than on the full $\Y$? Need the full $\Y$ for gradients, though.} 

% \begin{equation}
% d{\rm Exp}_{\bar{\X}}[\X] : T_{\bar{\X}} \mathbf{M} \to H64(\T;\R^3)
% \end{equation}
% You could write things into coordinates.

Finally, to incorporate $\gamma \geq 0$, we observe that $g^\gamma$ and $g=g^0$ induce locally equivalent norms, so the conclusions of the Lojasiewicz inequality are equivalent, according to~\eqref{eq:appealtomeforsameresults}. \end{proof}
% the term $\mathbf{I} + \gamma \X_s \otimes \X_s$, one may consider the more general metric
% \begin{equation}
% g^\gamma_{\X}(\Y,\Z) = \int (\mathbf{I} + \gamma \X_s \otimes \X_s) \Y \cdot \Z \, ds \, , \quad \forall \Y,\Z \in T_{\X} \mathbf{M} \, .
% \end{equation}
%Since , Therefore, it suffices to prove the theorem only for $g^0$.

\subsection{Framed curves}

We now introduce the manifold structure on the space of framed curves
\begin{equation}
\mathbf{M}_{\rm f} := \big\{ (\X,\e) \in \mathbf{M} \times H^2(\T;\R^3) \,:\, |\e|^2 = 1 \, , \; \e \cdot \X_s = 0 \, , \; \forall s \in \T \big\} \, .
\end{equation}
The tangent space, realized in the ambient space, is 
\begin{equation}
T_{(\X,\e)} \mathbf{M}_{\rm f} = \big\{ (\Y,\bm{b}) \in T_{\X} \mathbf{M} \times H^2(\T;\R^3) \,:\, \bm{b} \cdot \e = 0 \, , \; \bm{b} \cdot \X_s + \e \cdot \Y_s = 0 \, , \; \forall s \in \T \big\} \, .
\end{equation}
It will sometimes be natural to parameterize %\footnote{Here we send $\bm{b} \to - \bm{b}$.} %\dacomment{Frame bundle}
\begin{equation}
	\label{eq:independentdegree}
\bm{b} = (\e \cdot \Y_s) \X_s + \underbrace{\theta \X_s \times \e}_{\wt{\bm{b}}} \, , %\, , \quad \wt{\bm{b}} \perp \X_s\,,\; \e \, ,
\end{equation}
which reflects that the frame moves passively with the curve, and $\theta$ is the independent degree of freedom associated with the frame. More concretely, we have the identification
\begin{equation}
\iota_{\X,\e} : T_{\X} \mathbf{M} \times H^2(\T;\R) \to T_{(\X,\e)} \mathbf{M}_{\rm f}
\end{equation}
\begin{equation}
\iota_{\X,\e}(\Y,\theta) = (\Y, (\e \cdot \Y_s) \X_s + \theta \X_s \times \e) \, .
\end{equation}
Next, the exponential map $\exp_{\X_s(s)} \Y_s(s)$ induces a rotation $Q[\X_s,\Y_s](s) \in SO(3)$. Then 
\begin{equation}
{\rm Exp}_{(\X,\e)} (\Y,\bm{b}) = ({\rm Exp}_{\X} \Y, Q[\X_s,\Y_s](s) \exp_{\e}^{S_2} \wt{\bm{b}})
\end{equation}
is analytic. For $\gamma \geq 0$ and $\alpha > 0$, we endow $\mathbf{M}_{\rm f}$ with the metric
\begin{equation}
	\label{eq:gmetricforrod}
g^{\gamma,\alpha}_{\X,\e}(\Y,\bm{b};\Z,\bm{c}) = \int_\T (\mathbf{I} + \gamma \X_s \otimes \X_s)^{-1} \Y \cdot \Z \, ds + \alpha \int_\T \theta(s) \phi(s) \, ds \, ,%, \quad \forall \Y, \Z \in \T_{\X} \mathbf{M} 
\end{equation}
where $\theta, \phi$ are as in~\eqref{eq:independentdegree}. For $\eta > 0$, we define the energy functional
\begin{equation}
E_\eta = \frac{1}{2} \int_\T |\X_{ss}|^2 \, ds + \frac{\eta}{2} \int_\T |\kappa_3|^2 \, ds \, ,
\end{equation}
where
\begin{equation}
\kappa_3[\X,\e] = \p_s \e \cdot (\X_s \times \e) : H^4 \times H^2 \to H^1 \, .
\end{equation}
% One of the points will be to use $1/\alpha$ to ensure the ellipticity of $L$.

\begin{proposition}[Rod Lojasiewicz inequality]
	\label{pro:rodlojascewicz}
Let $(\X,\e) \in \mathbf{M}_{\rm f}$. There exists $\delta(\X,\e) > 0$ and $\beta \in (0,1/2)$ such that if $(\Y,\mathbf{b}) \in \mathbf{M}_{\rm f}$ with $\| (\X,\e) - (\Y,\mathbf{b}) \|_{H^4 \times H^2} \leq \delta$ and $\eta > 0$, $\gamma \geq 0$, and $\alpha>0$, then
\begin{equation}
|E_\eta(\X,\e) - E_\eta(\Y,\mathbf{b})|^{1-\beta} \lesssim_{\eta,\gamma,\alpha} \sqrt{D_{\eta, \gamma,\alpha}[\Y,\bm{b}]} \, ,
\end{equation}
where, as in Section~\ref{sec:energeticsrod},
\begin{equation}
\begin{aligned}
D_{\eta, \gamma,\alpha}[\Y,\bm{b}] &= \int_{\T} ((\mathbf{I} + \gamma \X_s \otimes \X_s) \wt{\Z}_s) \cdot \wt{\Z}_s \, ds \, , \\
 \wt{\Z} &= \Y_{sss} - \lambda[\Y,\bm{b}] \Y_s - \eta \kappa_3[\Y,\b] \Y_s \times \Y_{ss} \, ,
 \end{aligned}
\end{equation}
with $\lambda[\Y,\bm{b}]$ as in~\eqref{eq:lambdadefforloja2}.
\end{proposition}

\begin{proof}[Proof of Proposition~\ref{pro:rodlojascewicz}]
We proceed with calculations analogous to the curve case. Define the Banach space $X = T_{(\bar{\X},\bar{\e})} \mathbf{M}_{\rm f}$ and Hilbert space
\begin{equation}
H = \big\{ (\Y,\b) \in L^2(\T;\R^3)^2 \,:\, \Y_s \cdot \X_s = 0 \, ,\; \bm{b} \cdot \e = 0 \, , \; \bm{b} \cdot \X_s + \e \cdot \Y_s = 0\, \text{ in } H^{-1}(\T) \big\}\,.
\end{equation}
We begin with $\gamma = 0$ and we suppress the dependence of $g$ on $\alpha$. The gradient $\nabla_g E$ with $g$ as in~\eqref{eq:gmetricforrod} will have two components. From the calculations in Section~\ref{sec:derivation}, specifically,~\eqref{eq:firstvariationofE1}-\eqref{eq:gradientofE3}, we have
\begin{equation}
(\nabla_g E)_1 = (\X_{sss} - \lambda[\X,\e] \X_{s} - \eta \kappa_3 \X_s \times \X_{ss})_s
\end{equation}
\begin{equation}
(\nabla_g E)_2 = - ((\nabla_g E)_{1,s} \cdot \e) \X_s - \frac{\eta}{\alpha} (\kappa_3)_s \X_s \times \e \, ,
\end{equation}
where $\lambda[\X,\e] : H^4 \times H^2 \to H^2$ is the solution $\lambda$ to
\begin{equation}
	\label{eq:lambdadefforloja2}
\lambda_{ss} - |\X_{ss}|^2 \lambda = - 4 (\X_{sss} \cdot \X_{ss})_s + |\X_{sss}|^2 - \eta \kappa_3 \X_s \cdot (\X_{ss} \times \X_{sss}) \, .
\end{equation}
Let $(\bar{\X},\bar{\e})$ be a rod equilibrium, which is smooth with twist $\kappa_3 \equiv C_0$ constant. Then
\begin{equation}
\begin{aligned}
\mathcal{L}_1[\Y,\b] &= d|_{(\bar{\X},\bar{\e})} \nabla_g E[\Y,\b]_1 \\
& = \big(\Y_{sss} - d|_{(\bar{\X},\bar{\e})} \lambda[\Y,\b] \X_s - \lambda[\bar{\X},\bar{\e}] \Y_s\big)_s \\
&\qquad - \eta \big(d|_{(\bar{\X},\bar{\e})} \kappa_3 [\Y,\b] \X_s \times \X_{ss} + C_0 (\Y_s \times \X_{ss} + \X_s \times \Y_{ss})\big)_s \\
\mathcal{L}_2[\Y,\b] &= d|_{(\bar{\X},\bar{\e})} \nabla_g E[\Y,\b]_2 \\
& = - (\mathcal{L}_1[\Y,\b]_{s} \cdot \e) \X_s - \frac{\eta}{\alpha} \big(d|_{(\bar{\X},\bar{\e})} \kappa_3[\Y,\b]\big)_{s} \X_s \times \e \, ,
\end{aligned}
\end{equation}
where in the $\mathcal{L}_2$ calculation we have used that $(\bar{\X},\bar{\e})$ is a critical point. We compute
\begin{equation}
d|_{(\X,\e)} \kappa_3[\Y,\b] = \p_s \b \cdot (\X_s \times \e) + \p_s \e \cdot (\Y_s \times \e + \X_s \times \b) \, .
\end{equation}
It will be advantageous to rewrite $\mathcal{L} = d|_{(\bar{\X},\bar{\e})} \nabla_g E$ under the map $\iota_{\X,\e}$ and examine the $\theta$ component of $\mathcal{L}$. This is done by writing $\bm{b}$ as in~\eqref{eq:independentdegree} and computing $\mathcal{L}_2 \cdot (\X_s \times \e)$:
\begin{equation}
- \frac{\eta}{\alpha} \big(d|_{(\bar{\X},\bar{\e})} \kappa_3[\Y,\b]\big)_{s} = - \frac{\eta}{\alpha} \theta_{ss} + R \, ,
\end{equation}
where $R$ maps into $H^1$. 
%, since $\p_s \b \cdot (\X_s \times \e) = \theta_s$. 
As in the curve-only setting~\eqref{eq:curveidentification}, we identify $\Y$ with its coefficients $c_1,c_2$ in a smooth frame and a center of mass $\vec{c}_0$. Under these identifications, the linearized operator $\mathcal{L}$ becomes, compared to~\eqref{eq:firstidentificationlinearized}, the block operator
\begin{equation}
\wt{\mathcal{L}} : (c_1,c_2,\vec{c}_0,\theta) \mapsto
\begin{bmatrix}
\p_s^3 + \wt{A}_{11} & A_{12} & & - \eta (\X_s \times \X_{ss}) \cdot \e_1 \p_{ss} + \wt{A}_{14} \\
A_{21} & \p_s^3 + \wt{A}_{22} & & - \eta (\X_s \times \X_{ss}) \cdot \e_2 \p_{ss}  + \wt{A}_{24} \\
& & 0 & \\
B_{41} & B_{42} & & - \frac{\eta}{\alpha} \p_{ss}  + \wt{B}_{44}
\end{bmatrix}
\end{equation}
\begin{equation}
H^3(\T;\R)^2 \times \R^2 \times H^2(\T;\R) \to H^{-1}(\T,\R)^2 \times \R^2 \times L^2(\T;\R) \, ,
\end{equation}
where the $A$, $\wt{A}$ operators map into $L^2(\T;\R)$ and the $B$, $\wt{B}$ operators map into $H^1(\T;\R)$. Therefore, we write
\begin{equation}
P = \begin{bmatrix}
(1 + \p_s^4)^{-1} &&& \\
& (1 + \p_s^4)^{-1} && \\
&& 1& \\
&&& (1 - \frac{\eta}{\alpha} \p_s^2)^{-1}
\end{bmatrix}\,,
\end{equation}
\begin{equation}
\begin{aligned}
P\wt{\mathcal{L}} &= {\rm Id} + P \begin{bmatrix}
0&&& - \eta (\X_s \times \X_{ss}) \cdot \e_1 \p_{ss} \\
&0&&- \eta (\X_s \times \X_{ss}) \cdot \e_2 \p_{ss}  \\
& & 0 & \\
&&& 0
\end{bmatrix}\\
&\qquad  + P \begin{bmatrix}
-1 + \wt{A}_{11} & A_{12} & & \wt{A}_{14} \\
A_{21} & -1 + \wt{A}_{22} & &  \wt{A}_{24} \\
& & -1 & \\
B_{41} & B_{42} & & -1 + \wt{B}_{44}
\end{bmatrix} \, ,
\end{aligned}
\end{equation}
% we add and subtract the identity operator and post-compose with $(1 + \p_s^3)^{-1}$ in the $c_1$ and $c_2$ equations and $(1 - \frac{\eta}{\alpha} \p_{ss})^{-1}$ in the $\theta$ equation to obtain
i.e.,  $P\wt{\mathcal{L}}={\rm Id} + \text{ strictly upper triangular } + \text{ compact} : H^3(\T;\R)^2 \times \R^2 \times H^2(\T;\R) \circlearrowleft$, 
which is a compact perturbation of an invertible operator and hence Fredholm index zero.

Finally, when $\gamma \geq 0$, we argue as before by appealing to equivalent norms, as in~\eqref{eq:appealtomeforsameresults}.
\end{proof}

% \da{General thought. From one perspective, we should be trying to do things with the minimal amount of derivatives. This means tracking the evolution for $\X$ and, possibly, the evolution for $\theta$. Can you close things at that level? Probably. Roughly equivalent to $\X_s$.}

% % General principle: The frame is passively advected.

% Note: Saverio also uses the $\kappa_3$ formulation, so it can't be all crazy.

% %$\alpha$ only appears in the inner product, notably.

% $\dot x = -\nabla V(x)$

% $\dot V(x) = -|\nabla V(x)|^2 = -|\dot x|^2$

% $V(x)$ is converging, so this actually makes $|\dot x|^2$ time-integrable

% Also, we have enough compactness to assume that $x$ and therefore $V(x)$ are converging strongly *along a sequence of times* (!) to something $\bar{x}$. The goal is to show that $\bar{x}$ must be a critical point of the energy. Seems obvious, right? The point is the subsequence of times. Okay, but you do have that $V(x)$ is converging to $V(\bar{x})$, since energy. Then $|V(x) - V(\bar(x))|$ is controlled by $|\dot x|^{1/\theta}$.

% $\dot V = -|\dot x|^2 \geq - |\nabla V|^{\alpha}$

% with $\alpha < 1$. 

% $\int_t^s \dot x^2 \leq V(x) - V(\bar{x}) \leq |\nabla V|^{1/\theta} = |\dot x|^{1/\theta}$

%!TEX root = RFT GWP.tex

\section{Further investigations}

\subsection{Vanishing rotational friction coefficient}\label{subsec:alphato0}
We next consider the closed rod in the limit as $\alpha\to 0$ and prove Theorem \ref{cor:alpha0}, global well-posedness without restriction on $\norm{\X_{ss}}_{L^\infty_s}$. Expressing the closed curve and frame system \eqref{eq:Xss_alpha0}-\eqref{eq:kap3_alpha0} as $\alpha\to 0$ in terms of $(\X_s,\overline\kappa_3)$, we have 
\begin{align}
    \p_t\X_{s} + (\X_{s})_{ssss}
    &=  (\mc{R}^0[\X_s])_s  \label{eq:Xss_alpha0} \\
    \p_t\overline\kappa_3 &= \overline{\mc{R}}(t) \label{eq:kap3_alpha0}\,,
\end{align}
with
\begin{equation}
\begin{aligned}
\mc{R}^0[\X_s] &= (1+\gamma)\lambda_{s}\X_s + \lambda\X_{ss} 
    + \eta\overline\kappa_3(\X_s\times\X_{sss}) +  3\gamma\X_{s}(\X_{ss}\cdot\X_{sss})\\
\overline{\mc{R}}(t) &= \eta\overline\kappa_3\int_{\T}\bigg(-\abs{\X_{sss}}^2 +\abs{\X_{ss}}^4\bigg)\,ds\\
&\qquad+ \int_{\T}\bigg(\X_{ssss} \cdot (\X_s\times\X_{sss})+\lambda\X_{sss}\cdot (\X_s\times\X_{ss})\bigg)\,ds\,.
\end{aligned}
\end{equation}
The tension $\lambda$ satisfies 
\begin{equation}
    (1+\gamma)\lambda_{ss}  -\abs{\X_{ss}}^2\lambda = 
    -(4+3\gamma)(\X_{sss}\cdot\X_{ss})_s +\abs{\X_{sss}}^2  - \eta\overline\kappa_3\X_s\cdot(\X_{ss}\times\X_{sss})\,.
\end{equation}

Since $\overline\kappa_3$ is constant in $s$, local well-posedness for $(\X_s,\overline\kappa_3) \in \big(C([0,T];H^1_s) \cap L^2_t H^3_s\big) \times C([0,T])$ follows from analysis similar to Section~\ref{sec:immersednoforcing}, and the solution can be continued provided that $\X_s \in L^\infty_t H^1_s \cap L^2_t H^3_s(\T \times (0,T))$ and $\overline\kappa_3$ remains bounded, which we verify below.

To obtain global well-posedness, we again turn to energy methods. We now have 
\begin{equation}
\wt{\bm{Z}} = \X_{sss}-\lambda \X_s-\eta\overline\kappa_3\X_s\times\X_{ss}\,,
\end{equation}
and the analogous energy and dissipation quantities are given by 
 \begin{equation}
 E(t) = \frac{1}{2}\int_\T\bigg(\abs{\X_{ss}}^2+ \eta\overline\kappa_3^2\bigg)\,ds\,, \quad
D(t) = \int_\T\big(({\bf I}+\gamma\X_s\otimes\X_s)\wt{\bm{Z}}_s\big)\cdot\wt{\bm{Z}}_s\, ds\,.
 \end{equation}
The system \eqref{eq:Xss_alpha0}-\eqref{eq:kap3_alpha0} then satisfies the energy identity 
\begin{equation}
\begin{aligned}
    \frac{dE}{dt} =-D(t)\,.
\end{aligned}
\end{equation}

Recall that the size restriction on $\norm{\X_{ss}}_{L^\infty_s}$ in the $\alpha>0$ case stemmed from needing to bound $\norm{\X_{ssss}}_{L^2_s}$ by $D(t)$ and $E(t)$. In the $\alpha\to 0$ case, we may attain such a bound without the size restriction. In particular, as in the $\alpha>0$ case, we define 
\begin{equation}
\begin{aligned}
D_1(t) &= \int_\T\bigg( \abs{\bm{Z}_s}^2  + \gamma \big(\X_s\cdot\bm{Z}_s\big)^2 
      \bigg) \,ds\,.
\end{aligned}
\end{equation}
where again $\bm{Z}=\X_{sss}-\lambda\X_s$. Again we have 
\begin{equation}\label{eq:Xssss_closed_again}
\norm{\X_{ssss}}_{L^2_s}\lesssim D_1(t)^{1/2}(1+E(t)^{1/2})+ E(t)+E(t)^{5/2}
\end{equation}
by \eqref{eq:Xssss_closed0}. We then have
\begin{equation}\label{eq:D1D_ineq}
D_1(t) + \eta^2\int_\T\abs{\overline\kappa_3\big(\X_s\times\X_{ss} \big)_s}^2\,ds \le D(t) + 2\eta\int_\T
    \abs{\overline\kappa_3\bm{Z}_s\cdot\big(\X_s\times\X_{ss} \big)_s}  \,ds\,.
\end{equation}

Splitting up the right-hand side as in \eqref{eq:divvy}, we still have a term to estimate:
\begin{equation}
\begin{aligned}
    \eta^2\overline\kappa_3^2\int_\T\abs{\X_{sss}}^2\,ds 
    &\lesssim \overline\kappa_3^2\norm{\X_{ss}}_{L^2_s}\norm{\X_{ssss}}_{L^2_s}\\
    &\lesssim  D_1(t)^{1/2}(E(t)^{3/2}+E(t)^2)+ E(t)^{5/2}+E(t)^4\\
    &\le \varepsilon D_1(t) + C(E,\varepsilon)\,.
\end{aligned}
\end{equation}
Taking $\varepsilon= 1/4$, we may absorb $D_1(t)$ into the left-hand side of \eqref{eq:D1D_ineq} to obtain
\begin{equation}
 D_1(t) \le 4D(t) + C(E)\,, 
 \qquad \norm{\X_{ssss}}_{L^2_s} \le C(E)D(t) + C(E)\,.
 \end{equation} 
 This verifies that $\X_s \in L^2_t H^3_s(I \times (0,T))$ on any finite interval, and hence the solution can be continued indefinitely.
 %The rest of the proof of global well-posedness proceeds as in Section \ref{subsec:GWP_rodframe}. %\lo{This should reference section 3 instead...}

%%%%%%%%%%%%%%%%%%%%%%%%%%%%%%%%%%%%%%%%%%%%%%%%%%%%%%%%%%%%%%%%%%%%%%%%%%%%%%
%%%%%%%%%%%%%%%%%%%%%%%%%%%%%%%%%%%%%%%%%%%%%%%%%%%%%%%%%%%%%%%%%%%%%%%%%%%%%%
%%%%%%%%%%%%%%%%%%%%%%%%%%%%%%%%%%%%%%%%%%%%%%%%%%%%%%%%%%%%%%%%%%%%%%%%%%%%%%
%%%%%%%%%%%%%%%%%%%%%%%%%%%%%%%%%%%%%%%%%%%%%%%%%%%%%%%%%%%%%%%
%%%%%%%%%%%%%%%%%%%%%%%%%%%%%%%%%%%%%%%%%%%%%%%%%%%%%%%%%%%%%%%
%%%%%%%%%%%%%%%%%%%%%%%%%%%%%%%%%%%%%%%%%%%%%%%%%%%%%%%%%%%%%%%
\subsection{On unconditional curve-frame well-posedness}\label{subsec:uncond_GWP}
Despite the condition \eqref{eq:Xss_criterion} on the maximum filament curvature required in our proof of Theorem \ref{thm:GWP_curveframe}, we have reason to suspect that the coupled curve-frame evolution \eqref{eq:Xss_0}-\eqref{eq:kap3_0} is in fact well posed without such a condition. 
In particular, the apparent backward heat evolution may not actually be present, as it is cancelled by the strong coupling between \eqref{eq:Xss_0} and \eqref{eq:kap3_0}. 
We will use the $\kappa_j$ formulation \eqref{eq:kap1}-\eqref{eq:kap3} with $\zeta_1=\zeta_2=\zeta_3=0$ and consider the linearized curve-frame evolution for a closed filament about the multiply-covered covered circle with constant (very large) curvature $\overline\kappa$ and constant twist rate $b$. For this state, the filament tension is constant and satisfies $\lambda_0=(1-\eta)b^2$. If $b=0$, this is a steady state of the equations. 

Linearizing \eqref{eq:kap1}-\eqref{eq:kap3} about $(\kappa_1,\kappa_2,\kappa_3)=(\overline\kappa,0,b)$ for $\overline\kappa,b$ constant, we obtain the system 
\begin{align}
\p_t \kappa_1 &= -(\kappa_1)_{ssss} + \big( (5-2\eta)b^2 -(1+\gamma) \overline\kappa^2\big)(\kappa_1)_{ss} +(4-\eta)b(\kappa_2)_{sss} 
      \nonumber \\
    &\quad 
    - (2-\eta)b^3(\kappa_2)_s
    +(4-4\eta)b\overline\kappa(\kappa_3)_{ss}
     - 2(1-\eta)\overline\kappa b^3\kappa_3
    + \overline\kappa b^2\wt\lambda   
    + \gamma\overline\kappa\wt\lambda_{ss}\,, \\
\p_t \kappa_2 &= -(\kappa_2)_{ssss} + (5-2\eta)b^2(\kappa_2)_{ss}  + \big( (2-\eta)b^3 -(1+\gamma)b\overline\kappa^2 \big)(\kappa_1)_s  \nonumber \\
    &\quad  
    -(4-\eta)b(\kappa_1)_{sss}
     -(1-\eta)\overline\kappa(\kappa_3)_{sss} \nonumber \\
     &\quad + \overline\kappa\bigg( 5(1-\eta) b^2 - \frac{\eta}{\alpha}\bigg)(\kappa_3)_s  
    - (1-\gamma)\overline\kappa b\wt\lambda_s \,,\\
\p_t \kappa_3 &= \bigg(\frac{\eta}{\alpha} - (1-\eta)\overline\kappa^2\bigg)(\kappa_3)_{ss} +2(1-\eta)\overline\kappa^2b^2\kappa_3 -\overline\kappa(\kappa_2)_{sss}  \nonumber
      \\
    &\quad  
      + (2-\eta)b^2\overline\kappa(\kappa_2)_s- (3-\eta)b\overline\kappa(\kappa_1)_{ss}  
    -  b\overline\kappa^2\wt\lambda \,,
\end{align}
along with the linearized tension equation (where here $\wt\lambda = -(\lambda+\kappa^2)=-(\lambda+\kappa_1^2)$, as in~\cite{oelz2011curve}) 
\begin{equation}
(1+\gamma)\wt\lambda_{ss} - \overline\kappa^2\wt\lambda = (2+\gamma)\overline\kappa(\kappa_1)_{ss} - (2-\eta)b\overline\kappa(\kappa_2)_s  +2(\eta-1)\overline\kappa^2b\kappa_3 \,.
\end{equation}
Note that if $b=0$, the $\kappa_1$ evolution completely decouples from the $(\kappa_2,\kappa_3)$ evolution:
\begin{align}
\p_t \kappa_1 &= -(\kappa_1)_{ssss}  -(1+\gamma)\overline\kappa^2(\kappa_1)_{ss}  +\gamma \overline\kappa\wt\lambda_{ss}\,, \\
\p_t \kappa_2 &= -(\kappa_2)_{ssss} -(1- \eta)\overline\kappa(\kappa_3)_{sss} - \frac{\eta}{\alpha}\overline\kappa(\kappa_3)_s \,, \\
\p_t \kappa_3 &= \frac{\eta}{\alpha}(\kappa_3)_{ss} + (\eta-1)\overline\kappa^2(\kappa_3)_{ss}  -\overline\kappa(\kappa_2)_{sss}\,,
\end{align}
where
\begin{equation}
(1+\gamma)\wt\lambda_{ss}- \overline\kappa^2\wt\lambda = (2+\gamma)\overline\kappa(\kappa_1)_{ss}  \,. 
\end{equation}

Expanding each $\kappa_j(s,t)$, $j=1,2,3$, and $\wt\lambda(s,t)$ as Fourier series in $s$,
\begin{equation}
 \kappa_j(s,t)=\sum_{n=-\infty}^\infty e^{i2\pi n s}\wh\kappa_j(n,t)\,, \qquad 
 \wt\lambda(s,t) =\sum_{n=-\infty}^\infty e^{i2\pi n s}\wh\lambda(n,t)\,,
\end{equation} 
we may solve for the tension coefficients $\wh\lambda(n,t)$ as
\begin{equation}
\wh\lambda = \frac{(2\pi n)^2(2+\gamma)\overline\kappa}{(1+\gamma)(2\pi n)^2+ \overline\kappa^2}\wh\kappa_1 + \frac{(i2\pi n)(2-\eta)b\overline\kappa}{(1+\gamma)(2\pi n)^2+ \overline\kappa^2}\wh\kappa_2  -\frac{2(\eta-1)\overline\kappa^2b}{(1+\gamma)(2\pi n)^2+ \overline\kappa^2}\wh\kappa_3 \,.
\end{equation}
The coefficients $\wh\kappa_j(n,t)$ then satisfy the following system of ODEs:
\begin{align}
\p_t\wh\kappa_1 &= -(2\pi n)^2\bigg( (2\pi n)^2 +(5-2\eta)b^2 -(1+\gamma) \overline\kappa^2 +\frac{(2+\gamma)\overline\kappa^2((2\pi n)^2 \gamma-b^2)}{(1+\gamma)(2\pi n)^2+ \overline\kappa^2}\bigg) \wh\kappa_1 \nonumber \\
    &\quad -i(2\pi n)\bigg( (2\pi n)^2(4-\eta)b+ (2-\eta)b^3 +\frac{(2-\eta)b\overline\kappa^2((2\pi n)^2 \gamma-b^2)}{(1+\gamma)(2\pi n)^2+ \overline\kappa^2}\bigg)\wh\kappa_2  \label{eq:hatsys1}\\
    &\quad 
    -\bigg( (2\pi n)^2(4-4\eta)b\overline\kappa +2(1-\eta)\overline\kappa b^3 +\frac{2(1-\eta)\overline\kappa^3b((2\pi n)^2\gamma-b^2)}{(1+\gamma)(2\pi n)^2+ \overline\kappa^2}\bigg)\wh\kappa_3 \nonumber \\
\p_t\wh\kappa_2 &= i(2\pi n)b\bigg( (2\pi n)^2(4-\eta) + (2-\eta)b^2 -(1+\gamma)\overline\kappa^2 -\frac{(2\pi n)^2(1-\gamma)(2+\gamma)\overline\kappa^2}{(1+\gamma)(2\pi n)^2+ \overline\kappa^2}\bigg)\wh\kappa_1 \nonumber \\
    &\quad -(2\pi n)^2\bigg( (2\pi n)^2 +(5-2\eta)b^2 -\frac{(1-\gamma)(2-\eta)\overline\kappa^2 b^2}{(1+\gamma)(2\pi n)^2+ \overline\kappa^2}\bigg)\wh\kappa_2  \label{eq:hatsys2}
    \\
    &\quad + i(2\pi n)\bigg( (2\pi n)^2(1-\eta)\overline\kappa + (5-5\eta)\overline\kappa b^2 - \frac{\eta}{\alpha}\overline\kappa - \frac{2(1-\gamma)(1-\eta)\overline\kappa^3 b^2}{(1+\gamma)(2\pi n)^2+ \overline\kappa^2}\bigg)\wh\kappa_3   \nonumber \\
\p_t\wh\kappa_3 &= (2\pi n)^2\bigg( (3-\eta)b\overline\kappa -\frac{(2+\gamma)b\overline\kappa^3}{(1+\gamma)(2\pi n)^2+ \overline\kappa^2}\bigg)\wh\kappa_1  \nonumber
      \\
    &\quad  
     +i(2\pi n)\bigg((2\pi n)^2\overline\kappa + (2-\eta)b^2\overline\kappa - \frac{(2-\eta)b^2\overline\kappa^3}{(1+\gamma)(2\pi n)^2+ \overline\kappa^2}\bigg)\wh\kappa_2  \label{eq:hatsys3} \\
    &\quad 
     -\bigg((2\pi n)^2\bigg(\frac{\eta}{\alpha} - (1-\eta)\overline\kappa^2\bigg) -2(1-\eta)\overline\kappa^2b^2 -\frac{2(\eta-1)b^2\overline\kappa^4}{(1+\gamma)(2\pi n)^2+ \overline\kappa^2}\bigg)\wh\kappa_3 \,. \nonumber
\end{align}
As $n\to\infty$, the leading order asymptotic behavior of the system \eqref{eq:hatsys1}-\eqref{eq:hatsys3} is given by 
\begin{equation}
\p_t \begin{pmatrix}
\wh\kappa_1\\
\wh\kappa_2\\
\wh\kappa_3
\end{pmatrix}
\sim
-(2\pi n)^2
\begin{pmatrix}
(2\pi n)^2 & i(2\pi n)(4-\eta)b & 4(1-\eta)b\overline\kappa\\
-i(2\pi n)(4-\eta)b & (2\pi n)^2 & -i(2\pi n)(1-\eta)\overline\kappa \\
-(3-\eta)b\overline\kappa & -i(2\pi n)\overline\kappa & \frac{\eta}{\alpha}-(1-\eta)\overline\kappa^2
\end{pmatrix}
\begin{pmatrix}
\wh\kappa_1\\
\wh\kappa_2\\
\wh\kappa_3
\end{pmatrix}\,.
\end{equation}
We may calculate the eigenvalues $\mu_1(n),\mu_2(n),\mu_3(n)$ of the above matrix, which, as $n\to\infty$, satisfy 
\begin{equation}\label{eq:muks}
\mu_1\,,\;\mu_2 = -(2\pi n)^4 + O(n^3)\,, \quad \mu_3 = -(2\pi n)^2\frac{\eta}{\alpha} + O(n)\,, \qquad n\to\infty\,.
\end{equation}
In particular, no matter how large $\overline\kappa$ is, all $\mu_j$ are negative beyond some finite wavenumber $n$. This suggests that the backward heat operator apparently present in \eqref{eq:kappa3strongcoupling} for large initial curvatures is not really there; in fact, the strong coupling between the curve and frame evolutions actually helps counteract the development of backward heat. 

To see the effect of the strong coupling more explicitly, we consider the $b=0$ setting. The system \eqref{eq:hatsys1}-\eqref{eq:hatsys3} then simplifies to 
\begin{align}
\p_t \wh\kappa_1 &= \frac{(2\pi)^2(1+\gamma)n^2}{(2\pi)^2(1+\gamma)n^2+\overline\kappa^2}\big( -(2\pi)^4n^4 +\overline\kappa^4  \big)\wh\kappa_1 \,,  \label{eq:kap1_nob}\\
\p_t \begin{pmatrix}
\wh\kappa_2\\
\wh\kappa_3
\end{pmatrix} &= 
\begin{pmatrix}
-(2\pi n)^4 &  -i2\pi n\overline\kappa \big( (2\pi n)^2(\eta-1) + \frac{\eta}{\alpha}\big) \\
i(2\pi n)^3\overline\kappa & -(2\pi n)^2\big(\frac{\eta}{\alpha}+ (\eta-1)\overline\kappa^2\big)
\end{pmatrix}
\begin{pmatrix}
\wh\kappa_2\\
\wh\kappa_3
\end{pmatrix} \,.
\label{eq:kap2kap3_nob}
\end{align}
We may explicitly write down the eigenvalues of the matrix in \eqref{eq:kap2kap3_nob} as 
\begin{equation}\label{eq:mu1mu2}
\begin{aligned}
\mu_2\,, \mu_3&= -\frac{(2\pi n)^2}{2}\bigg[ (2\pi n)^2 + \frac{\eta}{\alpha}+ (\eta-1)\overline\kappa^2 \\
&\qquad \pm \sqrt{\bigg((2\pi n)^2 - \bigg(\frac{\eta}{\alpha}+ (\eta-1)\overline\kappa^2\bigg)\bigg)^2
+ 4\overline\kappa^2\bigg(\frac{\eta}{\alpha}+ (2\pi n)^2(\eta-1) \bigg)}\bigg]\\
&=-\frac{(2\pi n)^2}{2}\bigg[ (2\pi n)^2 + \frac{\eta}{\alpha}+ (\eta-1)\overline\kappa^2 \\
&\qquad \pm \sqrt{\bigg( (2\pi n)^2 +\frac{\eta}{\alpha}+ (\eta-1)\overline\kappa^2 \bigg)^2
+ 4\frac{\eta}{\alpha}\big(\overline\kappa^2
- (2\pi n)^2 \big)}\bigg]\,.
\end{aligned}
\end{equation}
As $n\to\infty$, we obtain the same leading order behavior as \eqref{eq:muks}. Here, however, we can see more precisely how the strong coupling between the $\kappa_2$ and $\kappa_3$ equations exactly cancels the backward heat part $(2\pi n)^2\overline\kappa^2(\eta-1)$ appearing in \eqref{eq:kap2kap3_nob}.

Our aim with the above spectral calculations is to highlight the possible regularizing role of the strong curve-frame coupling, and in particular, the apparent cancellation of the backward heat behavior through this coupling. Although our focus here is on the asymptotic spectral behavior as $n\to\infty$, we note that there are a variety of low wavenumber instabilities that are interesting in their own right. For the case $b=0$ above, we can see from the expression \eqref{eq:mu1mu2} that for large enough $\overline\kappa$, $\mu_3$ is positive for small $n$, so the multiply-covered circle is linearly unstable to long-wave perturbations. This is related to the well-studied Michell's instability \cite{michell1890small,goriely2006twisted,zajac1962stability,bartels2020numerical}, which arises due to the balance between the bending and twisting energy of the rod. For Michell's instability, a build-up of twist energy in a closed planar ring is relieved by deformation, i.e., releasing some of the extra twist energy into bending. Here we see the analogous phenomenon: a build-up of bending energy may be relieved by twisting the rod.

\appendix
\numberwithin{theorem}{section}
%!TEX root = RFT GWP.tex

\section{Hyperviscous regularization}
\label{sec:hyperviscous}

We consider \eqref{eq:Xss_0-repeat}-\eqref{eq:tension_0-repeat} with an additional hyperviscous term in the $\kappa_3$ equation: 
\begin{align}
    \p_t\X_s + (\X_s)_{ssss}
    &=  \big(\mc{R}^{\rm c}_1[\X_s,\kappa_3] \big)_s + \big(\mc{R}^{\rm c}_2[\X_s,\kappa_3] \big)_{ss} \label{eq:Xs_evol}\\
    \p_t\kappa_3 + \nu(\kappa_3)_{ssss}
    &= \mc{R}^{\rm f}_0[\X_s,\kappa_3] + \big(\mc{R}^{\rm f}_1[\X_s,\kappa_3]\big)_s + \big(\mc{R}^{\rm f}_2[\X_s,\kappa_3]\big)_{ss}\,. \label{eq:kap3_evol}
\end{align}
Here $\nu>0$ is a parameter that we will eventually send to zero. The remainder terms are given by
\begin{align}
    \mc{R}^{\rm c}_1[\X_s,\kappa_3] &= -\gamma\lambda\X_{ss} +  3\gamma\X_{s}(\X_{ss}\cdot\X_{sss}) \\
    \mc{R}^{\rm c}_2[\X_s,\kappa_3] &= (1+\gamma)\lambda\X_s + \eta\kappa_3\X_s\times\X_{ss} \\
    \mc{R}^{\rm f}_0[\X_s,\kappa_3] &= -2\eta\X_{ss}\cdot\X_{sss}(\kappa_3)_s - \eta\kappa_3\X_{ssss}\cdot\X_{ss} -2\eta\kappa_3\abs{\X_{sss}}^2   \nonumber \\
    &\quad 
    + \eta\kappa_3\abs{\X_{ss}}^4
      +\X_{ssss} \cdot (\X_s\times\X_{sss}) +\lambda\X_{sss}\cdot (\X_s\times\X_{ss}) \\  
    \mc{R}^{\rm f}_1[\X_s,\kappa_3] &= -\X_{ssss} \cdot (\X_s\times\X_{ss}) \\  
    \mc{R}^{\rm f}_2[\X_s,\kappa_3] &= \bigg(\frac{\eta}{\alpha}+\eta\abs{\X_{ss}}^2\bigg)\kappa_3 \,.
\end{align}
The filament tension $\lambda$ continues to satisfy \eqref{eq:tension_0-repeat}, which we rewrite here for convenience: 
\begin{equation}\label{eq:tension3}
    (1+\gamma)\lambda_{ss}  -\abs{\X_{ss}}^2\lambda = 
    -(4+3\gamma)(\X_{sss}\cdot\X_{ss})_s +\abs{\X_{sss}}^2  - \eta\kappa_3\X_s\cdot(\X_{ss}\times\X_{sss})\,.
\end{equation}

We show the following qualitative existence and uniqueness result.
\begin{lemma}[Qualitative short-time theory with hyperviscosity]
    \label{lem:qualitativeexistencehyperviscosity}
Let
\begin{equation}
\X_{s}^{\rm in} \in H^3_s \, , \quad |\X_s^{\rm in}|^2 = 1 \, ,
\quad
\kappa_3^{\rm in} \in H^1_s\,,
\end{equation}
satisfying the boundary conditions, and consider \eqref{eq:Xs_evol}-\eqref{eq:kap3_evol} for some $\nu>0$. Then there exists $T = T(\X^{\rm in},\kappa_3^{\rm in},\nu) > 0$ and a solution
\begin{equation}
\X_{s} \in L^\infty_t H^3_s \cap L^2_t H^5_s\,, \quad 
\kappa_3 \in L^\infty_t H^1_s \cap L^2_t H^3_s
\end{equation}
to the hyperviscous problem. The solution is unique in this class and can be extended to a maximal solution provided the norms do not blow up.
\end{lemma}

%%%%%%%%%%%%%
\begin{proof}[Proof of Lemma~\ref{lem:qualitativeexistencehyperviscosity}]
We will consider the pair $(\X_s,\kappa_3)$ belonging to the space 
\begin{equation}
    \label{eq:Yspaceinappendix}
    \mc{Y}= (L^\infty_tH^2_s\cap L^2_tH^4_s)\times (L^\infty_tH^1_s\cap L^2_tH^3_s) \cap \{ {\rm BCs} \} \,,
\end{equation}
equipped with the norm 
\begin{equation}
    \norm{(\X_s,\kappa_3)}_{\mc{Y}} := \norm{\X_s}_{L^\infty_tH^2_s} + \norm{\X_s}_{L^2_tH^4_s} + \norm{\kappa_3}_{L^\infty_tH^1_s} + \norm{\kappa_3}_{L^2_tH^3_s}\,.
\end{equation}
The space $\mc{Y}$ does not fall exactly into the energy estimates in Section~\ref{sec:lineartheory}, and we instead appeal to the estimates in Lemma~\ref{lem:semigroupestimates}. 
To close a fixed point argument for $(\X_s,\kappa_3)$ in $\mc{Y}$, we will need bounds for $\norm{\mc{R}^{\rm c}_1}_{L^{4/3}_tH^2_s}$, $\norm{\mc{R}^{\rm c}_2}_{L^2_tH^2_s}$, $\|\mc{R}^{\rm f}_0\|_{L^1_tH^1_s}$, $\|\mc{R}^{\rm f}_0\|_{L^{4/3}_tH^1_s}$, and $\|\mc{R}^{\rm f}_2\|_{L^2_tH^1_s}$. We will not explicitly track the dependence of these estimates on the parameter $\nu$ as $\nu\to 0$. 

\textbf{Stabilizes a ball}. We begin by estimating the tension $\lambda$ in terms of $\X_{ss}$ and $\kappa_3$. By analogous arguments to Lemma \ref{lem:tension} using the form of \eqref{eq:tension3}, we obtain 
\begin{equation}\label{eq:Y_H1_lam}
    \norm{\lambda}_{H^1_s} \lesssim \norm{\X_s}_{H^2_s}^2(1 + \norm{\kappa_3}_{H^1_s})\,.
\end{equation}
We may also obtain the $H^2_s$ bound
\begin{equation}\label{eq:Y_H2_lam}
    \norm{\lambda}_{H^2_s} 
    \lesssim (\norm{\X_s}_{H^4_s}^{1/2}\norm{\X_s}_{H^2_s}^{3/2} + \norm{\X_s}_{H^2_s}^4)(1 + \norm{\kappa_3}_{H^1_s}) 
\end{equation}
directly from the equation \eqref{eq:tension3}.
Using the tension estimates \eqref{eq:Y_H1_lam} and \eqref{eq:Y_H2_lam}, we may obtain the following $H^2_s$ bounds for $\mc{R}^{\rm c}_1$ and $\mc{R}^{\rm c}_2$:
\begin{align}
    \norm{\mc{R}^{\rm c}_1}_{H^2_s} &\lesssim 
    \norm{\lambda}_{H^2_s}\norm{\X_s}_{H^2_s} + \norm{\lambda}_{H^1_s}\norm{\X_s}_{H^2_s}^{1/2}\norm{\X_s}_{H^4_s}^{1/2} \nonumber \\
    &\quad + \norm{\X_s}_{H^2_s}^{5/2}\norm{\X_s}_{H^4_s}^{1/2} + \norm{\X_s}_{H^2_s}\norm{\X_s}_{H^4_s} \nonumber \\
    &\lesssim 
    (\norm{\X_s}_{H^4_s}^{1/2}\norm{\X_s}_{H^2_s}^{5/2} + \norm{\X_s}_{H^2_s}^5)(1 + \norm{\kappa_3}_{H^1_s}) + \norm{\X_s}_{H^2_s}\norm{\X_s}_{H^4_s} \,, \label{eq:Rc1}\\
    \norm{\mc{R}^{\rm c}_2}_{H^2_s} &\lesssim
    \norm{\lambda}_{H^2_s} + \norm{\lambda}_{H^1_s}\norm{\X_s}_{H^2_s}
    + \norm{\kappa_3}_{H^3_s}^{1/2}\norm{\kappa_3}_{H^1_s}^{1/2}\norm{\X_s}_{H^2_s} \nonumber \\
    &\quad + \norm{\kappa_3}_{H^1_s}\norm{\X_s}_{H^2_s}^2
    + \norm{\kappa_3}_{H^1_s}\norm{\X_s}_{H^4_s}^{1/2}\norm{\X_s}_{H^2_s}^{1/2} \nonumber \\
    &\lesssim \big(\norm{\X_s}_{H^4_s}^{1/2}(\norm{\X_s}_{H^2_s}^{3/2}+ \norm{\X_s}_{H^2_s}^{1/2}) + \norm{\X_s}_{H^2_s}^4\big)(1 + \norm{\kappa_3}_{H^1_s})  \label{eq:Rc2} \\
    &\quad + \norm{\kappa_3}_{H^3_s}^{1/2}\norm{\kappa_3}_{H^1_s}^{1/2}\norm{\X_s}_{H^2_s}\,. \nonumber
\end{align}
We may further obtain the following $H^1_s$ bounds for $\mc{R}^{\rm f}_0$, $\mc{R}^{\rm f}_1$, and $\mc{R}^{\rm f}_2$:
\begin{align}
    \norm{\mc{R}^{\rm f}_0}_{H^1_s} &\lesssim 
    \norm{\X_s}_{H^2_s}\norm{\X_s}_{H^4_s}(1+\norm{\kappa_3}_{H^1_s}) 
    +\norm{\X_s}_{H^2_s}^2\norm{\kappa_3}_{H^1_s}^{1/2}\norm{\kappa_3}_{H^3_s}^{1/2} \nonumber \\
    &\quad 
    + \norm{\kappa_3}_{H^1_s}\norm{\X_s}_{H^2_s}^4 
    + \norm{\X_s}_{H^4_s}^{1/2}\norm{\X_s}_{H^2_s}^{5/2}
    +\norm{\lambda}_{H^1_s}\norm{\X_s}_{H^2_s}^{3/2}\norm{\X_s}_{H^4_s}^{1/2} \nonumber \\
    &\lesssim 
    \big(\norm{\X_s}_{H^2_s}\norm{\X_s}_{H^4_s} + \norm{\X_s}_{H^4_s}^{1/2}\norm{\X_s}_{H^2_s}^{7/2}\big)(1+\norm{\kappa_3}_{H^1_s}) \label{eq:Rf0} \\
    &\quad +\norm{\X_s}_{H^2_s}^2\norm{\kappa_3}_{H^1_s}^{1/2}\norm{\kappa_3}_{H^3_s}^{1/2} \,, \nonumber \\
    \norm{\mc{R}^{\rm f}_1}_{H^1_s} &\lesssim \norm{\X_s}_{H^2_s}\norm{\X_s}_{H^4_s} \,, \label{eq:Rf1} \\
    \norm{\mc{R}^{\rm f}_2}_{H^1_s} &\lesssim \norm{\kappa_3}_{H^1_s}(1+\norm{\X_s}_{H^2_s}^2) \,. \label{eq:Rf2}
\end{align}

Let $S_T(\X_{ss},\kappa_3)$ denote the time-$T$ solution map for the system \eqref{eq:Xs_evol}-\eqref{eq:kap3_evol}. We may now show that $S_T(\cdot,\cdot)$ maps a ball to a ball in $\mc{Y}$. Let $\norm{(\X_{ss},\kappa_3)}_{\mc{Y}} \le N$ for some $N>0$.
Integrating the inequalities \eqref{eq:Rc1}, \eqref{eq:Rc2}, \eqref{eq:Rf0}, \eqref{eq:Rf1}, and \eqref{eq:Rf2} in time from 0 to $T$, we may estimate 
\begin{align}
    \norm{\mc{R}^{\rm c}_1}_{L^{4/3}_tH^2_s} &\lesssim 
    (T^{1/2}\norm{\X_s}_{L^2_tH^4_s}^{1/2}\norm{\X_s}_{L^\infty_tH^2_s}^{5/2} + T^{3/4}\norm{\X_s}_{L^\infty_tH^2_s}^5)(1 + \norm{\kappa_3}_{L^\infty_tH^1_s}) \nonumber \\
    &\quad  + T^{1/4}\norm{\X_s}_{L^\infty_tH^2_s}\norm{\X_s}_{L^2_tH^4_s}  \nonumber \\
    &\lesssim  T^{1/2}N^3(1 + T^{1/4}N^2)(1 + N) + T^{1/4}N^2 \,, \label{eq:Rc1T}\\
    \norm{\mc{R}^{\rm c}_2}_{L^2_tH^2_s} &\lesssim
    (1 + \norm{\kappa_3}_{L^\infty_tH^1_s})\big(T^{1/4} \norm{\X_s}_{L^2_tH^4_s}^{1/2}(\norm{\X_s}_{L^\infty_tH^2_s}^{3/2}+ \norm{\X_s}_{L^\infty_tH^2_s}^{1/2}) \nonumber \\
    &\quad + T^{1/2}\norm{\X_s}_{L^\infty_tH^2_s}^4\big) + T^{1/4}\norm{\kappa_3}_{L^2_tH^3_s}^{1/2}\norm{\kappa_3}_{L^\infty_tH^1_s}^{1/2}\norm{\X_s}_{L^\infty_tH^2_s}  \nonumber \\
    &\lesssim 
    T^{1/4} N \big(1+N + T^{1/4}N^3\big)(1 + N) \,, \label{eq:Rc2T} \\
    \norm{\mc{R}^{\rm f}_0}_{L^1_tH^1_s} &\lesssim 
    (1+\norm{\kappa_3}_{L^\infty_tH^1_s})\big(T^{1/2}\norm{\X_s}_{L^\infty_tH^2_s}\norm{\X_s}_{L^2_tH^4_s}  \nonumber \\
    &\quad + T^{3/4}\norm{\X_s}_{L^2_tH^4_s}^{1/2}\norm{\X_s}_{L^\infty_tH^2_s}^{7/2}\big) + T^{3/4}\norm{\X_s}_{L^\infty_tH^2_s}^2\norm{\kappa_3}_{L^\infty_tH^1_s}^{1/2}\norm{\kappa_3}_{L^2_tH^3_s}^{1/2}  \nonumber \\
    &\lesssim 
    T^{1/2}N^2\big(1+ T^{1/4}N^2\big)(1+N) + T^{3/4}N^3\,, \label{eq:Rf0T} \\
    \norm{\mc{R}^{\rm f}_1}_{L^{4/3}_tH^1_s} &\lesssim T^{1/4}\norm{\X_s}_{L^\infty_tH^2_s}\norm{\X_s}_{L^2_tH^4_s} 
    \lesssim T^{1/4}N^2\,, \label{eq:Rf1T} \\
    \norm{\mc{R}^{\rm f}_2}_{L^2_tH^1_s} &\lesssim T^{1/2}\norm{\kappa_3}_{L^\infty_tH^1_s}(1+\norm{\X_s}_{L^\infty_tH^2_s}^2) 
    \lesssim  T^{1/2}N(1+N^2) \,. \label{eq:Rf2T}
\end{align}
Using the form of the estimates in Lemma~\ref{lem:semigroupestimates} for \eqref{eq:Xs_evol} and \eqref{eq:kap3_evol}, we thus obtain 
\begin{equation}
\begin{aligned}
    \norm{S_T(\X_{ss},\kappa_3)}_\mc{Y}
    &\le \norm{\X_s^{\rm in}}_{H^2_s} + \norm{\kappa_3^{\rm in}}_{H^1_s} + C(\norm{\mc{R}^{\rm c}_1}_{L^{4/3}_tH^2_s} + \norm{\mc{R}^{\rm c}_2}_{L^2_tH^2_s}) \\
    &\quad + C(\norm{\mc{R}^{\rm f}_0}_{L^1_tH^1_s} + \norm{\mc{R}^{\rm f}_1}_{L^{4/3}_tH^1_s}+ \norm{\mc{R}^{\rm f}_2}_{L^2_tH^1_s}) \\
    &\le \norm{\X_s^{\rm in}}_{H^2_s} + \norm{\kappa_3^{\rm in}}_{H^1_s} \\
    &\quad + C\, T^{1/4}\,N \big( 1+N^2 + T^{1/4}N^2(1 + T^{1/4}N^2)(1 + N^4) \big)  \,.
\end{aligned}
\end{equation}
Choosing $N=2(\norm{\X_s^{\rm in}}_{H^2_s} + \norm{\kappa_3^{\rm in}}_{H^1_s})$ and taking $T$ small enough that $C\, T^{1/4}\,N \big( 1+N^2 + T^{1/4}N^2(1 + T^{1/4}N^2)(1 + N^4) \big)\le \frac{N}{4}$\,, we obtain 
\begin{equation}\label{eq:ball2ball}
    \norm{S_T(\X_s,\kappa_3)}_{\mc{Y}} \le \frac{3}{4}N\,.
\end{equation}

\textbf{Contraction}. We next show that the solution map is a contraction on the ball of radius $N$ in $\mc{Y}$. Consider two curve-frame pairs $(\X_s^{(1)},\kappa_3^{(1)})$ and $(\X_s^{(2)},\kappa_3^{(2)})$ belonging to this ball. We again begin with the tension, and again follow the curve-only setting. As in \eqref{eq:Rlambdak}, we let $R_\lambda^{(k)}$ denote the right-hand side of the tension equation \eqref{eq:tension3} for curve-frame pair $(\X_s^{(k)},\kappa_3^{(k)})$, $k=1,2$. We may estimate the difference  
\begin{equation}\label{eq:Y_Rlam_diff}
\begin{aligned}
    \norm{R_\lambda^{(1)}- R_\lambda^{(2)}}_{(H^1)^*} &\lesssim  \norm{\X_{s}^{(1)}-\X_{s}^{(2)}}_{H^2_s} \big(1 + (1+ \norm{\X_{s}^{(1)}}_{H^2_s}) \|\kappa_3^{(2)}\|_{H^1_s} \big)\times  \\ 
    &\qquad 
    \times \big(\norm{\X_{s}^{(1)}}_{H^2_s} +\norm{\X_{s}^{(2)}}_{H^2_s}\big)
    + \norm{\X_{s}^{(1)}}_{H^2_s}^2\|\kappa_3^{(1)}- \kappa_3^{(2)}\|_{H^1_s} \,.
\end{aligned}
\end{equation}
Defining $q^{(1)}$, $q^{(2)}$ as in \eqref{eq:qdef}, by analogous arguments, we obtain the estimate
\begin{equation}\label{eq:qdiff}
\begin{aligned}
    \norm{q^{(1)}-q^{(2)}}_{H^1_s}&\lesssim \big(\norm{\X_{s}^{(1)}}_{H^2_s} +\norm{\X_{s}^{(2)}}_{H^2_s}\big)\times\\
    &\quad \times\norm{\X_{s}^{(1)}-\X_{s}^{(2)}}_{H^2_s}
    \norm{\X_{s}^{(1)}}_{H^2_s}^2(1 + \|\kappa_3^{(1)}\|_{H^1_s})\,.
\end{aligned}
\end{equation}
Combining \eqref{eq:Y_Rlam_diff} and \eqref{eq:qdiff}, we may then obtain the following bound for $\lambda^{(1)}-\lambda^{(2)}$ on $\mc{Y}$:
% \begin{equation}\label{eq:lamdiff_H1}
% \begin{aligned}
%     \norm{\lambda^{(1)}-\lambda^{(2)}}_{H^1_s} 
%     &\lesssim \big(\norm{\X_{s}^{(1)}}_{H^2_s} +\norm{\X_{s}^{(2)}}_{H^2_s}\big)\times\\
%     &\quad \times\big(1 + \norm{\X_{s}^{(1)}}_{H^2_s}^2\big) \big(1+ \|\kappa_3^{(1)}\|_{H^1_s} +\|\kappa_3^{(2)}\|_{H^1_s}\big)\norm{\X_{s}^{(1)}-\X_{s}^{(2)}}_{H^2_s} \\
%     &\quad + \norm{\X_{s}^{(1)}}_{H^2_s}^2\|\kappa_3^{(1)}- \kappa_3^{(2)}\|_{H^1_s} \,.
% \end{aligned}
% \end{equation}
\begin{align}
    \norm{\lambda^{(1)}-\lambda^{(2)}}_{L^\infty_tH^1_s} 
    \lesssim N(1 + N^3) \norm{\X_{s}^{(1)}-\X_{s}^{(2)}}_{L^\infty_tH^2_s} 
     + N^2\|\kappa_3^{(1)}- \kappa_3^{(2)}\|_{L^\infty_tH^1_s} \,. \label{eq:Y_lamdiff_H1} 
%
    % \norm{\lambda^{(1)}-\lambda^{(2)}}_{L^2_tH^2_s} 
    % &\lesssim  T^{1/4}\,N\big(1 + T^{1/4}\,N(1+N^4)\big)\norm{(\X_{s}^{(1)}-\X_{s}^{(2)},\kappa_3^{(1)}- \kappa_3^{(2)})}_{\mc{Y}} \,.
    % \label{eq:Y_lamdiff_H2}
\end{align}
% \begin{align*}
%     \norm{\lambda^{(1)}-\lambda^{(2)}}_{H^2_s} &\le c(\eta,\gamma)\bigg( \norm{\X_{ss}^{(1)}-\X_{ss}^{(2)}}_{H^3_s}^{1/2}\norm{\X_{ss}^{(1)}-\X_{ss}^{(2)}}_{H^1_s}^{1/2}\norm{\X_{ss}^{(1)}}_{H^1_s} 
%     + \norm{\X_{ss}^{(2)}}_{H^3_s}^{1/2}\norm{\X_{ss}^{(2)}}_{H^1_s}^{1/2}\norm{\X_{ss}^{(1)}-\X_{ss}^{(2)}}_{H^1_s}\\
%     & + \norm{\kappa_3^{(1)}-\kappa_3^{(2)}}_{H^1_s}\norm{\X_{ss}^{(1)}}_{H^1_s}^2
%     + \norm{\kappa_3^{(2)}}_{H^1_s}\norm{\X_{ss}^{(1)}-\X_{ss}^{(2)}}_{H^1_s} (\norm{\X_{ss}^{(2)}}_{H^1_s}+ \norm{\X_{ss}^{(1)}}_{H^1_s}^2)\\
%     & + \norm{\X_{ss}^{(1)}-\X_{ss}^{(2)}}_{H^1_s}(\norm{\X_{ss}^{(1)}}_{H^1_s} +\norm{\X_{ss}^{(2)}}_{H^1_s})\norm{\lambda^{(1)}}_{H^1_s}
%     + \norm{\X_{ss}^{(2)}}_{H^1_s}^2\norm{\lambda^{(1)}-\lambda^{(2)}}_{H^1_s}
% \end{align*}
%
%
%
From the tension equation \eqref{eq:tension3} and the $L^\infty_tH^1_s$ bound \eqref{eq:Y_lamdiff_H1}, we may additionally obtain the following $L^2_tH^2_s$ estimate for $\lambda^{(1)}-\lambda^{(2)}$ on $\mc{Y}$:
\begin{equation}\label{eq:Y_lamdiff_H2}
\begin{aligned}
    \norm{\lambda^{(1)}-\lambda^{(2)}}_{L^2_tH^2_s} 
    &\lesssim  T^{1/4}\,N\bigg( \norm{\X_{s}^{(1)}-\X_{s}^{(2)}}_{L^2_tH^4_s}^{1/2}\norm{\X_{s}^{(1)}-\X_{s}^{(2)}}_{L^\infty_tH^2_s}^{1/2} \\
    &\quad +\norm{\X_{s}^{(1)}-\X_{s}^{(2)}}_{L^\infty_tH^2_s} 
    \bigg) + T^{1/2}\,N^2\bigg((1+N^2)\|\kappa_3^{(1)}- \kappa_3^{(2)}\|_{L^\infty_tH^1_s} \\
    &\quad + \big(1+ N^4\big) \norm{\X_{s}^{(1)}-\X_{s}^{(2)}}_{L^\infty_tH^2_s} \bigg)\,.
\end{aligned}
\end{equation}
Furthermore, each of the right-hand side differences $(\mc{R}^{\rm c}_j)^{(1)} -(\mc{R}^{\rm c}_j)^{(2)}$ and $(\mc{R}^{\rm f}_j)^{(1)} -(\mc{R}^{\rm f}_j)^{(2)}$ may be shown to satisfy
\begin{equation}
\begin{aligned}
&\norm{(\mc{R}^{\rm c}_1)^{(1)}-(\mc{R}^{\rm c}_1)^{(2)}}_{L^{4/3}_tH^2_s} \lesssim 
    T^{1/4}N\norm{\lambda^{(1)}-\lambda^{(2)}}_{L^2_tH^2_s}
    + T^{1/2}N\norm{\lambda^{(1)}-\lambda^{(2)}}_{L^\infty_tH^1_s} \\
    &\qquad + T^{1/4}N\big(1+ T^{1/4}N(1+T^{1/4}N^2)(1+N)\big)\norm{(\X_{s}^{(1)}-\X_{s}^{(2)},\kappa_3^{(1)}- \kappa_3^{(2)})}_{\mc{Y}} \,,\\
&\norm{(\mc{R}^{\rm c}_2)^{(1)}-(\mc{R}^{\rm c}_2)^{(2)}}_{L^2_tH^2_s} \lesssim
    \norm{\lambda^{(1)}-\lambda^{(2)}}_{L^2_tH^2_s} + T^{1/2}N\norm{\lambda^{(1)}-\lambda^{(2)}}_{L^\infty_tH^1_s} \\
    &\qquad + T^{1/4}N\big(1+T^{1/4}(1+N^2)\big)(1+N^2)\norm{(\X_{s}^{(1)}-\X_{s}^{(2)},\kappa_3^{(1)}- \kappa_3^{(2)})}_{\mc{Y}}\,, \\
&\norm{(\mc{R}^{\rm f}_0)^{(1)}-(\mc{R}^{\rm f}_0)^{(2)}}_{L^1_tH^1_s} \lesssim 
    T^{1/2}\big(1+ T^{1/4}N\big)(N+N^3)\norm{(\X_{s}^{(1)}-\X_{s}^{(2)},\kappa_3^{(1)}- \kappa_3^{(2)})}_{\mc{Y}} \,, \\
& \norm{(\mc{R}^{\rm f}_1)^{(1)}-(\mc{R}^{\rm f}_1)^{(2)}}_{L^{4/3}_tH^1_s} 
    \lesssim T^{1/4}N\norm{(\X_{s}^{(1)}-\X_{s}^{(2)},\kappa_3^{(1)}- \kappa_3^{(2)})}_{\mc{Y}} \,, \\
&\norm{(\mc{R}^{\rm f}_2)^{(1)}-(\mc{R}^{\rm f}_2)^{(2)}}_{L^2_tH^1_s} \lesssim 
    T^{1/2}(1+N^2)\norm{(\X_{s}^{(1)}-\X_{s}^{(2)},\kappa_3^{(1)}- \kappa_3^{(2)})}_{\mc{Y}} \,.
\end{aligned}
\end{equation}

Altogether, using Lemma~\ref{lem:semigroupestimates}, we have
\begin{equation}
\begin{aligned}
    \norm{S_T(\X_{s}^{(1)} -\X_{s}^{(2)},\kappa_3^{(1)} -\kappa_3^{(2)})}_{\mc{Y}} & \le C\big(T^{1/4}(N+N^3)+ T^{1/2}(1+N^6) \\
    & + T^{3/4}(N^2+N^7) \big)\norm{(\X_{s}^{(1)}-\X_{s}^{(2)},\kappa_3^{(1)}- \kappa_3^{(2)})}_{\mc{Y}} \,.
\end{aligned}
\end{equation}
Taking $T$ small enough that $C\big(T^{1/4}(N+N^3)+ T^{1/2}(1+N^6) + T^{3/4}(N^2+N^7)\big)<1$, we obtain a contraction on $\mc{Y}$, and hence the existence of a unique solution $(\X_s,\kappa_3)\in (L^\infty_tH^2_s\cap L^2_tH^4_s)\times(L^\infty_tH^1_s\cap L^2_tH^3_s)$.

\textbf{Bootstrapping}. We next apply a bootstrapping argument to demonstrate, qualitatively, that
\begin{equation}\label{eq:qualboot}
\X_s \in L^\infty_tH^3_s\cap L^2_tH^5_s(I \times (0,T)) \, . % ) \times (L^\infty_tH^1_s\cap L^2_tH^3_s)
\end{equation}
This will be enough to justify the quantitative energy estimates in Section \ref{sec:rod3D}. We again consider the $\X_s$ equation~\eqref{eq:Xs_evol}. According to Lemma~\ref{lem:parabolicestimates},~\eqref{eq:qualboot} will hold provided that $(\mathcal{R}^{\rm c}_1)_s, (\mathcal{R}^{\rm c}_2)_{ss} \in L^2_t H^1_s$. %Since $\mathcal{D}(L^{1/4})$ encodes no boundary condition,
It suffices to verify $\mathcal{R}^{\rm c}_1 \in L^2_t H^2_s$ and $\mathcal{R}^{\rm c}_2 \in L^2_t H^3_s$ based on our current information that $\X_s \in L^\infty_t H^2_s \cap L^2_t H^4_s$ and $\kappa_3 \in L^\infty_t H^1_s \cap L^2_t H^3_s$. Throughout, we rely on the tame estimates
\begin{equation}
\| fg \|_{H^k} \lesssim \| f \|_{L^\infty} \| g \|_{H^k} + \| f \|_{H^k} \| g \|_{L^\infty}
\end{equation}
and 1D Sobolev embedding into $L^\infty$. The main non-trivial point to verify is that $\lambda \in L^2_t H^3_s$. For this, we examine the equation
\begin{equation}
(1+\gamma) \lambda_{ss} = |\X_{ss}|^2 \lambda - (4+3\gamma) (\X_{sss} \cdot \X_{ss})_s + |\X_{sss}|^2 - \eta \kappa_3 \X_s (\X_{ss} \times \X_{sss}) \, .
\end{equation}
Under the current assumptions, the right-hand side belongs to $L^2_t H^1_s$. This completes the bootstrapping argument.
\end{proof}

\subsection{Parabolic estimates, revisited}
We prove an extension of Lemma~\ref{lem:parabolicestimates}.
\begin{lemma}
    \label{lem:semigroupestimates}
Let $T > 0$ and $I = [0,1]$ or $\T$. Suppose $u^{\rm in} \in H^1_0$, $f_0 \in L^1_t (H^1_0)_s$, $f_1 \in L^{4/3}_t H^1_s$, and $f_2 \in L^2_t H^1_s$ (with norms evaluated on $I \times (0,T)$). Then the solution $u$ to the Dirichlet biharmonic problem
\begin{equation}
    \label{eq:biharmonictointerpolate}
\begin{aligned}
(\p_t + \p_s^4) u &= f_0 + \p_s f_1 + \p_s^2 f_2 \\
u\,,\,\p_s u\big|_{\p I} &= 0 \\
u\big|_{t=0} &= u^{\rm in}
\end{aligned}
\end{equation}
belongs to $C([0,T];H^1_s) \cap L^2_t H^3_s$, and
\begin{equation}
    \label{eq:myinterpolation}
\| u \|_{L^\infty_t H^1_s} + \| \p_s^2 u \|_{L^2_t H^1_s} \lesssim \| f_0 \|_{L^1_t (H^1_0)_s} + \| f_1 \|_{L^{4/3}_t H^1_s} + \| f_2 \|_{L^2_t H^1_s} \, .
\end{equation}
\end{lemma}
\begin{proof}
We concern ourselves only with $I=[0,1]$, since the periodic estimates can be justified on the basis of Fourier series. Consider the unbounded operator $L : D(L) \subset L^2_s \to L^2_s$ defined~by 
\begin{equation}
L = \p_s^4 \, , \quad D(L) := \{ H^4_s(I) : u, \p_s u|_{s=0,1} = 0 \} \, .
\end{equation}
Then $L$ is positive and self-adjoint, and $-L$ generates an analytic semigroup. The eigenvalues $\lambda_k$ and $L^2_s$-normalized eigenfunctions $\phi_k$, $k \in \N$, were computed in~\cite{mori2023well}. The fractional powers of $L$ are expressed in terms of the eigenfunction expansions
\begin{equation}
L^\theta u = \sum_{k \in \N} \lambda_k^\theta \langle u_k, \phi_k \rangle \phi_k(s) \, ,
\end{equation}
where $\theta \geq 0$ and $u \in \mathcal{D}(L^\theta)$ consists of those $u \in L^2$ satisfying $(\lambda_k^\theta |\langle u_k, \phi_k \rangle|)_{k \in \N} \in \ell^2(\N)$. Crucially, we have the following characterization of the domains of fractional powers of the operators~\cite{Grisvard1967}:
\begin{equation}
\mathcal{D}(L^\theta) := H^\sigma \cap H^{\min(\sigma,2)}_0 \, , \quad \forall \sigma = 4 \theta \in [0,4] \, ,
\end{equation}
% If $\theta = k+s'/4$, $k \in \N_0$ and $s' \in [0,4)$, then $u \in \mathcal{D}(L^{k+s'/4})$ if $u, Lu, \hdots, L^k u \in L^2$ and $L^k u \in \mathcal{D}(L^{s'/4})$.
with the norm equivalence
\begin{equation}
\| u \|_{H^\sigma_s} \simeq_\sigma \| u \|_{L^2_s} + \| L^\theta u \|_{L^2_s} \, , \quad \forall u \in \mathcal{D}(L^\theta) \, .
\end{equation}
We can moreover identify domains of $L^{-\theta}$, suitable defined, with $H^{-4\theta}$ when $4\theta \leq 5/2$. With this notation, the standard energy estimate\footnote{One can prove the estimates also for forcing terms in the Chemin-Lerner spaces $\wt{L}^q_t \mathcal{D}(L^\theta)_s$, which are controlled by $L^q_t \mathcal{D}(L^\theta)_s$ when $q \leq 2$:
\begin{equation}
\| g \|_{L^q_t \mathcal{D}(L^\theta)_s} = \| \| ( \lambda_k^{\theta} g_k ) \|_{\ell^2(\N)} \|_{L^q_t} \geq \| ( \lambda_k^{\theta} \| g_k \|_{L^q_t} ) \|_{\ell^2(\N)} =: \| g \|_{\wt{L}^q_t \mathcal{D}(L^\theta)_s} \, .
 \end{equation}}
 in $C([0,T];L^2) \cap L^2_t \dot H^2_s(I \times (0,T))$ for the Dirichlet problem
\begin{equation}
\begin{aligned}
(\p_t + \p_s^4) u &= g^{(0)} + g^{(1)} + g^{(2)} \\
u,\p_s u|_{s=0,1} &= 0 \\
u|_{t=0} = u^{\rm in} \in H^2_0(I)
\end{aligned}
\end{equation}
with
\begin{equation}
g^{(0)} \in L^1_t L^2_s \, , \quad g^{(1)} \in L^{4/3}_t H^{-1}_s \, , \quad g^{(2)} \in L^2_t H^{-2}_s
\end{equation}
can be ``raised" by simply applying $L^\theta$ to the equation. The estimate~\eqref{eq:myinterpolation} follows from the choice $\theta = 1/4$. \end{proof}

\section{Curvature formulation with forcing}\label{app:forcing}
Here we record the full expressions for the rod evolution \eqref{eq:kap1}-\eqref{eq:kap3} forced by intrinsic curvature $\zeta_1(s,t)$, $\zeta_2(s,t)$ and twist $\zeta_3(s,t)$. Writing out the system as an evolution for $(\kappa_1,\kappa_2,\kappa_3)$ highlights the potential backward heat operator in the $\kappa_3$ evolution and emphasizes the complicated effects of forcing terms. 

Recalling the expressions \eqref{eq:f_form} and \eqref{eq:u_RFT}, we may calculate 
\begin{align}
    &\bu_s\cdot\e_1 = 
    -(\kappa_1 - \zeta_1)_{sss} + (\kappa_3)_{ss}(\kappa_2 - \zeta_2) -\eta\kappa_2(\kappa_3 - \zeta_3)_{ss} +\mc{R}_{\rm a}[\kappa_j,\zeta_j,\lambda]  \label{eq:us_e1}\\
    &\bu_s\cdot\e_2 = 
     -(\kappa_2 - \zeta_2)_{sss} - (\kappa_3)_{ss}(\kappa_1 - \zeta_1) + \eta\kappa_1(\kappa_3 - \zeta_3)_{ss} + \mc{R}_{\rm b}[\kappa_j,\zeta_j,\lambda] \label{eq:us_e2} \\
    &(1+\gamma)\lambda_{ss}- (\kappa_1^2+\kappa_2^2)\lambda = -\mc{R}_{\rm t}[\kappa_j,\zeta_j] \,, \label{eq:wtlambda}
\end{align}
where 
\begin{equation}\label{eq:Rdefs}
\begin{aligned}
    \mc{R}_{\rm a} &= 3\kappa_3(\kappa_2 - \zeta_2)_{ss}-\eta(\kappa_2)_{ss}(\kappa_3 - \zeta_3)+ 3(\kappa_3)_s(\kappa_2 - \zeta_2)_s -2\eta(\kappa_2)_s(\kappa_3 - \zeta_3)_s \\
        &\quad  +3\kappa_3 (\kappa_3)_s(\kappa_1 - \zeta_1)  + 3\kappa_3^2(\kappa_1 - \zeta_1)_s - \eta\big(\kappa_1\kappa_3(\kappa_3 - \zeta_3)\big)_s - \eta\kappa_3\big(\kappa_1(\kappa_3 - \zeta_3)\big)_s\\
        &\quad - \kappa_3^3(\kappa_2 - \zeta_2) +\eta\kappa_2\kappa_3^2(\kappa_3 - \zeta_3)  
        +(1+\gamma) \big[\kappa_3\kappa_1^2(\kappa_2 - \zeta_2)
        -\kappa_1^2(\kappa_1 - \zeta_1)_s \\  
        &\quad -\kappa_1\kappa_2(\kappa_2 - \zeta_2)_s  -\kappa_1\kappa_2\kappa_3(\kappa_1 - \zeta_1) -\kappa_1(\lambda+\kappa_1^2+\kappa_2^2)_s\big]\\
        &\quad + \big((\lambda+\kappa_1^2+\kappa_2^2) \kappa_1\big)_s- (\lambda+\kappa_1^2+\kappa_2^2)\kappa_2\kappa_3 \\
\mc{R}_{\rm b} &=  
    - 3\kappa_3(\kappa_1 - \zeta_1)_{ss}
     + \eta(\kappa_1)_{ss}(\kappa_3 - \zeta_3)
     - 3(\kappa_3)_s(\kappa_1 - \zeta_1)_s
      + 2\eta(\kappa_1)_s(\kappa_3 - \zeta_3)_s\\
    &\quad  + 3\kappa_3 (\kappa_3)_s(\kappa_2 - \zeta_2) +3\kappa_3^2(\kappa_2 - \zeta_2)_s
    -\eta\big(\kappa_2\kappa_3(\kappa_3 - \zeta_3)\big)_s 
    -\eta\kappa_3\big(\kappa_2(\kappa_3 - \zeta_3)\big)_s \\
    &\quad +\kappa_3^3(\kappa_1 - \zeta_1) - \eta\kappa_1\kappa_3^2(\kappa_3 - \zeta_3) +(1+\gamma)\big[-\kappa_2\kappa_1(\kappa_1 - \zeta_1)_s  -\kappa_2^2(\kappa_2 - \zeta_2)_s\\
    &\quad +\kappa_2\kappa_1 \kappa_3(\kappa_2 - \zeta_2)-\kappa_2^2\kappa_3(\kappa_1 - \zeta_1) -\kappa_2(\lambda+\kappa_1^2+\kappa_2^2)_s\big] \\
    &\quad + \big((\lambda+\kappa_1^2+\kappa_2^2) \kappa_2\big)_s+ (\lambda+\kappa_1^2+\kappa_2^2) \kappa_1\kappa_3 \\
\mc{R}_{\rm t} &= (1+\gamma)\big(-\kappa_1(\kappa_1 + \zeta_1)_s- \kappa_2(\kappa_2 + \zeta_2)_s -\kappa_3[\kappa_1 (\kappa_2 - \zeta_2)-\kappa_2(\kappa_1 - \zeta_1)]\big)_s  \\
    &\quad  +(\kappa_1^2+\kappa_2^2)^2+ \kappa_1(\kappa_1 - \zeta_1)_{ss} +\kappa_2(\kappa_2 - \zeta_2)_{ss} 
    +\eta(\kappa_1^2+\kappa_2^2)\kappa_3(\kappa_3 - \zeta_3)
     \\
    &\quad - \eta\big(\kappa_2\big(\kappa_1(\kappa_3 - \zeta_3)\big)_s-\kappa_1\big(\kappa_2(\kappa_3 - \zeta_3)\big)_s\big) +(\kappa_3)_s\big(\kappa_2 (\kappa_1 - \zeta_1)-\kappa_1 (\kappa_2 - \zeta_2)\big)  \\
    &\quad + 2\kappa_3\big(\kappa_2(\kappa_1 - \zeta_1)_s-\kappa_1(\kappa_2 - \zeta_2)_s\big) -\kappa_3^2\big(\kappa_1(\kappa_1 - \zeta_1) +\kappa_2(\kappa_2 - \zeta_2)\big) \,.
\end{aligned}
\end{equation}

Using \eqref{eq:us_e1}, \eqref{eq:us_e2}, and \eqref{eq:T_form}, we may write out the right-hand side of \eqref{eq:kappa1dot}-\eqref{eq:kappa3dot} as
\begin{align}
\p_t \kappa_1 
    &= -(\kappa_1 - \zeta_1)_{ssss} + (\kappa_3)_{sss}(\kappa_2 - \zeta_2)-\eta\kappa_2(\kappa_3 - \zeta_3)_{sss} +\mc{R}_1[\kappa_j,\zeta_j,\lambda] \label{eq:kap1B} \\
\p_t \kappa_2 
    &= -(\kappa_2 - \zeta_2)_{ssss} - (\kappa_3)_{sss}(\kappa_1 - \zeta_1) + \eta\kappa_1(\kappa_3 - \zeta_3)_{sss} + \mc{R}_2[\kappa_j,\zeta_j,\lambda] \label{eq:kap2B} \\
\p_t \kappa_3 
    &= \frac{\eta}{\alpha}(\kappa_3 - \zeta_3)_{ss} +\eta(\kappa_1^2+\kappa_2^2)(\kappa_3 - \zeta_3)_{ss} -\big( \kappa_1(\kappa_1 - \zeta_1)+\kappa_2(\kappa_2 - \zeta_2)\big)(\kappa_3)_{ss}   \nonumber \\
    &\qquad  -\kappa_1(\kappa_2 - \zeta_2)_{sss}+\kappa_2(\kappa_1 - \zeta_1)_{sss} +\mc{R}_3[\kappa_j,\zeta_j,\lambda]  \label{eq:kap3B}\,,
\end{align}
with
\begin{equation}
\begin{aligned}
    \mc{R}_1&= \big(\mc{R}_{\rm a}\big)_s  + (\kappa_3)_{ss}(\kappa_2 - \zeta_2)_s -\eta\big((\kappa_2)_s+\kappa_1\kappa_3\big)(\kappa_3 - \zeta_3)_{ss}\\
    &\quad +\kappa_3(\kappa_2 - \zeta_2)_{sss} +\kappa_3 (\kappa_3)_{ss}(\kappa_1 - \zeta_1) -\kappa_3\mc{R}_{\rm b}  + \frac{\eta}{\alpha}\kappa_2(\kappa_3-\zeta_3)_s\\
    \mc{R}_2&= 
    \big(\mc{R}_{\rm b}\big)_s - (\kappa_3)_{ss}(\kappa_1 - \zeta_1)_s -\eta\big(\kappa_2\kappa_3-(\kappa_1)_s\big)(\kappa_3 - \zeta_3)_{ss}  \\
    &\quad -\kappa_3(\kappa_1 - \zeta_1)_{sss}+ \kappa_3(\kappa_3)_{ss}(\kappa_2 - \zeta_2)  +\kappa_3\mc{R}_{\rm a} - \frac{\eta}{\alpha}\kappa_1(\kappa_3-\zeta_3)_s \\
    \mc{R}_3 &= -\kappa_2\mc{R}_{\rm a}+ \kappa_1\mc{R}_{\rm b} \,.
\end{aligned}
\end{equation}

\subsubsection*{Acknowledgments} The authors thank Sigurd Angenent and Saverio Spagnolie for helpful discussions. The authors acknowledge support from the NSF under grants DMS-2406947 (awarded to DA) and DMS-2406003 (awarded to LO) and from the Office of the Vice Chancellor for Research and Graduate Education at the University of Wisconsin-Madison with funding from the Wisconsin Alumni Research Foundation. 

%%%%%%%%%%%%%%%%%%%%%%%%%%%%%%%%%%%%%%%%%%%%%%%%%%%%%%%

\bibliographystyle{abbrv}
\bibliography{RFT_bib}

%%%%%%%%%%%%%%%%%%%%%%%%%%%%%%%%%
%%%%%%%%%%%%%%%%%%%%%%%%%%%%%%%%%
%%%%%%%%%%%%%%%%%%%%%%%%%%%%%%%%%
\end{document}